\documentclass[12pt]{amsart}
\usepackage{amssymb}

\usepackage{tikz}
\usetikzlibrary{decorations.pathreplacing}
\usetikzlibrary{matrix}
\usetikzlibrary{calc}
 \usepackage{tikz-cd}
\tikzset{
  symbol/.style={
    draw=none,
    every to/.append style={
      edge node={node [sloped, allow upside down, auto=false]{$#1$}}}
  }
}
\usepackage[normalem]{ulem}  

\makeatletter
\newcommand*{\centerfloat}{%
  \parindent \z@
  \leftskip \z@ \@plus 1fil \@minus \textwidth
  \rightskip\leftskip
  \parfillskip \z@skip}
\makeatother

\usepackage[pdfencoding=auto,psdextra]{hyperref}
\usepackage{tabmac}
\usepackage{ytableau}
\usepackage[hmargin=1in,vmargin=1.25in]{geometry}
\definecolor{amethyst}{rgb}{0.6, 0.4, 0.8}
\definecolor{kellygreen}{rgb}{0.3, 0.73, 0.09}
\definecolor{americanrose}{rgb}{1.0, 0.01, 0.24}
\definecolor{teal}{rgb}{0, 0.5, 0.5}
\definecolor{lime}{rgb}{0.75, 1, 0}
\definecolor{darklime}{rgb}{0.25, .33, 0}

\makeatletter
\newcommand{\addresseshere}{%
  \enddoc@text\let\enddoc@text\relax
}
\makeatother

    \title{The nonsymmetric shuffle theorem}

   \author{J. Blasiak}
   \author{M. Haiman}
   \author{J. Morse}
   \author{A. Pun}
   \author{G. H. Seelinger}

    \address[Blasiak]{
    Dept.\ of Mathematics\\
    Drexel University\\
    Philadelphia, PA}
    \email{jblasiak@gmail.com}

   \address[Haiman]{Dept.\ of Mathematics\\
            University of California\\
            Berkeley, CA}
   \email{mhaiman@math.berkeley.edu}

   \address[Morse]{
   Dept.\ of Mathematics\\
            University of Virginia\\
            Charlottesville, VA}
   \email{morsej@virginia.edu}

    \address[Pun]{Dept.\ of Mathematics, CUNY-Baruch College, New York,
    NY}
    \email{anna.pun@baruch.cuny.edu}

   \address[Seelinger]{Dept.\ of Mathematics\\
   University of Michigan\\
   Ann Arbor, MI}
   \email{ghseeli@umich.edu}

   \thanks{Authors were supported by NSF Grants DMS-2154282 and 2452208
(J.B.), DMS-2154281 and 2452209 (J.M.), DMS-2303175 (G.S.) and
DMS-1929284 (all authors, in residence at MathematicsCollaborate@ICERM)}

\newtheorem{thm}{Theorem}[subsection]
\newtheorem{lemma}[thm]{Lemma}
\newtheorem{prop}[thm]{Proposition}
\newtheorem{cor}[thm]{Corollary}
\newtheorem{conj}[thm]{Conjecture}

\theoremstyle{definition}
\newtheorem{defn}[thm]{Definition}
\newtheorem{propdef}[thm]{Proposition-Definition}

\theoremstyle{remark}
\newtheorem{example}[thm]{Example}
\newtheorem{problem}[thm]{Problem}
\newtheorem{remark}[thm]{Remark}

\newcommand{\Dcal}{{\mathcal D}}

\newcommand{\PP}{{\mathsf{P}}}   
\newcommand{\Q}{{\mathsf{Q}}}   
\DeclareMathOperator{\maxtdinv}{maxtdinv}
\DeclareMathOperator{\area}{area}
\DeclareMathOperator{\Area}{Area}
\DeclareMathOperator{\content}{content}

\DeclareMathOperator{\FWPF}{FWPF}
\newcommand{\umnab}{\bar{\nabla}}   
\newcommand{\modnab}{\boldsymbol{\nabla}}
\DeclareFontFamily{U}{mathx}{\hyphenchar\font45}
\DeclareFontShape{U}{mathx}{m}{n}{
      <5> <6> <7> <8> <9> <10>
      <10.95> <12> <14.4> <17.28> <20.74> <24.88>
      mathx10
      }{}
\DeclareSymbolFont{mathx}{U}{mathx}{m}{n}
\DeclareFontSubstitution{U}{mathx}{m}{n}
\DeclareMathAccent{\widecheck}{0}{mathx}{"71}
\newcommand{\Mnab}{\widecheck{\nabla}}
\renewcommand{\H}{\mathcal{ H}}
\newcommand{\Pas}{\mathcal{P}_{\rm as}}
\DeclareMathOperator{\Par}{Par}
\DeclareMathOperator{\spn}{span}
\renewcommand{\P}{\mathcal{P}}
\newcommand{\barw}{\bar{w}}
\newcommand{\barwb}{\widecheck{w}}
\newcommand{\N}{\mathcal{N}}
\DeclareMathOperator{\st}{st}  
\DeclareMathOperator{\rot}{rot}
\DeclareMathOperator{\pairs}{CompPar}
\newcommand{\pairsr}{\ensuremath{\NN^r \times \Par}}
\newcommand{\tE}{ \check{\mathsf H} }   
\newcommand{\tEnoflip}{ \mathsf H }   
\newcommand{\tEcal}{\check{\mathcal E}}  
\newcommand{\stE}{\check{J}}  
\newcommand{\flagh}{\ensuremath{\mathsf{h}}}
\newcommand{\nsC}{\ensuremath{\mathsf{C}}}
\newcommand{\alg}{\nsC}
\newcommand{\brho}{\ensuremath{\boldsymbol{\rho}}}

\newcommand{\Qdot}{\ensuremath{\dot{\mathbf{Q}}}}
\newcommand{\Qddot}{\ensuremath{\ddot{\mathbf{Q}}}}
\DeclareMathOperator{\tdinv}{tdinv}
\newcommand{\wtdinv}{\rm{tdinv}'}
\newcommand{\flagGcal}{\ensuremath{\mathcal{G}}}
\DeclareMathOperator{\chr}{ch}

\newcommand{\idelm}{\text{\rm id}}

\newcommand{\fh}{\flagh}

\newcommand{\pibold}{{\boldsymbol \pi }}
\newcommand{\Weyl}{\pibold }

\DeclareMathOperator{\dg}{dg}
\newcommand{\cc}[1]{#1}
\DeclareMathOperator{\ord}{ord}
\newcommand{\hsym}{\widehat{\Scal}}
\newcommand{\thbold}{{\boldsymbol \theta }}
\newcommand{\kapbold}{{\boldsymbol \kappa }}

\newcommand{\Y}{Y}   
\newcommand{\YIW}{Y}  
\newcommand{\X}{X}  
\newcommand{\Grow}{\mathsf{G}}   
\newcommand{\Growpm}{\mathsf{G}^\pm}   
\newcommand{\Gcol}{\tilde{\mathsf{G}}}   
\newcommand{\Gcolpm}{\tilde{\mathsf{G}}^\pm}   

\newcommand{\Pisf}{{\mathsf \Pi }}

\newcommand{\NN}{{\mathbb N}}
\newcommand{\QQ}{{\mathbb Q}}

\newcommand{\DD}{{\mathbb D}}
\newcommand{\RR}{{\mathbb R}}

\newcommand{\ZZ}{{\mathbb Z}}
\renewcommand{\AA}{{\mathbb A}}

\newcommand{\kk}{{\mathbf k}}

\newcommand{\nubold}{{\boldsymbol \nu }}
\newcommand{\mubold}{{\boldsymbol \mu }}
\newcommand{\etabold}{{\boldsymbol \eta }}

\newcommand{\Acal}{{\mathcal A}}
\newcommand{\Ecal}{{\mathcal E}}

\newcommand{\Gcal}{{\mathcal G}}

\newcommand{\Pcal}{{\mathcal P}}
\newcommand{\Scal}{{\mathcal S}}

\newcommand{\ctild}{\tilde{c}}
\newcommand{\Htild}{\tilde{H}}

\DeclareMathOperator{\dinv}{dinv}
\DeclareMathOperator{\inv}{inv}

\DeclareMathOperator{\pol}{pol}

\DeclareMathOperator{\GL}{GL}
\DeclareMathOperator{\SL}{SL}

\DeclareMathOperator{\FSSYT}{FSW}
\DeclareMathOperator{\FCSSYT}{\widetilde{FSW}}
\DeclareMathOperator{\FCSSYTpm}{\widetilde{FSW}{}^{\pm}\hspace{-0.1em}}

\newcommand{\defeq}{\overset{\text{{\em def}}}{=}}

\newcommand{\xx}{{\mathbf x}}

\newcommand{\spa}{\hspace{.3mm}}

\newcommand{\bx}[2]{{\boldsymbol {#1}[#2]}}
\renewcommand{\SS}{S}

\newcommand{\crc}[1]{-#1}

\newlength{\mycellsize}
\newcommand\mytbl[1]{
\vcenter{
\let\\=\cr
\baselineskip=-16000pt \lineskiplimit=16000pt \lineskip=0pt
\halign{&\mytblcell{##}\cr#1\crcr}}}

\newcommand{\mytblcell}[1]{{%
\def \arg{#1}\def \void{}%
\ifx \void \arg
\vbox to \mycellsize{\vfil \hrule width \mycellsize height 0pt}%
\else \unitlength=\mycellsize
\begin{picture}(1,1)
\put(0,0){\makebox(1,1){$#1\vphantom{\crc{#1}}$}}
\put(0,0){\line(1,0){1}}
\put(0,1){\line(1,0){1}}
\put(0,0){\line(0,1){1}}
\put(1,0){\line(0,1){1}}
\end{picture}%
\fi}}

\newlength{\cellsize}
\newcommand\mytableau[1]{
\vcenter{
\let\\=\cr
\baselineskip=-16000pt \lineskiplimit=16000pt \lineskip=0pt
\halign{&\mytableaucell{##}\cr#1\crcr}}}

\newcommand{\mytableaucell}[1]{{%
\def \arg{#1}\def \void{}%
\ifx \void \arg
\vbox to \cellsize{\vfil \hrule width \cellsize height 0pt}%
\else \unitlength=\cellsize
\begin{picture}(1,1)
\put(0,0){\makebox(1,1){$#1\vphantom{\crc{#1}}$}}
\put(0,0){\line(1,0){1}}
\put(0,1){\line(1,0){1}}
\put(0,0){\line(0,1){1}}
\put(1,0){\line(0,1){1}}
\end{picture}%
\fi}}

\DeclareMathSymbol{\shortminus}{\mathbin}{AMSa}{"39}

\newcounter{rownum}

\newcommand{\nth}[2]{
    \foreach \xx[count=\kk] in #1{
      \ifnum\kk=#2
        \xx
      \fi
    }
}

\begin{document}


\begin{abstract}

The shuffle conjecture of Haglund et al.\  expresses the symmetric function $\nabla e_n$ as a sum over labeled Dyck paths.
Here $\nabla$ is an operator on symmetric functions defined in terms of its diagonal action on the basis of modified Macdonald polynomials.
The shuffle conjecture was later refined by Haglund-Morse-Zabrocki to the compositional shuffle conjecture, expressing $\nabla  C_\alpha$ as a sum over labeled Dyck paths with touchpoints specified by $\alpha$,
where $C_\alpha$ is
a compositional Hall-Littlewood polynomial.
Carlsson-Mellit settled both versions  by developing the theory of a variant of the DAHA called the double Dyck path algebra.

In a recent paper, we discovered a notion of nonsymmetric plethsym
which led us to a construction of modified nonsymmetric Macdonald polynomials
$\mathsf{H}_{\eta|\lambda}(\mathbf{x};q,t)$.
These polynomials
Weyl symmetrize to their symmetric counterparts and are conjecturally atom positive.
Here we introduce a nonsymmetric version  $\boldsymbol{\nabla}$ of $\nabla$, now acting diagonally on the basis given by the functions
$\mathsf{H}_{\eta|\lambda}(\mathbf{x};q,t)$.
Weaving together our theory with results of Carlsson-Mellit and Mellit,
we establish a nonsymmetric version of the compositional shuffle theorem,
which equates $\boldsymbol{\nabla}^{-1}$ applied to a nonsymmetric version $\mathsf{C}_\alpha$ of $C_\alpha$ with a sum over flagged
labeled Dyck paths with touchpoints given by $\alpha$.
This combinatorial sum is conjecturally atom positive, refining the known Schur positivity of its symmetric counterpart.

\end{abstract}

\maketitle

\section{Introduction}
\label{s intro}

The introduction of Macdonald's symmetric functions in the 1980s sparked investigations in
many directions and ultimately gave rise to several robust theories.  Our focus is on the
unification of two of the major theories.
The first of these emerged from efforts to prove Macdonald's positivity conjecture, which asserted
the Schur positivity of the plethystically modified
Macdonald polynomials $\Htild_\mu(\xx;q,t) $.
Garsia and Haiman~\cite{GarsiaHaiman} approached this problem by constructing modules $\mathcal M_\mu$
whose Frobenius series are the $\Htild_\mu$. Their work launched a broad program centered around
the space of diagonal harmonics ${\rm DH}_n$,
an $\SS_n$-module containing all the $\mathcal M_\mu$.

Studies of the Frobenius series $\mathcal{F}({\rm DH}_n)$
drew
on both algebraic and combinatorial
methods.  Haiman~\cite{Haiman02} proved
that $\mathcal F({\rm DH}_n)=\nabla e_n$, where
$\nabla$ is an operator defined by a diagonal action on
$\Htild_\mu$. It was later discovered that $\nabla$ yields an automorphism of the elliptic Hall algebra
of Burban and Schiffmann \cite{BurbSchi12}.
Meanwhile,
Haglund et al.\ \cite{HaHaLoReUl05}
formulated the {\em shuffle conjecture}, an identity giving a positive expansion of
$\nabla e_n$
in terms of
the combinatorial
LLT basis.
The
shuffle conjecture
took over a decade to prove, during which time a wealth of new conjectures and ideas emerged, significantly broadening the scope of the field. Ultimately, Carlsson and Mellit~\cite{CarMel} established
the {\it compositional shuffle conjecture},
a more general formula for $\nabla$ applied to any compositional Hall–Littlewood polynomial
conjectured by Haglund-Morse-Zabrocki~\cite{HagMoZa12}.

In parallel with the developments
around modified Macdonald polynomials and diagonal harmonics,
a separate study of nonsymmetric Macdonald polynomials $E_\alpha(x_1,\ldots,x_n;q,t)$ $(\alpha \in \ZZ^n)$ was initiated by Opdam~\cite{Opdam}, Macdonald~\cite{Macdonald03}, and Cherednik~\cite{Cheredniknsmac}.
This nonsymmetric framework drew on techniques from representation theory and topology, and quickly developed into a rich subject in its own right. Notably, it also shed new light on the symmetric theory.
For example, the polynomials $E_\alpha$ satisfy
a recurrence due to Knop and Sahi~\cite{Knop97, Sahiinterpolation},
which has no symmetric analog.
Knop \cite{Knop97} also constructed integral forms $\Ecal_\alpha $ of the nonsymmetric Macdonald polynomials $E_
\alpha$ from which the
symmetric integral forms
can be recovered by symmetrization using Hecke algebra operators.

It was long unclear whether
the theory of plethystically modified Macdonald polynomials \(\Htild_\mu\) could be enriched by lifting it to the nonsymmetric setting.
In~\cite{BHMPS-nspleth}, we constructed such a lift by defining a
  nonsymmetric plethysm operator $\Pisf_r$.
Applying $\Pisf_r$ to a stable version of
 the $\Ecal_\alpha$'s, we constructed
a family of modified $r$-nonsymmetric Macdonald polynomials,
$\tEnoflip_{\eta|\lambda}(\xx;q,t)$,
which satisfy properties bearing a striking resemblance to
those
of the symmetric polynomials $\Htild_\mu$.
In particular,
the $\tEnoflip_{\eta|\lambda}$
are monomial positive and are conjecturally positive in the Demazure atom basis\footnote{More precisely, we conjecture that the $\tEnoflip_{\eta|\lambda}(x_1,\dots, x_N, 0, \dots;q,t)$ are atom positive. See \S\ref{ss atom pos 2}. The same comment applies to other atom positivity conjectures mentioned before \S\ref{ss atom pos}.}.
Moreover, $\tEnoflip_{\eta|\lambda}$ admits a combinatorial expansion as a positive sum of {\it flagged LLT polynomials},
a nonsymmetric variant of LLT polynomials.

Here we develop
 a nonsymmetric diagonal harmonics program
 by reinterpreting key ideas from Carlsson-Mellit~\cite{CarMel}
in the new framework of modified polynomials $\tEnoflip_{\eta|\lambda}$ from~\cite{BHMPS-nspleth}.
Carlsson and Mellit introduce an algebra $\AA_{t,q}$, called the double Dyck path algebra.
They
deduce the compositional shuffle theorem by applying an operator
$d_- \in\AA_{t,q}$ to
 an identity which is in effect a nonsymmetric lift of this theorem.
A key ingredient in
their setup
is a family of
 nonsymmetric polynomials
that are used as intermediate steps in building LLT polynomials. One of our main insights
here
is that
the images
 of these intermediate objects
under the nonsymmetric plethysm $\Pisf_r$
 are precisely flagged LLT polynomials.
This perspective leads us to a nonsymmetric $\nabla$ operator that acts diagonally on the modified polynomials $\tEnoflip_{\eta|\lambda}(\xx;q^{-1}, t)$.
Pulling these ideas together, we
obtain a nonsymmetric compositional shuffle theorem fully parallel to the one in the symmetric case, expressed in terms of combinatorial objects
with conjectural atom positivity,
refining the known Schur positivity of their symmetric counterparts.

\subsection{Main theorems}
We now give an account of our main results, deferring a few definitions to later in the paper.

The compositional shuffle theorem,
conjectured by Haglund-Morse-Zabrocki \cite{HagMoZa12} and proved by Carlsson-Mellit \cite{CarMel},
states that for any strict composition $\alpha= (\alpha_1,\dots, \alpha_r)$ of size  $k$ (i.e., $\alpha \in \ZZ_+^r$ and  $\sum_{i=1}^r \alpha_i = k$),
\begin{align}
\label{e intro comp shuffle}
\nabla C_{\alpha}(\xx ; t)
= \sum_{\pi \in \DD^\alpha_{k,k} }
\sum_{\substack{\text{word parking} \\ \text{functions $w$ on  $\pi$}}}
 q^{\area(\pi)} t^{\dinv(\pi, w)}
 \xx^{\content(w)}.
\end{align}
On the algebraic
left hand
side, $C_{\alpha}(\xx ;t)$ is the compositional Hall-Littlewood polynomial from \cite{HagMoZa12}.
On the combinatorial
right hand side,
$\DD^\alpha_{k,k}$ denotes the set of Dyck paths from  $(0,0)$ to  $(k,k)$
which touch the diagonal exactly at the points
$(k_{j},k_{j})$, where $k_{j} = \sum _{i < j}\alpha _{i}$ for $j=1,\ldots, r+1$;
a \emph{word parking function} on $\pi$ is a map from the north steps of
$\pi$ to  $\ZZ_+$ which is strictly increasing going north along vertical runs;
$\area(\pi )$ is the number of full lattice squares between $\pi $ and the diagonal, and $\dinv$ is another integer statistic on Dyck paths
(see \S \ref{ss Dyck path stats} and \eqref{e dinv HMZ def}).

The right hand side of \eqref{e intro comp shuffle} can be packaged as a sum of LLT polynomials for columns, and these are known to be Schur positive. As such,
\begin{align}
\label{e pos nab intro}
\nabla C_{\alpha} \text{ is Schur positive.}
\end{align}

Our nonsymmetric compositional shuffle theorem takes place in the space of \emph{almost symmetric polynomials}
$\Pas = \bigcup_{r \ge 0} \P(r)$, where
$\P(r) = \QQ(q,t)[x_1,\dots,x_r] \otimes \Lambda_{\QQ(q,t)}(x_{r+1}, \dots)$.
We work with a
variant of the modified $r$-nonsymmetric Macdonald polynomials introduced in \cite{BHMPS-nspleth}, namely,
$\tE_{\eta|\lambda}(\xx;q,t) =
q^{\mathsf{n}(\mu^*)} \tEnoflip_{\eta| \lambda}(\xx;q^{-1},t) \in \Pas$,
where $\mu =(\eta; \lambda)_+$ is the partition rearrangement of the concatenation of  $\eta$ and  $\lambda$.
The  $\tE_{\eta|\lambda}$ are indexed by pairs  $(\eta|\lambda)$, where $\eta \in \NN^r$ with
$\eta_r >0$, and  $\lambda$ a partition.
We show that  $\tE_{\eta|\lambda}$ Weyl symmetrizes to
$\omega\Htild_{(\eta;\lambda)_+}(\xx;q,t)$.

Here, we define a nonsymmetric version $\modnab$ of  $\omega\nabla\omega$ whose eigenvectors are the
$\tE_{\eta| \lambda}(\xx;q,t)$. Specifically,  $\modnab \colon \Pas \to \Pas$ is the  $\QQ(q,t)$-linear map determined by
\begin{align}
\label{e mod nabla intro}
\modnab \tE_{\eta| \lambda}(\xx ;q,t) =
q^{\mathsf{n}(\mu^*)}t^{\mathsf{n}(\mu)}\tE_{\eta| \lambda}(\xx ;q,t),  \quad \text{for  $\mu = (\eta; \lambda)_+$}.
\end{align}

In our nonsymmetric compositional shuffle theorem, the $C_\alpha$ are replaced by \emph{nonsymmetric compositional Hall-Littlewood polynomials}  $\nsC_\alpha$, which Weyl symmetrize to  $C_\alpha(\xx;t^{-1})$.  They are given by
\begin{align}\label{e:nsCalpha}
\nsC_\alpha =  (-t)^{|\alpha|-r}\pol
\Big(\frac{x_1^{\alpha_1} \cdots x_r^{\alpha_r}}{\prod_{1 \le i < j \le r}(1-t x_i/x_j)}  \Big),
\end{align}
where   $\pol$ is the linear
projection from Laurent polynomials onto ordinary polynomials
determined by its action on Demazure characters, as follows:
$\pol(\Dcal_\alpha) = \Dcal_\alpha$ for $\alpha \in \NN^r$ and $\pol(\Dcal_\alpha)=0$ for $\alpha \in \ZZ^r \setminus \NN^r$.
On symmetric functions,  $\pol$ is the polynomial truncation operator which truncates $\GL_r$-characters (indexed by  $\lambda_1 \ge \cdots \ge \lambda_r$) to polynomial ones (indexed by  $\lambda_1 \ge \cdots \ge \lambda_r \ge 0$).

\begin{thm}
\label{t mod nsshuffle intro}
For every strict composition  $\alpha = (\alpha_1, \dots, \alpha_r)$ of size $k$,
\begin{align}
\label{et mod nsshuffle intro}
\modnab^{-1} \nsC_\alpha
 =
\sum_{\pi \in \DD^\alpha_{k,k} } \sum_{\substack{\text{flagged word parking}\\
\text{functions $w$ on  $\pi$ } }} q^{-\area(\pi)}
t^{-\dinv(\pi,w)} \xx^{\content(w)},
\end{align}
where the sum is now over word parking functions  $w$ satisfying the following flagging condition:
for each  $j = 1,\dots, r$, the entry of  $w$ on the north step from the
touchpoint $(k_{j},k_{j})$ is at most $j$, where $k_{j} =\sum_{i < j}\alpha_i$.
\end{thm}

See \S\ref{ss examples} for examples illustrating this theorem.
We prove that \eqref{e intro comp shuffle} can be recovered by Weyl symmetrizing \eqref{et mod nsshuffle intro}
(see \S\ref{ss nablashuffle} for why  $\modnab, q,t$ are inverted in \eqref{et mod nsshuffle intro} but not \eqref{e intro comp shuffle}).

The right hand side of \eqref{et mod nsshuffle intro} can be packaged as a sum of the flagged LLT polynomials introduced in \cite{BHMPS-nspleth}.
We showed there that these Weyl symmetrize to ordinary LLT polynomials, and conjectured that they are atom positive.
This would imply that $\modnab^{-1} \nsC_\alpha$ is atom positive, thereby giving a nonsymmetric strengthening of \eqref{e pos nab intro}.
This also raises the question of whether  $\modnab^{-1} \nsC_\alpha$
has a natural representation-theoretic interpretation,
 which we discuss further in \S\ref{ss atom pos}.

In \cite{BeGaSeXi16}, Bergeron et al.\ proposed the  $(km,kn)$-compositional shuffle conjecture,
where the combinatorial side is now replaced with a sum over Dyck paths $\DD^\alpha_{km,kn}$ from $(0,0)$ to  $(km,kn)$ with touchpoints given by $\alpha$.
This was proved by Mellit \cite{Mellit16}, building on the methods in \cite{CarMel}.  Our methods
yield a nonsymmetric counterpart for this version as well.  We now describe this result and
along the way outline its proof and that of Theorem \ref{t mod nsshuffle intro}.

Our approach builds on a framework around signed and unsigned flagged variants of LLT polynomials introduced in
\cite{BHMPS-nspleth}.  A key discovery there was the \emph{$r$-nonsymmetric plethysm map} $\Pisf_r$, a nonsymmetric variant of the plethystic transformation $g[\xx] \mapsto g[\xx/(1-t)]$.
It is the  $\QQ(q,t)$-linear map  $\Pisf_r \colon \P(r) \to \P(r)$
determined by its action on elements
$f(x_1,\dots, x_r) g(\xx)$ by
\begin{align}
\label{e: nspleth def}
\Pisf_r(f(x_1, \dots, x_r) g(\xx)) =
g\Big[\frac{\xx}{(1-t)}\Big]
\pol
\Big( \frac{f(x_1, \dots, x_r)}{\prod_{1 \le i < j \le r} (1-t \spa x_i/x_j)} \Big),
\end{align}
where $g$ is symmetric in  $\xx = x_1,x_2,\dots$.
Remarkably, $\Pisf_r$ maps signed flagged LLT polynomials to unsigned ones (see Theorem \ref{t signed to unsigned}).

Here we capitalize on this
result by showing that it meshes beautifully with objects
introduced in Carlsson and Mellit's proof of the shuffle theorem \cite{CarMel}. More specifically,
we show that certain nonsymmetric polynomials used in their proof to
construct symmetric LLT polynomials inductively are,
after an (ordinary) plethystic transformation $\PP_r^{-1}$, equal to our signed flagged LLT polynomials
for a sequence of
single-row skew shapes (see Theorem~\ref{t Gcal vs chiPP}).
This enables us to interpret an intermediate step in Mellit's proof of the $(km,kn)$-compositional shuffle conjecture as
saying that, for a certain
operator $\brho _{m,n}^{\alpha }$ in the double Dyck path algebra
$\AA_{t,q}$, the expression $\PP_r^{-1} \brho _{m,n}^{\alpha } \cdot 1$ is equal to a sum of signed flagged LLT
polynomials.
Applying  $\Pisf_r$ then yields our nonsymmetric $(km,kn)$-shuffle theorem, which expresses  $\Pisf_r \,\PP_r^{-1} \brho_{m,n}^\alpha \cdot 1$
as a sum of (unsigned) flagged LLT polynomials $\Grow_r(\pi', \Sigma_\pi)$:
\begin{align}
\label{e ns shuffle kmkn mod intro}
\Pisf_r \,\PP_r^{-1} \brho_{m,n}^\alpha \cdot 1 =
\sum_{\pi \in \DD^\alpha_{km,kn} }  q^{\area(\pi)} t^{\dinv(\pi) - \maxtdinv(\pi) } \,
\Grow_r(\pi', \Sigma_\pi)\,,
\end{align}
with the complete statement and notation explained in Theorem \ref{t ns shuffle kmkn mod}.

The advantage of this framing is that we now obtain a statement that Weyl symmetrizes to the original symmetric  $(km,kn)$-shuffle theorem, and
where the flagged LLT polynomials $\Grow_r(\pi', \Sigma_\pi)$ are conjecturally atom positive.
Hence, this gives a conjectural strengthening of
the Schur positivity of the symmetric  $(km,kn)$-shuffle theorem.

Theorem \ref{t mod nsshuffle intro} can then be deduced from \eqref{e ns shuffle kmkn mod intro}
by relating $\modnab$ to the endomorphism  $N$ of  $\AA_{t,q}$ introduced in \cite{Mellit16};
this, in turn, is proved using results of Bechtloff Weising \cite{BechtloffWeising23} and Ion-Wu \cite{IonWu}
as well as
a Monk's rule for nonsymmetric Macdonald polynomials due to
Halverson-Ram \cite{HalversonRam}.

\subsection{Atom positivity and representation theory}
\label{ss atom pos}

A guiding principle in our work is that atom positivity in a polynomial signals the presence of underlying representation-theoretic structure. Here, we sketch ideas for a representation-theoretic interpretation of $\modnab^{-1} \nsC_\alpha$, following some preliminaries on Demazure characters and atoms.

The \emph{Demazure operator} or \emph{isobaric divided difference operator $\pi_i$} is given by
\vspace{-1mm}
\begin{align}
\pi_i (f) &= (x_i f - s_i(x_i f))/(x_i - x_{i+1}).  \\[-6.5mm]
\notag
\end{align}
For  $w\in \SS_N$, define the operator $\pi_w := \pi_{s_{j_1}} \cdots \pi_{s_{j_d}}$, where
$w = s_{j_1}s_{j_2}\cdots s_{j_d}$ is any reduced expression.  This does not depend on the choice of reduced expression since the $\pi_i$ satisfy the braid relations.
For  $\alpha\in \ZZ^N$, the \emph{Demazure character} $\Dcal_\alpha$ is
\begin{align}
\Dcal_\alpha = \pi_{w} (\mathbf{x}^{\alpha_+}),
\end{align}
where  $w$ is the shortest permutation such that $w(\alpha_+) = \alpha$.
The \emph{Demazure atom}  $\Acal_\alpha$ is defined similarly but with $\hat{\pi}_w$ in
place of  $\pi_w$, defined from  $\hat{\pi}_i  = \pi_i - 1$.

Demazure characters $\{\Dcal_\alpha\}_{\alpha \in \ZZ^N}$ and atoms  $\{\Acal_\alpha\}_{\alpha \in \ZZ^N}$
are each a basis of $\QQ(q,t)[x_1^{\pm 1}, \dots, x_N^{\pm 1}]$.
Demazure characters are positive sums of Demazure atoms:
\begin{align}
\label{e:dem to atom}
\Dcal_{\beta} =  \sum_{\substack{\alpha  \, \in \, \SS_N \beta, \\[.8mm]  w_\alpha \le w_\beta }} \Acal_\alpha,
\end{align}
where  $\SS_N \beta$ denotes the  $\SS_N$ orbit of  $\beta$, $w_\alpha$ is the shortest permutation such that  $w_\alpha(\alpha_+)= \alpha$,
 and  $\le$ is Bruhat order on  $\SS_N$.
For the longest permutation  $w_0$ of  $\SS_N$,  $\pi_{w_0}$ is the \emph{Weyl symmetrization operator}.
It takes the monomial $\xx^\lambda$ indexed by a weakly decreasing weight $\lambda \in \NN^N$ to the
Schur function $s_{\lambda }(x_1, \dots, x_N)$.

Consider the Lie algebra $\mathfrak{gl}_N$ of  $N \times N$ matrices and its Lie subalgebra  $\mathfrak{b} \subset \mathfrak{gl}_N$ of upper triangular matrices.
For $\nu = (\nu_1 \ge  \cdots \ge \nu_N) \in \ZZ^N$,
let $V(\nu)$ be the irreducible  $\mathfrak{gl}_N$-module of highest weight $\nu$.
The \emph{Demazure module} $D(\alpha) \subset V(\alpha_+)$
is the  $\mathfrak{b}$-module  $U(\mathfrak{b}) \spa u_\alpha$, where  $u_\alpha$ is an element of the (one-dimensional) $\alpha$-weight space of $V(\alpha_+)$.
The \emph{Demazure atom module} $A(\alpha)$
is the quotient of $D(\alpha)$ by the sum of all Demazure modules properly contained in  $D(\alpha)$.
The characters $\chr(D(\alpha))$ and $\chr(A(\alpha))$ are $\Dcal_\alpha$ and $\Acal_\alpha$, respectively.

Just as a symmetric polynomial being Schur positive suggests it is the character of a natural  $\mathfrak{gl}_N$-module,
a nonsymmetric polynomial being Demazure character (resp.\ Demazure atom) positive suggests it is the character of a  $\mathfrak{b}$-module with a filtration
whose subquotients are isomorphic to
Demazure modules (resp. Demazure atom modules).
Such filtrations have been
extensively
 studied from the geometric perspective---see, e.g.,
\cite{Joseph85, Mathieu, Polo, vanderKallenBModule}.

For any polynomial $f$, we have
\begin{align}
\label{e atom key pos}
\text{$f$ is Demazure character positive} \implies
\text{$f$ is atom positive} \implies
\text{$\pi_{w_0} (f)$ is Schur positive}.
\end{align}
Thus, Demazure character and atom positivity are natural strengthenings of Schur positivity to the nonsymmetric setting.

For  $f\in \P(r)$, write  $f[\xx_N]$ as shorthand for $f(x_1, \dots, x_{N}, 0, \dots)$.
Any $f \in  \P(r)$ of degree  $d$ is determined by
$f[\xx_{r+d}]$, so it is mostly harmless to work with $f(\xx)$ and $f[\xx_{r+d}]$ interchangeably.
See also Remark \ref{r stable atom vs atom}.

The conjectured atom positivity of  $(\modnab^{-1} \nsC_\alpha)[\xx_N]$ then suggests the following problem:
\begin{problem}
\label{pr rep theory}
Find a natural representation-theoretic interpretation of
$\modnab^{-1} \nsC_\alpha$.  More precisely, find a (natural) $\mathfrak{b}$-module whose character is  $(\modnab^{-1} \nsC_\alpha)[\xx_N]$, with a
filtration whose subquotients are Demazure atom modules.
\end{problem}

A possible starting point for this problem is a construction of Feigin-Loktev \cite{FeiginLoktev}
of a $\mathfrak{gl}_N$-module ${\rm W}_{n}^N$ which is Schur-Weyl dual to the diagonal harmonics module ${\rm DH}_n$.  Using the $\mathcal{F}({\rm DH}_n) = \nabla e_n$ theorem, they showed that ${\rm W}_{n}^N$ has character $(\nabla e_n)[\xx_N]$.
One might then expect to find the $\mathfrak{b}$-module in Problem \ref{pr rep theory} as a subquotient of  ${\rm Res}_{\spa \mathfrak{b}} {\rm W}_n^N$.
Also note that the symmetric version of this problem---find a representation-theoretic interpretation for  $\nabla C_\alpha$---is still open;
some partial progress has been made by Bergeron-Descouens-Zabrocki \cite{BDZ}.

\subsection{Organization}
To prove Theorem~\ref{t mod nsshuffle intro} and its
  $(km,kn)$-generalization, Theorem~\ref{t ns shuffle kmkn mod}, we need to
provide the language necessary to state both the algebraic and
combinatorial sides, along with their connections to previously known
objects.
To establish this groundwork,
we summarize definitions from~\cite{CarMel, IonWu, Mellit16} in \S\ref{s ddpa}
and establish properties of the \(r\)-nonsymmetric plethysm map (\S\ref{s nspleth}), flagged LLT
polynomials (\S\ref{s LLT}), and modified \(r\)-nonsymmetric Macdonald polynomials (\S\ref{s mod nsmac}).
In Section~\ref{s ns comp kmkn thm}, we prove
Theorem~\ref{t ns shuffle kmkn mod} using the
machinery from~\cite{Mellit16} and results from \S\ref{s LLT}.
In Section~\ref{s ns nab ns shuffle}, we establish results about \(\modnab\) that we use
to deduce Theorem~\ref{t mod nsshuffle intro} from the $m=n=1$ case of Theorem~\ref{t ns shuffle kmkn mod}.
Additionally, we provide examples of Theorem \ref{t mod nsshuffle intro} in \S\ref{ss examples} and a
summary of notation in
Figure~\ref{fig:summary diagram}.

\section{Preliminaries}
\label{S:Prelim}

We fix some notation used throughout the paper.

We use the shorthand $[n] = \{1,2,\dots, n\}$.

Integer partitions are written $\lambda = (\lambda _{1}\geq \cdots
\geq \lambda _{l})$, possibly with trailing zeroes.
We set $|\lambda| = \lambda _{1}+\cdots +\lambda _{l}$, let $\ell(\lambda )$ be the
number of nonzero parts, and define the integer statistic \(\mathsf{n}(\lambda) = \sum_i (i-1) \lambda_i\).
The conjugate partition of $\lambda$ is denoted $\lambda^*$.
Let  $\Par$ denote the set of partitions.

We use the notation $\xx_N = x_1, \dots, x_N$, $\xx= x_1, x_2, \dots$,
and for $f(\xx) \in \P(r)$, we write
$f[\xx_N]$ for $f(\xx)|_{x_{N+1} = x_{N+2} = \cdots = 0}$.
For $\alpha \in \ZZ^N$, set $\xx^\alpha = x_1^{\alpha_1} \cdots x_N^{\alpha_N}$.

\subsection{Symmetric functions}
Let $\Lambda(\xx) = \Lambda _{\QQ(q,t)}(\xx)$ be the algebra of symmetric
functions in the infinite alphabet $\xx$, with coefficients in the field $\QQ (q,t)$.
We follow the notation of Macdonald~\cite{Macdonald95} for the graded
bases of $\Lambda(\xx)$.  We write $\omega \colon \Lambda(\xx) \rightarrow
\Lambda(\xx)$ for the involutory $\QQ(q,t)$-algebra automorphism determined by
$ \omega s_{\lambda } = s_{\lambda^*}$ for Schur functions $s_{\lambda}$.

Given $f\in \Lambda(\xx)$
and any expression $A$
involving indeterminates, the plethystic evaluation $f[A]$ is defined
by writing $f$ as a polynomial
in the power-sums $p_{k}$ and evaluating with $p_{k}\mapsto p_{k}[A]$,
where $p_{k}[A]$ is the result of substituting $a^{k}$ for every
indeterminate $a$ occurring in $A$.  The variables $q, t$ from our
ground field count as indeterminates.
When  $\xx_r$ and  $\xx$ are used in a plethystic formula, we interpret these as $\xx_r = x_1 + \cdots + x_r$ and  $\xx = x_1 + x_2 + \cdots$.
For
example, the evaluation $f[\xx/(1-t)]$ is the image of $f(\xx)$ under
the $\QQ(q,t)$-algebra automorphism of $\Lambda(\xx)$ that sends $p_{k}$ to
$p_{k}/(1-t^{k})$.

We fix notation for the series
\begin{equation}\label{e:Omega}
\Omega = 1 + \sum _{k>0} h_{k} = \exp \sum _{k>0} \frac{p_{k}}{k},
\quad \text{or}\quad \Omega [a_{1}+a_{2}+\cdots -b_{1}-b_{2}-\cdots ]
= \frac{\prod_{i} (1-b_{i})}{\prod_{i} (1-a_{i})}.
\end{equation}

\subsection{Strong and \texorpdfstring{$t$}{t}-adic convergence}
\label{ss convergence}

For  $a \in \QQ(q,t)$, write $a= b/c$ with  $b, c \in \QQ[q,t]$. The order of vanishing at  $t=0$
of  $a$, denoted  $\ord (a)$, is the order of vanishing of  $b$ minus the order of vanishing of  $c$.
For  $e \in \ZZ$, let  $(\QQ(q,t)[\xx_N])_e$ denote the subset of  $\QQ(q,t)[\xx_N]$ consisting of linear combinations
of monomials $\xx^\lambda$ with coefficients  $a_\lambda$ satisfying  $\ord (a_\lambda) \ge e$.
By convention, $\ord(0) = \infty $ and  $0 \in (\QQ(q,t)[\xx_N])_e$ for all  $e$.

We say that a sequence of polynomials  $g_{N_0}, g_{N_0+1}, \dots$, with
$g_{N}(\xx_N) \in \QQ(q,t)[\xx_N]$, {\em converges strongly to  $f(\xx) \in \P(r)$} if
\begin{equation}
g_{N}(\xx_{N}) = f[\xx_N]
\end{equation}
for all  $N \ge N_0$.

We say that a sequence of polynomials
$g_{N}(\xx_N) \in \QQ(q,t)[\xx_N]$ {\em converges $t$-adically to  $f(\xx) \in \P(r)$} if for all $e>0$, we have
\begin{equation}
\big( g_{N}(\xx_{N}) - f[\xx_N] \big) \in (\QQ(q,t)[\xx_N])_e
\end{equation}
for sufficiently large  $N$.

\begin{remark}
These are the same notions of convergence as in \cite[\S 2.6]{BHMPS-nspleth}. The definition of $t$-adic convergence here is closely related to
that in \cite{IonWu, BechtloffWeising23}.
In particular, if  $g_N$ converges  $t$-adically to  $f$ in the sense of
\cite{BechtloffWeising23, IonWu}, then it also does in the sense above.
\end{remark}

\section{The double Dyck path algebra}
\label{s ddpa}

Much of this paper involves stable versions of the double affine Hecke algebra and nonsymmetric Macdonald polynomials.

The double affine Hecke algebra (DAHA) or Cherednik algebra $\H_N$ is a $\QQ(q,t)$-algebra
generated by  $\X_1^{\pm 1},\dots, \X_N^{\pm 1}$,  $\Y_1^{\pm 1}, \dots, \Y_N^{\pm 1}$, and  $T_1, \dots,  T_{N-1}$.
It contains two copies of the type  $\tilde{A}_{N-1}$ extended affine Hecke algebra, and
reduces to
the group algebra  $\QQ [(\ZZ^N \times \ZZ^N) \rtimes \SS_N]$
when the parameters  $q$ and  $t$ are set to 1.
There is an action of   $\H_N$ on  $\QQ(q,t)[x_1^{\pm 1},\dots, x_N^{\pm 1}]$ in which $\X_i$ acts by multiplication by  $x_i$, the $T_i$ act by Demazure-Lusztig operators,
and the $\Y_i$ act by Cherednik operators (further discussed in \S\ref{ss IonWu DAHA}).
The nonsymmetric Macdonald polynomials  $E_\alpha$ can then be constructed as the joint eigenfunctions of $\Y_1,\dots, \Y_N \in \H_N$.

The double Dyck path algebra  $\AA_{t,q}$ of Carlsson-Mellit \cite{CarMel} and the $+$ stable limit DAHA  $\H^+$ of Ion-Wu \cite{IonWu} are
two algebras which both aim to capture the idea of a limiting object for the double affine Hecke algebras  $\H_N$ as  $N \to \infty$.  The latter accomplishes this through a certain $t$-adic limit, whereas the former avoids directly defining any kind of limit and instead contains a copy of something that is close to $\H_N$ for each  $N$ together with extra elements to go between these copies.  Because of these extra elements,  $\AA_{t,q}$ should be thought of as being bigger than  $\H^+$ (see Theorem \ref{t IW two actions} for the precise relationship).

In this section we describe the algebras $\AA_{t,q}$ and  $\H^+$ and some of their natural representations. In \S\ref{s mod nsmac} we will introduce stable versions of the  $E_\alpha$'s and their relation to  $\H^+$.

\subsection{The Dyck path algebra \texorpdfstring{$\AA_t$}{A\_t} and double Dyck path algebra \texorpdfstring{$\AA_{t,q}$}{A\_{t,q}}}

In this paper, we use the same notation for the (double) Dyck path algebra and its natural representation as Mellit \cite{Mellit16} except that we swap
the role of  $q$ and  $t$ to match the traditional DAHA notation, and we use  $W$ in place of  $X$ to avoid confusion with our notation $\xx=x_1,x_2,\dots$.

\begin{defn}
Let \(\Qdot\) be the quiver with vertex set \(\NN\), with arrows
labeled \(d_+\) from \(r \to r+1\), arrows labeled \(d_-\) from
\(r+1 \to r\), and with \(r-1\) loops from \(r\) to  $r$, labeled
\(T_1,\ldots,T_{r-1}\).
The \emph{Dyck path algebra} \(\AA_t\) is the path algebra of the quiver \(\Qdot\),
over the field  $\QQ(q,t)$, with the following relations:
  \begin{gather}
   \label{eq:Hecke relations}
  \begin{gathered}
   (T_i-1)(T_i+t) = 0\,,\ 1\leq i \leq r-1\,,\\
  T_i T_{i+1} T_i = T_{i+1} T_i T_{i+1}\,,\ 1
    \leq i \leq r-2\,,  \\
  T_i T_j = T_j T_i\,,\ |i-j| > 1\,,
    \end{gathered}\\[1mm]
   d_-^2 T_{r-1} = d_-^2\,,\ T_i d_- = d_- T_i\,,\ 1 \leq i \leq r-2\,,\\[1mm]
  \label{eq:Forward loop-path commutation}
  T_1 d_+^2 = d_+^2\,,\ d_+ T_i = T_{i+1} d_+\,,\ 1 \leq i \leq r-1\,,\\[1mm]
  \label{eq:Forward-back commutation ending}
  d_-[d_+, d_-] T_{r-1} =
    t\, [d_+, d_-] d_- \,,\ r \geq 2\,,\\[1mm]
  \label{eq:Forward-back commutation beginning}
  T_1[d_+,d_-] d_+ = t  \spa d_+[d_+, d_-]\,,\ r \geq 1\,,
   \end{gather}
   where each line is an identity of paths starting at vertex \(r\) and we impose these relations for all  $r \in \NN$.  We write  $e_r$ for the idempotent at vertex  $r$.

 Note that to simplify notation, many arrows have the same name, so to resolve potential ambiguity,
 we adopt the convention that whenever a product of elements of  $\AA_t$ is specified to start at vertex  $r$, we interpret this as the unique sequence of arrows which concatenate head to tail.
For example, in
\eqref{eq:Forward loop-path commutation}, $T_1 d_+^2$ is the directed path which begins with the arrow from  $r$ to  $r+1$ labeled  $d_+$, followed by the arrow from
$r+1$ to  $r+2$ labeled  $d_+$, and then the loop  $T_1$ from $r+2$ to  $r+2$.
\end{defn}

\begin{defn}\label{def:double Dyck path algebra}
Let \(\Qddot\) be the quiver formed from \(\Qdot\) by adding additional arrows labeled
\(d_+^*\) from \(r \to r+1\).
The \emph{double Dyck path algebra} \(\AA_{t,q}\) is the path
algebra of the quiver \(\Qddot\),
over the field  $\QQ(q,t)$, with relations~\eqref{eq:Hecke relations}--\eqref{eq:Forward-back commutation beginning} and the following additional relations, each regarded as a path starting at
vertex \(r\):
  \begin{gather}
  \label{eq:Forward loop-path commutation star} T_1(d_+^*)^2 =
    (d_+^*)^2\,,\ d_+^* T_i = T_{i+1} d_+^*\,, 1 \leq i \leq r-1\,,\\[.6mm]
  \label{eq:Forward-back commutation ending star} t d_-[d_+^*, d_-]  =
    [d_+^*, d_-] d_- T_{r-1} \,,\ r \geq 2\,,\\[.6mm]
  \label{eq:Forward-back commutation beginning star} t[d_+^*, d_-]
    d_+^* = T_1 d_+^*[d_+^*, d_-] \,, r \geq 1\,,\\[.6mm]
  \label{eq:Forward twisted loop-path commutation} d_+ z_i = z_{i+1}
     d_+\,,\ d_+^* y_i = y_{i+1} d_+^*\,,\ 1 \leq i \leq r \,,
    \\[.6mm]
  \label{eq:Star to unstar} z_1 d_+ = -qt^{r+1} y_1 d_+^* \,,
  \end{gather}
  where the elements \(y_i, z_i\), for \(1 \leq i \leq r\), are the following loops at vertex \(r\):
  \begin{align}
  y_1 & = \frac{1}{t^{r-1}(t-1)} [d_+, d_-] T_{r-1} \cdots T_1\,,\\[.6mm]
  y_{i+1} & = t \, T_i^{-1} y_i T_i^{-1}\,,\ 1 \leq i \leq r-1 \,,\\[.6mm]
  \label{e zi def 1}
  z_1 & = \frac{t^r}{1-t} [d_+^*, d_-] T_{r-1}^{-1} \cdots T_1^{-1}\,,\\[.6mm]
  \label{e zi def 2}
  z_{i+1} & = t^{-1} T_i z_i T_i\,,\ 1 \leq i \leq r-1\,.
  \end{align}
  We write  $e_r$ for the idempotent at vertex  $r$.
\end{defn}

\begin{remark}
\label{rem:why ddpa}
There is a $\QQ(q,t)$-algebra homomorphism $\rho \colon \AA_t \to \AA_{t,q}$ determined uniquely by
\begin{align}
e_r \mapsto e_r,
\ \ T_i \mapsto T_i,
\ \ d_- \mapsto d_-,
\ \ d_+ \mapsto d_+,
\end{align}
and a $\QQ$-algebra homomorphism $\rho^* \colon \AA_t \to \AA_{t,q}$ determined uniquely by
\begin{align}
q \mapsto q^{-1}, \ \ t  \mapsto t^{-1} , \ \
e_r \mapsto e_r,
\ \ T_i \mapsto T_i^{-1},
\ \ d_- \mapsto d_-,
\ \ d_+ \mapsto d_+^*.
\end{align}
Thus, $\AA_{t,q}$ contains the copy $\rho(\AA_t)$ of $\AA_t$ and the twisted copy $\rho^{*}(\AA_t)$ of $\AA_t$.
In~\cite{CarMel}, the authors introduce \(\AA_{t,q}\) as the free
product of these two copies with
additional relations between them.
\end{remark}

\subsection{Three actions of  \texorpdfstring{$\AA_{t,q}$}{A\_{t,q}}}
\label{ss two actions}

Let
\(V_r = \QQ(q,t)[y_1,\ldots,y_r] \otimes \Lambda_{\QQ(q,t)}(W)\)
and \(V_* = (V_r)_{r \geq 0}\).
Mellit \cite{Mellit16} and Carlsson-Mellit \cite{CarMel}
each define an action of  $\AA_{t,q}$ on  $V_*$; we write $d_+, d^*_+, d_-, T_i$ for the generators of  $\AA_{t,q}$ with Mellit's action and $d^{CM}_+, d^{* CM}_+, d^{CM}_-, T^{CM}_i$ for the action defined by Carlsson-Mellit
(they both use the same conventions for the generators and relations of the algebra  $\AA_{t,q}$).
We will also work with
an action obtained from Mellit's version by twisting by a plethystic transformation.
We first describe Mellit's version.

\begin{prop}
\label{p Mellit action}
The following formulas define a quiver representation of
\(\AA_{t,q}\) on \(V_*\).
The action of the arrows originating at quiver vertex  $r$ are given by
\begin{align}
\label{e Ti def 1}
T_i F &=  s_i F + (1-t) y_i \frac{F - s_i F}{y_i - y_{i+1}}, \\
\label{e d- def}
d_- F &=  \langle y_r^0 \rangle F[y_1,\ldots,y_r; W-(t-1)y_r]  \Omega[-y_r^{-1} W],\\
\label{e d+ def}
d_+ F &=  -T_1 \cdots T_r(y_{r+1} F[y_1,\ldots,y_r; W+(t-1)y_{r+1}]), \\
\label{e d+* def}
d_+^*F &=  \gamma \spa (F[y_1,\ldots,y_r; W+(t-1)y_{r+1}]).
\end{align}
where $F[y_1,\ldots,y_r; W] \in V_r$,  $i \in [r-1]$, and
$\gamma \colon V_{r+1} \to V_{r+1}$ is the operator given by  $\gamma \spa (G[y_1,\ldots,y_{r+1};W])
 = G[y_2,\ldots,y_{r+1},q y_1; W]$.

In addition, \eqref{e Ti def 1}--\eqref{e d+ def} imply that the loop $y_i \in \AA_{t,q}$ at vertex  $r$ acts on  $V_r$ as multiplication by  $y_i$, for any $i \le r$.
\end{prop}

For the action defined by Carlsson-Mellit, we only need the following result relating the two actions.
Let  $\Q_r$ denote the $\QQ(q,t)$-linear isomorphism between the subspace  $y_1\cdots y_rV_r \subseteq V_r$ and  $V_r$ given by
\begin{align}
\label{e Q def}
\Q_r \colon y_1 \cdots y_r V_r \to V_r,  \quad \quad y_1 \cdots y_r v \mapsto  v.
\end{align}

\begin{prop}
\label{CarMel vs Mel}
The spaces $(y_1 \cdots y_r V_r)_{ r \ge 0}$ form a submodule for Mellit's action of $\AA_{t,q}$ on  $V_*$, and we have the following identities of operators on  $V_*$:
\begin{align}
\label{e CarMel vs Mel 1}
T_i^{CM}  &=\Q_r \spa T_i \spa \Q_r^{-1}, \\
\label{e CarMel vs Mel 2}
d_-^{CM}&= \Q_{r-1} \spa (-d_-) \spa \Q_{r}^{-1},   \\
\label{e CarMel vs Mel 3}
 d_+^{CM} &= \Q_{r+1} \spa (-d_+)\spa \Q_{r}^{-1},  \\
\label{e CarMel vs Mel 4}
 \spa d_+^{* CM}   &=\Q_{r+1}\spa y_1 d_+^* \spa \Q_{r}^{-1},
\end{align}
where  $i < r$, and $T_i^{CM}$ and $T_i$ are loops at quiver vertex  $r$.
\end{prop}
\begin{proof}
Identities \eqref{e CarMel vs Mel 1}--\eqref{e CarMel vs Mel 3} are immediate
from \cite[Remark 3.11]{Mellit16}.
The action of  $d_+^{* CM}$, as defined in \cite[Theorem 6.1]{CarMel}, is the same as $d_+^*$ given in \eqref{e d+* def}.
Then \eqref{e CarMel vs Mel 4} holds by the following direct computation. For any
$F[y_1,\ldots,y_r; W] \in V_r$,
\begin{multline*}
y_1 d_+^* \Q_r^{-1}F = y_1 d_+^* (y_1 \cdots y_r F)
= y_1 \spa \gamma \spa (y_1 \cdots y_r F[y_1,\ldots,y_r; W+(t-1)y_{r+1}])\\
= y_1 y_2 \cdots y_{r+1} \spa \gamma \spa ( F[y_1,\ldots,y_r; W+(t-1)y_{r+1}])
=\Q_{r+1}^{-1} \spa d_+^{* CM} F.  \qedhere
\end{multline*}
\end{proof}

The third action is the main one we will use.
This version was studied by Ion-Wu \cite{IonWu} who connected it to a representation
of the $+$ stable limit DAHA  $\H^+$ (see \S\ref{ss IonWu DAHA}),
and it is the action we will use to connect to signed flagged LLT polynomials.

As in \S\ref{s intro}, let $\P(r) = \QQ(q,t)[x_1,\dots,x_r] \otimes \Lambda_{\QQ(q,t)}(x_{r+1}, \dots)$.
Define the isomorphism $\PP_r$ to be the $\QQ(q,t)$-linear map determined by
\begin{gather}
\label{e Pr def}
\begin{aligned}
\PP_r \colon \P(r) \xrightarrow{\ \cong \ } V_r,  \hspace{3.5cm}\\
x_1^{\eta_1}\cdots x_r^{\eta_r} g[x_{r+1}+ \cdots ] \mapsto y_1^{\eta_1}\cdots y_r^{\eta_r} g[W/(t-1)],
 \end{aligned}
\end{gather}
where $\eta_1,\dots, \eta_r \in \NN$, and $g$ is an arbitrary symmetric function.
Thus, we have an action of  $\AA_{t,q}$ on  $(\P(r))_{r \ge 0}$
induced by Mellit's action on $(V_{r})_{r \ge 0}$ via the
isomorphisms $\PP_r$.
We write $d_+^\PP, d^{* \PP}_+, d_-^\PP, T_i$ for the generators of  $\AA_{t,q}$ when using this action.
For example,  $d_-^\PP \colon \P(r) \to \P(r-1)$ and $d_+^\PP \colon \P(r) \to \P(r+1)$ are defined by
\begin{align}
\label{e d-PP def}
d_-^\PP = \PP_{r-1}^{-1} \, d_- \, \PP_r,  \quad \quad d_+^\PP = \PP_{r+1}^{-1} \, d_+ \, \PP_r.
\end{align}

Explicit formulas for these operators are readily calculated from Proposition \ref{p Mellit action},
and we record them in the proposition below (see \cite[\S7.5]{IonWu}).
Let  $\iota_r \colon \P(r) \to \P(r+1)$ denote the inclusion map.
Let  $\mathcal{B}_n$ denote Jing's Hall-Littlewood vertex operator \cite{Jing},
which acts on a
symmetric function  $f(Z) \in \Lambda(Z)$ by
\begin{align}
\label{e Jing}
\mathcal{B}_n \spa f(Z) = \langle u^0 \rangle u^n f[Z-u] \Omega[(1-t)u^{-1}Z].
\end{align}
\begin{prop}
\label{p Aqt action v3}
The following formulas define a quiver representation of
\(\AA_{t,q}\) on \( (\P(r))_{r \ge 0}\).
The action of
$d_-^\PP$ at quiver vertex  $r$ is the
unique $\QQ(q,t)$-linear map  $\P(r) \to \P(r-1)$ determined by
\begin{align}
\label{e d-PP def 2}
& d_-^\PP (x_1^{\eta_1} \cdots x_r^{\eta_r}
g(x_{r+1}, \dots) ) = x_1^{\eta_1} \cdots x_{r-1}^{\eta_{r-1}} \mathcal{B}_{\eta_r} (g[x_r + x_{r+1} + \cdots]),
\end{align}
where $g$ is any symmetric function and  $\mathcal{B}_{\eta_r}$ is as in \eqref{e Jing} with  $Z = x_r + x_{r+1} +\cdots$.
The action of the other arrows originating at quiver vertex  $r$ are given by
\begin{align}
\label{ep Ti Aqt action v3}
& T_i \spa f =  s_i f + (1-t) x_i \frac{f - s_i f}{x_i - x_{i+1}}, \\
& d_+^\PP \spa f =  -T_1 \cdots T_r \, (x_{r+1} \, \iota_r (f)), \\
& d_+^{* \PP} f =  \tilde{\gamma} \, \iota_r(f),
\end{align}
where $f= f(\xx) = f(x_1,x_2,\dots) \in \P(r)$,  $i \in [r-1]$, and
$\tilde{\gamma} \colon \P(r+1) \to \P(r+1)$ is the operator given by  $\tilde{\gamma} (g(\xx))
= g(x_2,\ldots,x_{r+1},q x_1, x_{r+2}, x_{r+3}, \dots)$.
\end{prop}

\subsection{Ion-Wu's  \texorpdfstring{$+$}{+} stable limit DAHA}
\label{ss IonWu DAHA}
In \cite{IonWu}, Ion and Wu propose two ways to make precise the idea of a limiting object of the $\H_n$ as  $n \to \infty$.
The version needed for this paper is a kind of  $t$-adic limit of the positive subalgebras $\H_n^+
\defeq \langle \X_1,\dots, \X_n, \Y_1,\dots, \Y_n, T_1, \dots, T_{n-1} \rangle \subseteq \H_n$.
Its presentation is given as follows.

\begin{defn}[{\cite[Definition 4.1]{IonWu}}]
\label{d H+}
The \emph{$+$ stable limit DAHA}
\(\H^+\) is the \(\QQ(q,t)\)-algebra with generators \(T_i, \X_i\),
and \(\Y_i\) for all \(i \geq 1\), satisfying the following relations:
\begin{gather*}
  T_i T_j = T_j T_i\,,\ |i-j| > 1 \,, \\
  T_i T_{i+1} T_i = T_{i+1} T_i T_{i+1} \,,\ i \geq 1\,, \\
  (T_i-1)(T_i+t) = 0 \,, \ i \geq 1\,,\\[.6mm]
  t\, T_i^{-1} \X_i T_i^{-1} = \X_{i+1}\,, \ i \geq 1\,,\\
  T_i \X_j = \X_j T_i\,,\ j \neq i,i+1\,, \\
  \X_i \X_j = \X_j \X_i\,,\ i,j \geq 1\,,\\[.6mm]
  t^{-1} T_i \Y_i T_i = \Y_{i+1}\,,\ i \geq 1\,,\\
  T_i \Y_j = \Y_j T_i\,,\ j \neq i,i+1\,,\\
  \Y_i \Y_j = \Y_j \Y_i\,,\ i,j \geq 1\,,\\[.6mm]
  \Y_1 T_1 \X_1 = \X_2 \Y_1 T_1 \,.
\end{gather*}
\end{defn}

As in \S\ref{s intro}, the space of \emph{almost symmetric functions}  $\Pas$ is
\begin{align}
  \Pas = \bigcup_{r \ge 0}  \P(r) \ \subseteq \, \QQ(q,t)[[x_1,\dots]].
\end{align}

It is shown in \cite{IonWu} that
there is an action of  $\H^+$ on  $\Pas$
 in which
$\X_i$ acts by multiplication by $x_i$,
$T_i$ acts on  $f(\xx) \in \Pas$ by\footnote{The same formula as in \eqref{ep Ti Aqt action v3}, but there the meaning is slightly different since  $T_i \in \AA_{t,q}$ stands for a different loop at each quiver vertex  $r > i$.}
\begin{align}
\label{e Ti IonWu}
T_i f &=  s_i f + (1-t) \spa x_i \, \frac{f - s_i f}{x_i - x_{i+1}} \, ,
\end{align}
and the action of $\YIW_i$ is defined as follows: first define operators on  $\QQ(q,t)[\xx_N]$ by
\begin{align}
\Phi \spa g &= g(x_2, \dots, x_N, q \spa x_1), \\
\widehat{x}_1 \spa g &= g(0,x_2,\dots, x_n).
\end{align}
Then the action of $\YIW_i$ on  $f(\xx) \in \Pas$ is given by
\begin{align}
\label{e y IW}
\YIW_i \spa f(\xx) = \lim_{N \to \infty}  t^{N-i+1}  \spa T_{i-1} \cdots T_1 \spa (1- \widehat{x}_1) \spa \Phi  \spa T_{N-1}^{-1} \cdots T_i^{-1} f[\xx_N],
\end{align}
where the sequence converges $t$-adically to the limit.

For context, this action can be viewed as a limit of the action of  $\H_n^+$ on  $\QQ(q,t)[\xx_N]$,
which, in the conventions of \cite{IonWu}, is given as follows: $\X_i$ acts
by multiplication by $x_i$,
$T_i$ by \eqref{e Ti IonWu}, and $\Y_i \in \H_n^+$ by  $t^{-i+1} T_{i-1} \cdots T_1 \spa  \X_1  \spa \Phi  \spa T_{N-1}^{-1} \cdots T_i^{-1}$.

Recall from \eqref{e Pr def} the isomorphism
$\PP_r \colon \P(r) \xrightarrow{\ \cong \ } V_r$ given by
$x_1^{\eta_1}\cdots x_r^{\eta_r} g[x_{r+1}+ \cdots ] \mapsto y_1^{\eta_1}\cdots y_r^{\eta_r} g[W/(t-1)]$.

\begin{thm}[{\cite[Theorems 7.13 and 7.14]{IonWu}}]
\label{t IW two actions}
The action of  $\X_i, \YIW_i, T_j \in \H^+$ on  $\P(r)$ and $y_i, z_i, T_j \in \AA_{t,q}$ on  $V_r$, for  $1 \le i \le r$,  $1 \le j <r$,
are related via  $\PP_r$ as follows:
\begin{align}
\X_i = \PP^{-1}_r y_i \PP_r, \  \
\YIW_i = \PP^{-1}_r z_i \PP_r, \ \
T_j = \PP^{-1}_r T_j \PP_r.
\end{align}
In other words, the following diagrams commute:
\[
\begin{tikzcd}
  \P(r) \ar[r,"\X_i"] \ar[d,"\PP_r"] & \P(r) \ar[d,"\PP_r"]\\
  V_r \ar[r,"y_i"] & V_r
\end{tikzcd}
\quad \quad
\begin{tikzcd}
  \P(r) \ar[r,"\YIW_i"] \ar[d,"\PP_r"] & \P(r) \ar[d,"\PP_r"]\\
  V_r \ar[r,"z_i"] & V_r
\end{tikzcd}
\quad \quad
\begin{tikzcd}
  \P(r) \ar[r,"T_j"] \ar[d,"\PP_r"] & \P(r) \ar[d,"\PP_r"]\\
  V_r \ar[r,"T_j"] & V_r
\end{tikzcd}.
\]
\end{thm}

\begin{remark}
(i)
The chief advantage of the  $\H^+$ action over that of $\AA_{t,q}$ (in Proposition \ref{p Aqt action v3}) is that
$z_i \in \AA_{t,q}$ stands for many elements, one for each quiver
vertex  $r$ with  $r\ge i$, each of which acts separately on the different spaces $(V_r)_{r \ge i}$,
whereas their counterpart $\YIW_i$ in $\H^+$ is a single element which acts on all of  $\Pas$ uniformly (and similarly for other generators).
In particular, this means that
$\YIW_i \iota_r = \iota_r \YIW_i$ for the inclusion  $\iota_r \colon \P(r) \to \P(r+1)$, whereas the corresponding property
$\PP_{r+1}^{-1} z_i\PP_{r+1} \iota_r = \iota_r \PP_{r}^{-1} z_i\PP_{r}$ is not immediate
from the definition \eqref{e zi def 1}--\eqref{e zi def 2}.

(ii) Theorem \ref{t IW two actions} can be thought of informally as saying that
$\H^+$ is a subalgebra of  $\AA_{t,q}$, but this is not technically true because of the inherent difference between these actions just discussed in (i).
\end{remark}

\section{Nonsymmetric plethysm, Weyl and Hecke symmetrization, and \texorpdfstring{$d_-$}{d\_-}}
\label{s nspleth}

In this section we define stable versions,  $\hsym_r$ and  $\Weyl_r$, of Hecke and Weyl symmetrization,
recall a result of Bechtloff Weising \cite{BechtloffWeising23} that $\hsym_r = d_-^\PP$, and prove that  $d_-^\PP$ conjugated by nonsymmetric plethysm is equal to  $\Weyl_r$.
The latter was proved in \cite[\S 7.8]{BHMPS-nspleth}, but the proof goes through several
 results involving flagged LLT polynomials, and
it seems worthwhile to give a short self-contained proof.

\subsection{Weyl and Hecke symmetrization}
\label{ss Weyl and Hecke sym}

Following \cite[\S 7.8]{BHMPS-nspleth},
define a stable version of Weyl symmetrization $\Weyl_{r} : \P(r) \to \P(r-1)$ by
\begin{align}
\label{e Weylr strong converge} \Weyl_{r} f(\xx) = \lim_{N \to
\infty} \pi_{x_r, x_{r+1},\dots,x_N} f[\xx_N],
\end{align}
where $\pi_{x_{r}, \dots,x_N}$ is the operator  $\pi_w$ (see \S \ref{ss atom pos})
for  $w$ the longest permutation in
the Young subgroup $\SS_{(1^{r-1}; N-r+1)}$ of the symmetric group $\SS _{N}$ which permutes the last $N-r+1$ indices.
By \cite[\S 7.7]{BHMPS-nspleth}, the sequence on the right hand side converges strongly to a
well-defined limit.

Following \cite{BechtloffWeising23}, define the partial Hecke symmetrizer   $\hsym_{r+1}^N \colon \QQ(q,t)[\xx_N] \to \QQ(q,t)[\xx_N]$ by
\begin{equation}\label{e:Hecke-symmetrizer}
\hsym_{r+1}^N = \frac{1}{[N-r]_t !} \sum _{w\in \SS_{(1^r; N-r)}} t^{{\binom{N-r}{2}} - \ell(w)} T_{w},
\end{equation}
where $T_w = T_{i_1} T_{i_2}\cdots T_{i_k}$
for any reduced expression $w=s_{i_1}s_{i_2}\cdots s_{i_k}$ with $T_i$ acting as in \eqref{e Ti IonWu},
and  $[N-r]_t ! = \prod_{k = 1}^{N-r} [k]_t$ for $[k]_t = \frac{1-t^{k}}{1-t}$.
The operator  $\hsym_{r+1}^N$ is normalized so that $\hsym_{r+1}^N \, g(\xx_N) =g(\xx_N)$ for any polynomial $g(\xx_N)$ symmetric in  $x_{r+1}, \dots, x_N$.

Partial Hecke symmetrization gives rise to an operator $\hsym_{r} \colon \P(r) \to \P(r-1)$ defined on any  $f(\xx)\in \P(r)$ by
\begin{align}
\label{e H sym converge} \hsym_{r} f(\xx) = \lim_{N \to
\infty} \hsym_{r}^N f[\xx_N],
\end{align}
where, as shown in \cite[\S 4.3]{BechtloffWeising23}, the right
hand side converges $t$-adically to a well-defined limit.

\begin{prop}[{\cite[Proposition 4.25]{BechtloffWeising23}}]
\label{p: d- and hsym}
The operator $\hsym_r \colon \P(r) \to \P(r-1)$ is the same as $d_-^\PP \colon \P(r) \to \P(r-1)$ defined in \S\ref{ss two actions}.
\end{prop}

\begin{remark}
\label{rem: hsym and d-}
We will also make use of Proposition \ref{p: d- and hsym} as a notational tool:
since $d_-^\PP \in \AA_{t,q}$ actually stands for the different operators  $d_-^\PP : \P(r) \to \P(r-1)$ for each $r$,
it is sometimes useful to specify the  $r$ in the notation. We adopt the convention of using $\hsym_r$
when this is the case.
\end{remark}

\begin{remark}
(i) The version of the Demazure-Lusztig operators $T_{i}$ used here agrees
with \cite{BechtloffWeising23, CarMel, IonWu, Mellit16}.  In terms of the
notation here, the Demazure-Lusztig operators in \cite{BHMPS-nspleth, HagHaiLo08} are $t\, T_{i}^{-1}$.

(ii) The Hecke symmetrizer  $\hsym_{r+1}^N$ and its stable version $\hsym_{r}\colon \P(r) \to \P(r-1)$ agree with those in \cite[\S5.1 and \S7.8]{BHMPS-nspleth} and \cite[\S4.3]{BechtloffWeising23}, except that the latter reference defines a version of $\hsym_r$ on a larger domain, from
 $\Pas$ to  $\P(r-1)$.
\end{remark}

\subsection{The nonsymmetric plethysm conjugate of  \texorpdfstring{$d_-^\PP$}{dᴾ\_-}}
In this subsection, we
prove the operator identity $\Pisf_{r-1} \, d_-^\PP \, \Pisf_r^{-1} = \Weyl_{r}$.
The proof uses an alternative description---Theorem \ref{t nspleth on
flagh}, below---of the $r$-nonsymmetric plethysm map $\Pisf_r$ defined
in \eqref{e: nspleth def}, in terms of flagged versions of the complete
homogeneous symmetric functions $h_\lambda(\xx)$.

We continue to use the notation $\xx_r = x_1,\dots, x_r$ and $\xx = x_1,x_2, \cdots$;
recall that we interpret these as $\xx_r = x_1 + \cdots + x_r$ and  $\xx = x_1 + x_2 + \cdots$
when they are used in a plethystic formula.
 For \(\eta \in \NN^r\), set
\begin{align}
\flagh_\eta(\xx_r) = h_{\eta_1}(x_1) h_{\eta_2}(x_1,x_2) \cdots h_{\eta_r}(x_1,x_2,\ldots,x_r).
\end{align}
As \((\eta|\lambda)\) ranges over \(\NN^r \times \Par\), the elements
\begin{align}
  \label{e: stable-flag-h-def}
  \flagh_\eta(\xx_r)h_\lambda(\xx)
\end{align}
form a basis for  $\P(r)$.
Additionally, using plethystic notation, set
\begin{equation}
\flagh^\pm_\eta(\xx_r) = h_{\eta_1}(x_1)
h_{\eta_2}[x_2+(1-t)x_1]
h_{\eta_2}[x_3+(1-t)(x_1+x_2)]
\cdots
h_{\eta_r}[x_r+(1-t)(\xx_{r-1})]\,.
\end{equation}
Then, as  $(\eta | \lambda)$ ranges over  $\NN^r \times \Par$, the elements
\begin{align}
  \label{e: stable-sign-flag-h-def}
  \flagh^\pm_\eta(\xx_r)h_\lambda[(1-t)\xx]
\end{align}
 also form a basis for  $\P(r)$.

\begin{thm}[{\cite{BHMPS-nspleth}}]
\label{t nspleth on flagh}
The $r$-nonsymmetric plethysm map  $\Pisf_r \colon \P(r) \to \P(r)$
acts on the basis in \eqref{e: stable-sign-flag-h-def} as follows:
\begin{align}
\Pisf_r \big( \spa \flagh^\pm_\eta(\xx_r)h_\lambda[(1-t)\xx] \spa\big)
=\flagh_\eta(\xx_r)h_\lambda(\xx).
\end{align}
\end{thm}
\begin{proof}
This follows from the definition \eqref{e: nspleth def} of  $\Pisf_r$
and \cite[Definition~4.1.1 and Proposition~4.3.5]{BHMPS-nspleth}.
\end{proof}

\begin{lemma}\label{lem:d-minus-P-on-flagged-signed-h}
For any \(\eta \in \NN^r\) and partition \(\lambda\), we have
\begin{equation}
\label{e:d-minus-P-on-flagged-signed-h}
d_-^\PP \big( \spa \flagh_\eta^{\pm}(\xx_r) h_\lambda[(1-t)\xx] \spa \big) = \flagh_{(\eta_1,\ldots,\eta_{r-1})}^{\pm}(\xx_{r-1}) h_{(\eta_r,\lambda)_+}[(1-t)\xx],
\end{equation}
where $(\eta_r,\lambda)_+$ denotes the partition rearrangement of the concatenation of  $\eta_r$ and $\lambda$.
\end{lemma}

\begin{proof}
We use the description of $d_-^\PP$ in Proposition \ref{p Aqt action v3}.
Since the factor $\flagh_{(\eta_1,\ldots,\eta_{r-1})}^{\pm}(\xx_{r-1})$
of  $\flagh_\eta^{\pm}(\xx_r)$ does not depend on  $x_r$,
\eqref{e:d-minus-P-on-flagged-signed-h} reduces to the case  $\eta_1=\dots = \eta_{r-1} = 0$.
For this case, we let
$Z = x_{r}+x_{r+1}+\cdots$ and compute
\begin{align*}
 & d_-^\PP \, h_{\eta_r}[x_r+(1-t)(\xx_{r-1})] \, h_\lambda[(1-t)\xx] \\
& =  d_-^\PP \, h_{\eta_r}[x_r+(1-t)(\xx_{r-1})] \, h_\lambda[(1-t)(\xx_{r-1}+x_r+ (x_{r+1} + \cdots))] \\
& =
\langle u^0 \rangle  h_{\eta_r}[u+(1-t)(\xx_{r-1})] \, h_\lambda[(1-t)(\xx_{r-1}+u+(Z-u))] \, \Omega[(1-t)u^{-1} Z] \\
& =
h_\lambda[(1-t)(\xx_{r-1}+Z)] \
\langle u^0 \rangle  h_{\eta_r}[u+(1-t)(\xx_{r-1})] \Omega[(1-t)u^{-1} Z] \\
& =
h_\lambda[(1-t)(\xx_{r-1}+Z)] \,
h_{\eta_r}[(1-t)Z+(1-t)(\xx_{r-1})]
\\
& =
h_{(\eta_r, \lambda)_+}[(1-t)\xx],
\end{align*}
where the second to last equality follows from a general plethystic
identity that we now verify:
\begin{multline*}
\langle u^0 \rangle  h_{\eta_r}[u+A] \Omega[u^{-1}B]
=\langle u^0 \rangle
\sum_{d \ge 0} u^d h_{\eta_r-d}[A] \Omega[u^{-1}B]
=\sum_{d \ge 0} h_d[B] h_{\eta_r-d}[A]
= h_{\eta_r}[B+A]. \qedhere
\end{multline*}
\end{proof}

\begin{prop}
\label{p dminus and sym}
We have the following identity of operators  $\P(r) \to \P(r-1)$,
\begin{align}
\label{ep dminus and sym}
\Pisf_{r-1} \, d_-^\PP \, \Pisf_r^{-1} = \Weyl_{r}\, .
\end{align}
\end{prop}
\begin{proof}
We compute both sides on the basis of elements $\flagh_\eta(\xx_r)  h_\lambda(\xx)$.
Since $h_{\eta_r}(\xx_r)$ is the Demazure character  $\Dcal_{(0^{r-1}, \eta_r)}$, we have
$\pi_{x_r,\dots,x_N} h_{\eta_r}(\xx_r) = h_{\eta_r}(\xx_N)$. Hence,
\begin{equation}
\label{e Weyl sym}
\Weyl_{r} \,
\big( \spa \flagh_\eta(\xx_r)  h_\lambda(\xx) \spa\big)
= \flagh_{(\eta_1,\ldots,\eta_{r-1})}(\xx_{r-1})
  h_{(\eta_r,\lambda)_+}(\xx) \,.
\end{equation}
On the other hand, Theorem \ref{t nspleth on flagh} and Lemma~\ref{lem:d-minus-P-on-flagged-signed-h} give
 the matching identity
\begin{equation}
\Pisf_{r-1} \, d_-^\PP \, \Pisf_r^{-1}
\big( \spa \flagh_\eta(\xx_r)  h_\lambda(\xx) \spa\big)
= \flagh_{(\eta_1,\ldots,\eta_{r-1})}(\xx_{r-1})
  h_{(\eta_r,\lambda)_+}(\xx) \,.   \qedhere
\end{equation}
\end{proof}

\section{Flagged LLT polynomials}
\label{s LLT}

In \cite{BHMPS-nspleth} we
defined flagged versions of signed and unsigned LLT polynomials
and showed that the nonsymmetric plethysm map
takes signed flagged LLT polynomials to unsigned ones.
We begin this section
by stating a version of this result in a special case, using notation for LLT polynomials
in terms of a Dyck path which encodes the attacking pairs,
in line with the treatment of LLT polynomials
by Carlsson-Mellit \cite{CarMel}.
We then show (\S\ref{s:dpa2sllt})
how the methods of \cite{CarMel} yield a formula for signed flagged LLT polynomials (in the above special case)
in terms of the Dyck path algebra.

\subsection{Dyck paths}
\label{ss Dyck path stats}

We recall some definitions related to Dyck paths, following the
conventions in \cite{Mellit16}.  We also use terminology for partial Dyck paths
close to that in \cite{CarMel}.

For Dyck paths and skew shapes, we work with Cartesian coordinates
 and refer to the lattice square with northeast corner  $(c,r)$ as the \emph{box} $(c,r)$,
and we say it is in row  $r$ and column  $c$.

Let  $a,b \in \ZZ_+$.
An \emph{$(a,b)$-Dyck path} is a sequence of north and east lattice steps going from
$(0,0)$ to $(a,b)$ which stays weakly above the \emph{diagonal}, the
line segment from  $(0,0)$ to  $(a,b)$.
For an  $(a,b)$-Dyck path  $\pi$,
let  $\Area(\pi)$
denote the set of full boxes below  $\pi$ and above the diagonal.
Set  $\area(\pi) = |\Area(\pi)|$.

An \emph{$r$-partial  $d$-Dyck path}  $\pi'$ is a sequence of north and east lattice steps going from
$(0,r)$ to $(d,d)$ which stays weakly above the diagonal from  $(0,0)$ to  $(d,d)$.
Let  $\DD_r(d)$ denote the set of $r$-partial  $d$-Dyck paths.
Let  $\pi' \in \DD_r(d)$.
We set  $\Area(\pi') = \Area( \mathsf{N}^{r} \pi')$, where  $\mathsf{N}^{r} \pi' $ is the  $(d,d)$-Dyck path
obtained by adding  $r$ north steps to the southwest end of $\pi'$.
A \emph{corner} of $\pi'$ is a box above  $\pi'$ such that the boxes immediately south and immediately east of it are below  $\pi'$;
a \emph{marking}  $\Sigma$ of $\pi'$ is a subset of the corners of  $\pi'$ and its elements are \emph{marked corners}.
\begin{example}
\label{ex:partial Dyck path}
Shown below is a partial Dyck path $\pi' \in \DD_2(6)$,
with  $|\Area(\pi')| =7$, corners $\{(1,4),(3,5),(4,6)\}$, and a marking  $\Sigma = \{(1,4),(4,6)\}$ indicated by $\star$'s.
\smallskip
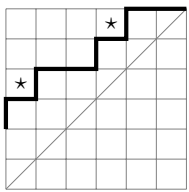
\begin{align*}
\raisebox{-1.7cm}{
\begin{tikzpicture}[scale=.4]
\draw[help lines] (0,0) grid (6,6);
\draw[very thin, gray] (0,0) -- (6,6);
\draw[ultra thick] (0,2) -- (0,3) -- (1,3) -- (1,4) -- (2,4) -- (3,4) -- (3,5) -- (4,5) -- (4,6) -- (5,6) --  (6,6);
\node at (0.5, 3.5) {\scriptsize $\star$};
\node at (3.5, 5.5) {\scriptsize $\star$};
\end{tikzpicture}}
\end{align*}
\end{example}

Returning to the case of any $a,b \in \ZZ_+$,
let  $s_- = \frac{b}{a}-\epsilon$ for small  $\epsilon > 0$.
We order north steps by the distance between their southern ends and the line  $y=s_- x$.
Let  $\pi$ be an  $(a, b)$-Dyck path.
The \emph{sweeping reading order} is this order restricted to the north steps of  $\pi$.
We say a north step  $z$ of  $\pi$
\emph{attacks} all the north steps of  $\pi$ which are greater than  $z$
in the sweeping reading order,
but less than the north step just above  $z$ (which need not be a north step of  $\pi$).
Numbering the north steps of $\pi$ in sweeping reading order,
the \emph{attacking Dyck path} of $\pi$ is the unique $(b,b)$-Dyck
path $\pi'$ such that a box $(i,j)$ with $i<j$ belongs to
$\Area(\pi')$ if and only if the north steps numbered $i$ and $j$ form
an attacking pair in $\pi$.  This is illustrated in
\cite[Figure 1 (pg. 22)]{Mellit16} and
Figure \ref{fig:LLT versions}(a, b) below.

\subsection{(Signed) flagged row LLT polynomials}
\label{ss flagged LLT polynomials}

Given  $\pi \in \DD_r(d)$ and a marking $\Sigma$ of  $\pi$,
a {\em $(\pi, \Sigma)$-row semistandard word} is a map
$T\colon [d] \rightarrow \ZZ _{+}$ such that  $T(j) \le T(i)$ for every  $(i,j) \in \Sigma$.
We say $T$ is \emph{flagged} if  $T(1) \le 1, \spa \dots, \, T(r) \le r$. Denote the set of
flagged $(\pi, \Sigma)$-row semistandard words by
$\FSSYT (\pi, \Sigma)$.
An {\em attacking inversion} in $T$ is
an ordered pair $(i,j) \in \Area(\pi)$ such that $T(j)\le T(i)$.
The number of attacking inversions in $T$ is denoted $\inv(T)$.
(Note: this notion of attacking inversion will actually line up with
attacking non-inversions via a bijection with tableaux in \S\ref{ss
Dpath vs shapes llt}.)

\begin{defn}
\label{d flag LLT}
For $\pi \in \DD_r(d)$ and a marking $\Sigma$ of $\pi$,
the associated \emph{flagged row LLT polynomial} is given by
\begin{align}
\Grow_r(\pi, \Sigma)(\xx;t) = \sum_{T \in \FSSYT(\pi, \Sigma)} t^{\inv(T)} \xx^T,
\end{align}
where  $\xx^T = \prod_{i \in [d]} x_{T(i)}$.
\end{defn}

\begin{example}
\label{ex row flag llt}
We illustrate \(\FSSYT(\pi,\Sigma)\) for the partial
Dyck path \(\pi \in \DD_2(4)\) shown below and marking
\(\Sigma=\{(2,4)\}\), indicated with the \(\star\).
The words  $T \in \FSSYT(\pi,\Sigma)$ are drawn with letters
\(T(1)T(2)T(3)T(4)\) written along the diagonal from southwest to
northeast and \(t\)'s labeling the boxes of \(\Area(\pi)\) which are
attacking inversions.
\newcommand{\thisllt}[4]{
\tikz[scale=.32,baseline=.3cm]{
\draw[help lines] (0,0) grid (4,4);
\draw[thick] (0,2)--(1,2)--(1,3)--(2,3)--(2,4)--(4,4);
\node at (1.5, 3.5) {{\tiny $\star$}};
\node at (.5, .5) {\tiny $#1$};
\node at (1.5, 1.5) {\tiny $#2$};
\node at (2.5, 2.5) {\tiny $#3$};
\node at (3.5, 3.5) {\tiny $#4$};
\pgfmathparse{#2<=#1}
\ifnum\pgfmathresult=1
  \node at (.5,1.5) {\tiny \(t\)};
\fi
\pgfmathparse{#3<=#2}
\ifnum\pgfmathresult=1
  \node at (1.5,2.5) {\tiny \(t\)};
\fi
\pgfmathparse{#4<=#3}
\ifnum\pgfmathresult=1
  \node at (2.5,3.5) {\tiny \(t\)};
\fi }
}
\begin{center}
\noindent\begin{tikzpicture}[ampersand replacement=\&]
  \matrix[matrix of math nodes, nodes={anchor=center}] (m) {
    T \in \FSSYT(\pi, \Sigma) \&
    \thisllt{1}{1}{1}{1} \&
    \thisllt{1}{1}{2}{1} \&
    \thisllt{1}{1}{3}{1} \&
    \thisllt{1}{1}{4}{1} \&
    \thisllt{1}{1}{5}{1} \&
    \cdots \\
    t^{\inv(T)} \xx^T
    \& t^3 x_1^4
    \& t^2 x_1^3 x_2
    \& t^2 x_1^3 x_3
    \& t^2 x_1^3 x_4
    \& t^2 x_1^3 x_5
    \\
    \&
    \thisllt{1}{2}{1}{1} \&
    \thisllt{1}{2}{2}{1} \&
    \thisllt{1}{2}{3}{1} \&
    \cdots \&
    \thisllt{1}{2}{1}{2} \&
    \thisllt{1}{2}{2}{2} \&
    \thisllt{1}{2}{3}{2} \&
    \cdots
    \\
    \&
    t^2 x_1^3 x_2 \&
    t^2 x_1^2 x_2^2 \&
    t \spa x_1^2 x_2 x_3 \&
    \&
    t \spa x_1^2 x_2^2 \&
    t^2 x_1 x_2^3 \&
    t \spa x_1 x_2^2 x_3 \\
  };
\end{tikzpicture}
\end{center}
Thus,
\(\Grow_2(\pi;\Sigma)(\xx;t) = t^3 x_1^4 + 2 t^2 x_1^3 x_2 + (t^2+t)
x_1^2 x_2^2 + t^2 x_1 x_2^3 + (t^2 x_1^3 + t x_1^2 x_2 + t x_1 x_2^2) e_1(x_3,x_4,\ldots)\).
\end{example}

We now let  $\Acal$ denote the totally ordered alphabet  $1 < \bar{1} < 2 < \bar{2} \cdots $
consisting of a \emph{positive letter} $v$ for each $v\in \ZZ_+$
and a \emph{negative letter} $\bar{v}$
for each $v \in \ZZ_+$.
Let $|\bar{v}| = v$ for a negative letter  $\bar{v}$ and  $|v| = v$ for a positive letter  $v$.
For letters  $a,b \in \Acal$, we say that  $ab$ is \emph{row-increasing}
if \(a < b\) or  $a$ and  $b$ are equal positive letters;
we say that  $ab$ is \emph{column-decreasing}
if \(a > b\) or  $a$ and  $b$ are equal negative letters.

A {\em $(\pi, \Sigma)$-row super word} is a map
$T\colon [d] \rightarrow \Acal$ such that for every  $(i,j) \in \Sigma$,
 $T(j)T(i)$ is row-increasing.
We say $T$ is \emph{flagged} if  $T(1) \le 1,\spa \dots,\, T(r) \le r$.  Denote the set of
flagged $(\pi, \Sigma)$-row super words by
$\FSSYT^{\pm}(\pi, \Sigma)$.
An {\em attacking inversion} in $T$ is an
ordered pair $(i,j) \in \Area(\pi)$ such that
$T(j)T(i)$ is row-increasing.
The number of attacking inversions in $T$ is denoted $\inv (T)$.

Note that a flagged $(\pi, \Sigma)$-row semistandard word is the same as
a flagged $(\pi, \Sigma)$-row super word with all its letters positive,
and the two notions of attacking inversions are compatible.

\begin{defn}
\label{def:LLT signed rows}
For $\pi \in \DD_r(d)$ and a marking $\Sigma$ of $\pi$,
the associated \emph{signed flagged row LLT polynomial} is given by
\begin{align}
\label{e LLT signed rows}
\Growpm_r(\pi, \Sigma)(\xx;t) = \sum_{T \in \FSSYT^{\pm}(\pi, \Sigma)} t^{\inv(T)} (-t)^{m(T)} \xx^{|T|},
\end{align}
where $m(T)$ denotes the number of negative letters in  $T$,
and $\xx ^{|T|} = \prod_{i \in [d]} x_{|T(i)|}$.
\end{defn}

\subsection{Comparison to the flagged LLT polynomials of \texorpdfstring{\cite{BHMPS-nspleth}}{[7]}}
\label{ss Dpath vs shapes llt}

In \cite{BHMPS-nspleth}, we defined (signed) flagged LLT polynomials
in a general setting, where they are indexed by a permutation  $\sigma$ and a tuple
 $\nubold = (\nu^{(1)}, \dots, \nu^{(p)})$ in which each  $\nu^{(i)}$ is a generalization of skew shape called a ragged-right skew shape.
The (signed) flagged row LLT polynomials defined above are essentially a special case of these, as we now make precise.
To avoid some subtleties arising in the fully general setup of \cite{BHMPS-nspleth},
we will only review this version in a special case, but one which is more than broad enough to include all the flagged LLT polynomials arising in this paper (see Remark \ref{rem:empty rows}).

For \(\alpha, \beta \in \ZZ^l\) with \(\beta_1 \geq \cdots \geq \beta_l\), \(\alpha_1 \geq \cdots \geq \alpha_l\), and \(\alpha_j \leq \beta_j\) for all \(j\), the (French style) \emph{skew shape} \(\beta/\alpha\) is the array of
 boxes whose northeast corners have coordinates \((i,j)\) for \(j = 1,\ldots,l\), and \(\alpha_j < i \leq \beta_j\) (if \(\alpha_j = \beta_j\), then the \(j\)-th row of \(\beta/\alpha\) is empty).

Let $\nubold = (\nu^{(1)}, \dots, \nu^{(p)})$ be a tuple of skew shapes.
We consider the set of boxes of $\nubold$ to be the disjoint union of the boxes 
of each $\nu^{(k)}$, and define 
the {\em adjusted content} of a box $b$
in row $j$, column $i$ of the skew shape $\nu^{(k)}$
to be $\ctild (b) = i-j+k\, \epsilon $, 
where $\epsilon$ is a fixed positive number such that $p\, \epsilon <1$.
The {\em reading order} on $\nubold $ is the total order $<$ on the
boxes of $\nubold $ such that $a<b \Rightarrow \ctild (a)\leq \ctild(b)$ and
it increases from southwest to northeast on boxes of fixed adjusted content.
Let
$\bx{\nubold}{1},\dots,\bx{\nubold}{d}$
denote the boxes of
$\nubold$ in increasing reading order.

Let  $b_1, \dots, b_k$ be the boxes of  $\nubold$ which are on the eastern end of their row and ordered so that
 $b_1 > \dots > b_k$ in reading order.  Let  $r \in \{0,1, \dots, k\}$.
A super tableau on $\nubold $ is a map
$T' \colon \nubold \rightarrow \Acal$,
weakly increasing along rows and
columns, with positive letters increasing strictly on columns and
negative letters increasing strictly on rows.
We say  $T'$ is \emph{flagged} if
$T'(b_1) \le 1$, $T'(b_2) \le 2$, $\ldots,$ $T'(b_r) \le r$ (when  $r=0$, this is no extra condition).

An ordered pair
of boxes $(a,b)$ in $\nubold $ is an {\em attacking pair} if  $0<\ctild (b)-\ctild(a)<1$.
An {\em attacking inversion} in $T'$ is
an attacking pair $(a,b)$ such that $T'(a)T'(b)$ is column-decreasing
(as defined before Definition~\ref{def:LLT signed rows}).
We define $\inv (T')$ to be the number of
attacking inversions in $T'$.

Define
\begin{align}
\label{e nspleth paper Gs0}
\flagGcal_{r, \nubold}(\xx;t) & = \sum_{T'} t^{\inv(T')} \xx^{T'}, \\
\label{e nspleth paper Gs}
\flagGcal^\pm_{r, \nubold}(\xx;t) & = \sum_{T'} t^{\inv(T')} (-t)^{-m(T')} \xx^{|T'|},
\end{align}
where the sum in \eqref{e nspleth paper Gs} is over flagged super tableaux $T'$ on $\nubold$, the sum in \eqref{e nspleth paper Gs0} is over the subset of these
with all positive entries, and  $\xx^{T'}$,  $\xx^{|T'|}$, and  $m(T')$ are as in \S \ref{ss flagged LLT polynomials}.

\begin{remark}
\label{rem:empty rows}
(i) In the language of \cite[\S \S 6.1--6.4 and eq.\ (302)]{BHMPS-nspleth}, assuming that $\nubold$ has no empty rows,
 $\flagGcal_{r, \nubold}(\xx;t)$ and  $\flagGcal^\pm_{r, \nubold}(\xx;t)$ are equal to
$\Gcal_{\nubold, \sigma}[x_1,\dots, x_r,Y, 0, \dots, 0; t]$ and $\Gcal^-_{\nubold, \sigma}(x_1,\dots, x_r,Y, 0, \dots, 0)$, respectively,
where  $Y = x_{r+1}+x_{r+2}+\cdots$, and $\sigma$ is the standard compatible permutation.
To see that these definitions match, one can use
 \cite[Theorem 6.4.10 and Proposition 6.7.2]{BHMPS-nspleth} as well as \cite[Remark 6.4.9(i)]{BHMPS-nspleth},
which shows that certain extra attacking pairs present in the definition
in \cite{BHMPS-nspleth}, but not here, do not actually contribute to
$\inv$ in this case.

(ii) The no empty rows assumption in (i) is necessary because empty rows play a significant bookkeeping role in \cite{BHMPS-nspleth}, which we have not used here.
This is harmless, however, since any  $\nubold$ can be replaced by a tuple $\hat{\nubold}$ without empty rows for which the super tableaux on $\nubold$ and $\hat{\nubold}$
are in a natural bijection that lines up attacking pairs.
\end{remark}

To realize the flagged LLT polynomials of \S\ref{ss flagged LLT polynomials} as a special case of those just defined, we need the following lemma.

\begin{lemma}
\label{l Dpath LLT}
Let  $\pi'$ be a  $(d,d)$-Dyck path and $\Sigma$ be a marking of  $\pi'$.  Then
there exists a tuple of single-row skew shapes  $\mubold = (\mu^{(1)}, \dots,\mu^{(p)})$ such that, for all \(i < j\),
\begin{equation}
\label{e attack0}
\begin{aligned}
&\text{$(i,j)\in \Area(\pi')$ }  \ \ \iff \ \
  \text{\(\mubold[i],\mubold[j]\) forms an attacking pair of \(\mubold\), and } \\
 &\text{$(i,j) \in \Sigma$ } \ \ \iff \ \
 \mubold[i], \mubold[j] \text{ forms a horizontal domino in some }\mu^{(k)}.
 \end{aligned}
\end{equation}
\end{lemma}
\begin{proof}
Let  $C$ denote the set of all corners of  $\pi'$.

The case  $\Sigma = C$ is implicit in \cite{HaHaLoReUl05}, and further explained in \cite[\S2.4]{CarMel}.
Note that there are typically many $\mubold$ satisfying \eqref{e attack0} for a given  $\pi', \Sigma$, but for $\Sigma =C$ there is a
convenient choice arising from the combinatorics of the shuffle theorem.
In the notation of \S\ref{ss Dyck path stats}, the correspondence  $\pi \mapsto \pi'$ is bijective for $a=b=d$ by \cite{Haglund08};
then let $\pi$ be the unique $(d,d)$-Dyck path such that its attacking Dyck path is  $\pi'$,
so that $\Area(\pi')$ coincides with the attacking pairs of north steps of  $\pi$.
Then the tuple $\mubold = (\mu^{(1)}, \dots, \mu^{(p)})$ of single-row skew shapes given by the runs of consecutive north steps in  $\pi$
satisfies \eqref{e attack0}.
More precisely, if  $R_1, \dots, R_p$ are the runs of north steps from south to north, with their southern ends at  $(x_1, y_1), \dots, (x_p, y_p)$ and with lengths  $l_1,\dots, l_p$, then  $\mu^{(i)} = (l_i+ y_i-x_i)/(y_i-x_i)$.  See Figure~\ref{fig:LLT versions} (a)--(c).

In the general case, we obtain a desired tuple of single-row skew shapes by modifying the  $\mubold$
from the previous paragraph.
Let $(i, j) \in C$; this is also a pair such that
$\mubold[i], \mubold[j]$ forms a horizontal domino.
Let  $\tilde{\mubold} = (\mu^{(1)}, \dots, \mu^{(k-1)}, \theta, \tau, \mu^{(k+1)}, \dots, \mu^{(p)})$ be the tuple of single-row skew shapes obtained by
replacing $\mu^{(k)}$ with the two single-row skew shapes $\theta$ and  $\tau$ formed from breaking  $\mu^{(k)}$
between $\mubold[i] $ and $\mubold[j]$; there are two ways to do this, and the one we want is where  $\tilde{\mubold}[i] \in \theta$ and
$\tilde{\mubold}[j] \in \tau$ so that
the attacking pairs of  $\tilde{\mubold}$ and  $\mubold$ are the same.  Then the tuple $\tilde{\mubold}$ and Dyck path $\pi'$ with marking $C \setminus \{(i,j)\}$ satisfy \eqref{e attack0}.
This row breaking can be done independently for each  $(i,j)\in C$,
so we can handle the case of an arbitrary marking by breaking rows for each pair deleted from  $C$.
\end{proof}

Now let $\hat{\pi}' \in \DD_r(d)$ and  $\Sigma$ be a marking of  $\hat{\pi}'$.
Let  $\pi' = \mathsf{N}^{r} \hat{\pi}'$ be the corresponding $(d,d)$-Dyck path
and choose a tuple $\mubold = ((\beta_1)/(\alpha_1), \dots, (\beta_p)/(\alpha_p))$ so that \eqref{e attack0} holds.
Form
\begin{align}
\label{e nubold def}
\nubold &= ((-\alpha_p)/(-\beta_p), \dots, (-\alpha_1)/(-\beta_1)),
\end{align}
which satisfies
\begin{equation}
\label{e attack}
\begin{aligned}
&\text{$(d-j,d-i)\in \Area(\pi')$ }  \ \ \iff \ \
  \text{\(\nubold[i],\nubold[j]\) forms an attacking pair of \(\nubold\), and } \\
 &\text{$(d-j,d-i) \in \Sigma$ } \ \ \iff \ \
 \nubold[i], \nubold[j] \text{ forms a horizontal domino in some }\nu^{(k)}.
 \end{aligned}
 \end{equation}
See Figure \ref{fig:LLT versions} (d).

There is a bijection taking a flagged $(\hat{\pi}', \Sigma)$-row super word  $T$ to a flagged super tableau $T'$ on $\nubold$ given by
setting $T'(\bx{\nubold}{i}) = T(d-i)$.  To see that the flagging conditions match, note that every $(i,j)$ with  $1 \le i < j \le r$ lies in  $\Area(\hat{\pi}')$
and correspondingly the last  $r$ boxes of  $\nubold$ in reading order are mutually attacking; moreover, these boxes are in distinct rows and all lie on the eastern end of their row.  Hence, the flagging condition on $T'$ is $T'(\bx{\nubold}{d}) \le 1, \dots, T'(\bx{\nubold}{d-r}) \le r$, which matches the flagging condition
$T(1) \le 1, \dots, T(r) \le r$ on $T$.
The attacking inversions in  $T$ exactly line up with attacking noninversions in  $T'$.
So the number of attacking inversions in $T'$ is equal to $|\Area(\hat{\pi}')| -\inv(T)$, and thus
\begin{align}
\label{e two signed flag LLT}
\Growpm_r(\hat{\pi}', \Sigma)(\xx;t) = t^{|\Area(\hat{\pi}')|}  \flagGcal^\pm_{r,\nubold}(\xx;t^{-1}).
\end{align}
By considering only the super words with positive letters and super tableaux with positive entries, we also have
\begin{align}
\label{e two flag LLT}
\Grow_r(\hat{\pi}', \Sigma)(\xx;t) =  t^{|\Area(\hat{\pi}')|}  \flagGcal_{r,\nubold}(\xx;t^{-1}).
\end{align}

\begin{figure}
\[
\begin{array}{c@{\hskip 0.5in}c}
\begin{tikzpicture}[scale=.5,baseline=.6cm]
\draw[help lines] (0,0) grid (8,8);
\draw[ultra thin, gray] (0,0) -- (8,8);
\draw[ultra thick]
(0,0) -- (0,1) -- (1,1) -- (1,2) -- (1,3) -- (1,4) -- (2,4) -- (2,5) -- (2,6)
-- (3,6)-- (4,6)-- (5,6)-- (6,6)-- (6,7)--(6,8) -- (7,8)-- (8,8);
\node at (-0.12+.5, .5) {{\tiny $1$}};
\node at (-0.12+1.5, 1.5) {{\tiny $2$}};
\node at (-0.12+1.5, 2.5) {{\tiny $4$}};
\node at (-0.12+1.5, 3.5) {{\tiny $6$}};
\node at (-0.12+2.5, 4.5) {{\tiny $7$}};
\node at (-0.12+2.5, 5.5) {{\tiny $8$}};
\node at (-0.12+6.5, 6.5) {{\tiny $3$}};
\node at (-0.12+6.5, 7.5) {{\tiny $5$}};
\node at (4, -1.2) {\parbox{4.5cm}{\small $\pi$ with north steps
    labeled \\ by sweeping reading order}};
\end{tikzpicture}
&
\begin{tikzpicture}[scale=.5,baseline=.6cm]
\draw[help lines] (0,0) grid (8,8);
\draw[very thin, gray] (0,0) -- (8,8);
\draw[ultra thick] (0,0) -- (0,1) -- (0,2) -- (0,3) -- (1,3) -- (2,3) -- (2,4) -- (3,4) -- (3,5) -- (4,5) -- (4,6) -- (4,7) -- (5,7)-- (6,7)-- (7,7) -- (7,8)-- (8,8);
\node at (4,-1) {{\small $\pi'$ and a marking $\Sigma$}};
\node at (1.5, 3.5) {\footnotesize $\star$};
\node at (3.5, 5.5) {\footnotesize $\star$};
\node at (2.5, 4.5) {\footnotesize $\star$};
\node at (6.5, 7.5) {\footnotesize $\star$};
\end{tikzpicture}
\\
\text{\small(a)} & \text{\small(b)}
\\[12pt]
\begin{tikzpicture}[scale=.37]
\draw [shift={(0,4.5)}] (0,0) grid (2,1);
\node at (0.5,5) {\tiny  $3$};
\node at (1.5,5) {\tiny  $5$};
\node [right] at (5,5) {\tiny $(\beta_{4})/(\alpha _{4}) = (2)/(0)$};

\draw [shift={(0,3)}] (2,0) grid (4,1);
\node at (3.5,3.5) {\tiny  $8$};
\node at (2.5,3.5) {\tiny  $7$};
\node [right] at (5,3.5) {\tiny $(\beta _{3})/(\alpha _{3}) = (4)/(2)$};

\draw [shift={(0,1.5)}] (0,0) grid (3,1);
\node at (.5,2) {\tiny  $2$};
\node at (1.5,2) {\tiny  $4$};
\node at (2.5,2) {\tiny  $6$};
\node [right] at (5,2) {\tiny $(\beta _{2})/(\alpha _{2}) = (3)/(0)$};

\draw [shift={(0,0)}] (0,0) grid (1,1);
\node at (0.5,0.5) {\tiny  $1$};
\node [right] at (5,0.5) {\tiny $(\beta _{1})/(\alpha _{1}) = (1)/(0)$};
\node [right] at (0, -1.2)  {\small $\mubold$, with reading order labeled};
\end{tikzpicture}
&
\begin{tikzpicture}[scale=.37]
\pgfmathsetmacro{\alphabetaleftx}{1}
\draw  [shift={(0,4.5)}] (-1,0) grid (0,1);
\node at (-.5,5) {\tiny  $8$};
\node [right] at (\alphabetaleftx,5) {\tiny $(-\alpha_{1})/(-\beta_{1}) = (0)/(-1)$};

\draw [shift={(0,3)}] (-3,0) grid (0,1);
\node at (-.5,3.5) {\tiny  $7$};
\node at (-1.5,3.5) {\tiny  $5$};
\node at (-2.5,3.5) {\tiny  $3$};
\node [right] at (\alphabetaleftx,3.5) {\tiny $(-\alpha_{2})/(-\beta_{2}) = (0)/(-3)$};

\draw [shift={(0,1.5)}] (-4,0) grid (-2,1);
\node at (-3.5,2) {\tiny  $1$};
\node at (-2.5,2) {\tiny  $2$};
\node [right] at (\alphabetaleftx,2) {\tiny $(-\alpha_{3})/(-\beta_{3}) = (-2)/(-4)$};

\draw [shift={(0,0)}] (-2,0) grid (0,1);
\node at (-0.5,0.5) {\tiny  $6$};
\node at (-1.5,0.5) {\tiny  $4$};
\node [right] at (\alphabetaleftx,0.5) {\tiny $(-\alpha_{4})/(-\beta_{4}) = (0)/(-2)$};

\node [right] at (-3, -1.2)  {\small $\nubold$, with reading order labeled};
\end{tikzpicture}\\
\text{\small(c)} & \text{\small(d)}
\end{array}\]
\caption{\label{fig:LLT versions}
In (a) and (b), $\pi '$ is the attacking Dyck path of $\pi $.
For the given
$\pi '$ and marking $\Sigma $,
(c) gives the tuple $\mubold = ((\beta_1)/(\alpha_1), \dots, (\beta_p)/(\alpha_p))$ from the proof of Lemma \ref{l Dpath LLT}
such that \eqref{e attack0} holds.  Reversing and negating $\beta$ and  $\alpha$ gives the
 tuple $\nubold$ in (d), for which the flagged row LLT polynomial  $\Grow_r(\hat{\pi}', \Sigma)$ equals
$t^{|\Area(\hat{\pi}')|}  \flagGcal_{r,\nubold}(\xx;t^{-1})$.
See \S\ref{ss Dpath vs shapes llt}.}
\end{figure}
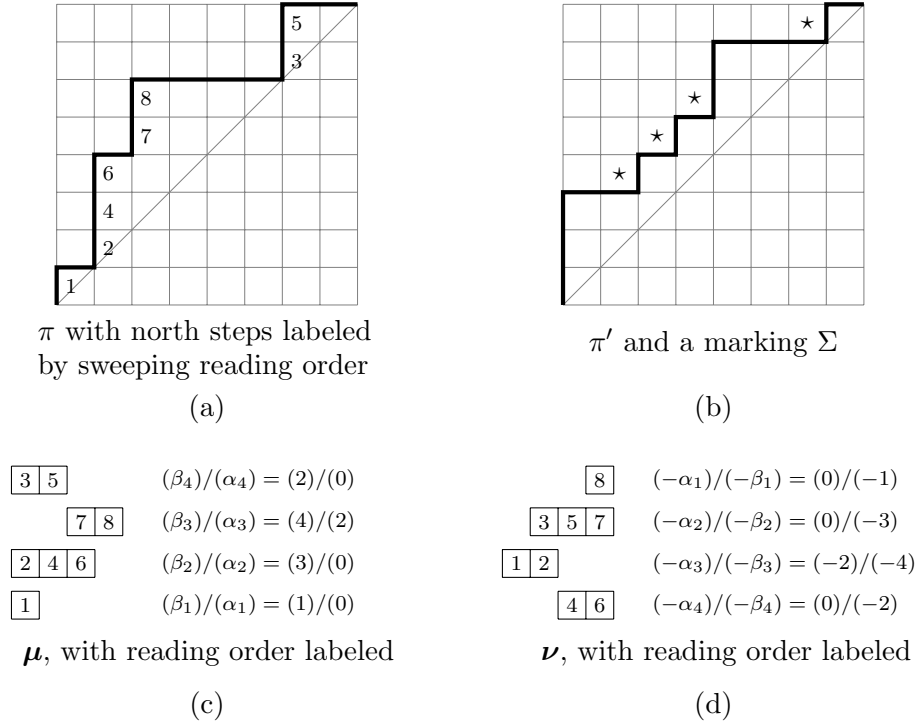

\subsection{Nonsymmetric plethysm and flagged row LLT polynomials}

\begin{thm}[{\cite{BHMPS-nspleth}}]
\label{t signed to unsigned}
The $r$-nonsymmetric plethysm map takes signed flagged LLT polynomials to flagged LLT polynomials:
for $r \in \NN$ and any tuple of skew shapes $\nubold$ with at least $r$ boxes,
\begin{align}
\label{e t signed to unsigned}
\Pisf_r( \flagGcal^\pm_{r,\nubold}(\xx;t^{-1}) ) = \flagGcal_{r,\nubold}(\xx;t^{-1}).
\end{align}
In particular, it takes signed flagged row LLT polynomials to flagged row LLT polynomials:
for $\pi \in \DD_r(d)$ and a marking $\Sigma$ of $\pi$,
\begin{align}
\label{e t signed to unsigned2}
\Pisf_r(\Growpm_r(\pi, \Sigma) ) = \Grow_r(\pi, \Sigma).
\end{align}
\end{thm}
\begin{proof}
The first identity follows from \cite[Lemma 7.6.4]{BHMPS-nspleth}
and Remark \ref{rem:empty rows}.
The second identity is then obtained from \eqref{e two signed flag LLT} and \eqref{e two flag LLT}.
\end{proof}

\begin{example}
\label{e pi on row LLT}
We illustrate \eqref{e t signed to unsigned2} for the partial Dyck path  $\pi \in \DD_2(3)$ shown below and marking  $\Sigma = \{(1,3)\}$ (indicated with the  $\star$).
The flagged $(\pi, \Sigma)$-row super words are shown below with the same conventions as Example \ref{ex row flag llt}.
\newcommand{\thisllt}[4]{
\begin{tikzpicture}[scale=.4,baseline=.3cm]
\draw[help lines] (0,0) grid (3,3);
\draw[thick] (0,2)--(1,2)--(1,3)--(2,3)--(3,3);
\node at (.5, 2.5) {{\tiny $\star$}};
\node at (.5, .5) {\tiny $#1$};
\node at (1.5, 1.5) {\tiny $#2$};
\node at (2.5, 2.5) {\tiny $#3$};
#4
\end{tikzpicture}
}
\[\begin{array}{c|cccccccc}
T \in \FSSYT^{\pm}(\pi, \Sigma)
&
\thisllt{1}{2}{1}{
\node at (1.5,2.5) {\tiny $t$};
}
&
\thisllt{1}{1}{1}{
\node at (.5,1.5) {\tiny $t$};
\node at (1.5,2.5) {\tiny $t$};
}
&
\thisllt{1}{\bar{1}}{1}{
\node at (1.5,2.5) {\tiny $t$};
}
\\[2.8mm]
t^{\inv(T)} &   t & t^2& t  \\[.6mm]
(-t)^{m(T)} &  1 & 1 & -t
\end{array}\]
By Definition \ref{def:LLT signed rows}, we have
\begin{align}
\label{e ex sign LLT}
\Growpm_2(\pi, \Sigma)(\xx;t) = t \spa x_1^2x_2 + t^2 x_1^3 + t (-t) x_1^3  =  t \spa x_1^2 x_2,
\end{align}
while $\Grow_2(\pi, \Sigma)$ is the sum of  $t^{\inv(T)} \xx^T$ over the first two words, giving
\begin{align}
\Grow_2(\pi, \Sigma)(\xx;t) = t^2 x_1^3 + t \spa x_1^2 x_2.
\end{align}
By the definition \eqref{e: nspleth def} of  $\Pisf_r$, we have
\begin{align}
\Pisf_2(t \, x_1^2 x_2) = t \pol( x_1^2x_2 + t \spa x_1^3 + t^2 x_1^4 x_2^{-1} + t^3 x_1^5 x_2^{-2} + \cdots) = t \spa ( x_1^2x_2 + t \, x_1^3),
\end{align}
where to compute the polynomial truncation we have used that all monomials appearing here are Demazure characters.
This verifies \eqref{e t signed to unsigned2} in this case.
\end{example}

We will also need the following lemma which gives another combinatorial description of the signed flagged row LLT polynomials.
We say that a word $T \in \FSSYT^{\pm}(\pi, \Sigma)$
is \emph{non-attacking} if
$|T(i)| \ne |T(j)|$ for every $(i,j) \in \Area(\pi)$.
\begin{lemma}[\cite{BHMPS-nspleth}]
\label{l nonattack}
For any $\pi \in \DD_r(d)$ and marking $\Sigma$ of  $\pi$,
\begin{align}
\label{e l nonattack}
\Growpm_r(\pi, \Sigma)(\xx;t) = \sum_{\substack{ T \in \FSSYT^{\pm}(\pi, \Sigma)\\ T \text{ non-attacking}}} t^{\inv(T)} (-t)^{m(T)} \xx^{|T|}.
\end{align}
\end{lemma}
\begin{proof}
The analogous result for the   $\flagGcal^\pm_{r,\nubold}(\xx;t^{-1})$ follows from \cite[Proposition 6.7.9]{BHMPS-nspleth},
and then \eqref{e two signed flag LLT} gives \eqref{e l nonattack}.
 \end{proof}

\begin{example}
In Example \ref{e pi on row LLT}, only the first word is non-attacking.  As we saw in \eqref{e ex sign LLT}, the terms for the other two words cancel, illustrating Lemma \ref{l nonattack}.
\end{example}

\subsection{Signed flagged row LLT polynomials and the Dyck path algebra}
\label{s:dpa2sllt}

\begin{defn}
\label{d chi dplusminus}
For $\pi \in \DD_r(d)$ and a marking $\Sigma$ of $\pi$,
let $\chi^\PP_r(\pi, \Sigma)$ be the result of applying the following
sequence of operators to $1 \in \P(0)$, determined
by traveling along  $\pi$ and  $\Sigma$
 starting from the northeast endpoint of $\pi $:
\begin{itemize}
\item[(i)]  $\frac{1}{t-1}[d_-^\PP, -d_+^\PP]$ for each marked corner in  $\Sigma$,
\item[(ii)]  $d_-^\PP$ for each vertical step not part of a marked corner,
\item[(iii)]  $-d_+^\PP$ for each horizontal step not part of a marked corner.
\end{itemize}
\end{defn}

\begin{example}
For the partial Dyck path  $\pi'$ and marking  $\Sigma $ in Example \ref{ex:partial Dyck path},
\begin{align}
\chi^\PP_2(\pi', \Sigma) = \frac{(-1)^6}{(t-1)^2}\, d_-^\PP \, [d_-^\PP, d_+^\PP] \, d_+^\PP \, d_+^\PP \, d_-^\PP \, [d_-^\PP, d_+^\PP] \, d_+^\PP \, d_+^\PP \cdot 1
\end{align}
\end{example}

Our next result provides a combinatorial expression for $\chi_r^P(\pi,\Sigma)$.  Our proof builds on a fundamental theorem from
\cite{CarMel}, which we now recall.

The objects  $\chi_r(\pi)$ and  $\chi'_r(\pi)$ defined in \cite[eq.~(4.5)]{CarMel} and after \cite[[eq.~(4.3)]{CarMel} are readily written in our notation as
\begin{align}
\label{e chiprime}
\chi_r(\pi) = (t - 1)^{d - r} \, \Q_r \PP_r(\chi'_r(\pi))\qquad\text{for}\qquad
\chi'_r(\pi) =
\sum_{\substack{T \in \FSSYT(\pi, \varnothing) \\ T \text{ non-attacking} }}
t^{\inv(T)} \xx^T\,.
\end{align}
Note that
our definition of  $\inv$ differs from that of \cite[\S2.4]{CarMel} for equal letters but this does not arise here because of the non-attacking condition.
The proof of \cite[Theorem 4.4]{CarMel} gives the following result.

\begin{thm}[{\cite{CarMel}}]
\label{t CM LLT}
For  $\pi \in \DD_r(d)$, $\chi_r(\pi)$ is
the result of applying the same sequence of operators to $1$ as in
Definition~\ref{d chi dplusminus} with $\Sigma = \varnothing$, except with $d_-^{CM}, d_+^{CM}$ in place
of  $d_-^\PP, -d_+^\PP$ in (ii) and (iii).
\end{thm}

We now prove an identity for $\chi_r^P(\pi,\Sigma)$, with any marking $\Sigma$, framed in terms of signed rather than non-attacking fillings.
This formulation will allow us to apply the powerful tool of nonsymmetric plethysm
to transform these objects to ones on the modified side (i.e. the right side of Figure \ref{fig:summary diagram}),
where there is visible potential for positivity and connections to representation theory.

\begin{thm}
\label{t Gcal vs chiPP}
For $\pi \in \DD_r(d)$ and a marking $\Sigma$ of  $\pi$,
\begin{align}
\Growpm_r(\pi, \Sigma) = \chi^\PP_r(\pi, \Sigma).
\end{align}
\end{thm}

\begin{proof}
We first establish the case  $\Sigma = \varnothing$.
By Lemma \ref{l nonattack},
\begin{align}
\label{e chiprime2}
\Growpm_r(\pi, \varnothing)
= \sum_{\substack{T \in \FSSYT^{\pm}(\pi, \varnothing) \\ T \text{ non-attacking} }}
t^{\inv(T)}
(-t)^{m(T)} \xx^{|T|}.
\end{align}
For  $T$ in this sum, the non-attacking and flagging conditions force  $T(1)=1,\, \dots,\, T(r)=r$.  Moreover, because of the non-attacking condition,
the set of words in \eqref{e chiprime2} can be built from the set in \eqref{e chiprime} by freely choosing the signs on $T(r+1), \dots, T(d)$.  It follows that
\begin{align}
\label{e Gcal chiprime}
\Growpm_r(\pi, \varnothing) = (1-t)^{d-r}\chi'_r(\pi) \,.
\end{align}
Combining this with \eqref{e chiprime} gives
\begin{align}
\label{e Gcal chiprime2}
\Growpm_r(\pi, \varnothing) = (-1)^{d-r} \spa \Q_r^{-1} \PP_r^{-1}(\chi_r(\pi)).
\end{align}
On the other hand, Theorem \ref{t CM LLT} and \eqref{e CarMel vs Mel 2}--\eqref{e CarMel vs Mel 3} give
\begin{align}
\label{e chi chiPP}
\chi^\PP_r(\pi, \varnothing) = (-1)^{d-r} \spa \Q_r^{-1}\PP_r^{-1} (\chi_r(\pi)).
\end{align}
Combining the last two equations proves the desired identity $\Growpm_r(\pi, \varnothing) = \chi^\PP_r(\pi, \varnothing)$.

We now address the case $\Sigma \ne \varnothing$, using an argument similar to the one going from Theorem 4.4 to 4.6 in \cite{CarMel}.
Consider a fixed $(i,j)\in\Sigma$.
Let \(\pi^* \in \DD_r(d)\) be the path obtained by flipping the corner  $(i,j)$ inside out, i.e., by replacing the  $\mathsf{E}\mathsf{N}$ steps bordering the box  $(i,j)$ to  $\mathsf{N}\mathsf{E}$,
so that box  $(i,j)$ now lies below  $\pi^*$.
Set \(\Sigma^* = \Sigma \setminus \{(i,j)\}\).
Directly from Definition \ref{d chi dplusminus}, we have
\begin{equation}
\label{ep chiPP}
  (t-1)\chi^\PP_r(\pi,\Sigma) =  \chi^\PP_r(\pi^*, \Sigma^*)-\chi^\PP_r(\pi, \Sigma^*) \,.
\end{equation}
Repeatedly applying identities of this form over the different elements of  $\Sigma$ allows  $\chi^\PP_r(\pi,\Sigma)$ to be
written as an alternating sum of  $2^{|\Sigma|}$ terms from the $\Sigma = \varnothing$ case.
Thus, it suffices to show that \(\Growpm_r(\pi,\Sigma)\) satisfies the same recursion.

First note that \(\FSSYT^{\pm}(\pi,\Sigma^*)\) and \(\FSSYT^{\pm}(\pi^*,\Sigma^*)\) are equal as sets, though inversions are counted differently; let  $S$ denote either of these sets.

Say
a word
$T \in S$
is \emph{row-like} (resp. \emph{column-like}) if  $T(j)T(i)$ is row-increasing (resp. column-decreasing).
For column-like $T \in S$, \((i,j)\) is never an attacking inversion,
hence the contributions from the column-like words in
\(\FSSYT^{\pm}(\pi,\Sigma^*)\) and \(\FSSYT^{\pm}(\pi^*,\Sigma^*)\) cancel in
\begin{equation}\label{eq:llt difference}
\Growpm_r(\pi^*,\Sigma^*)-\Growpm_r(\pi,\Sigma^*) \,.
\end{equation}
On the other hand, since \((i,j) \notin
\Area(\pi)\), for row-like  $T\in S$,
 $(i,j)$ is never an attacking inversion of $T$ regarded as an element of
\(\FSSYT^{\pm}(\pi,\Sigma^*)\) but always one when  $T$ is regarded as an element of
\(\FSSYT^{\pm}(\pi^*,\Sigma^*)\).
Thus, the contribution of the row-like words of
\(\FSSYT^{\pm}(\pi^*,\Sigma^*)\)
to $\Growpm_r(\pi^*,\Sigma^*)$ is  $t$ times
the contribution of the row-like words of
\(\FSSYT^{\pm}(\pi,\Sigma^*)\)
to $\Growpm_r(\pi,\Sigma^*)$, and
it follows that \eqref{eq:llt difference} is equal to
\((t-1)\Growpm_r(\pi,\Sigma)\). Hence, \(\Growpm_r(\pi,\Sigma)\) satisfies the same recursion as  $\chi^\PP_r(\pi, \Sigma)$, completing the proof.
\end{proof}

\begin{remark}
The non-attacking description of $\Growpm_r(\pi, \Sigma)(\xx;t)$ is particularly nice in the unicellular case ($\Sigma = \varnothing$) because it reduces to a sum over positive fillings times a power of  $(1-t)$.
Outside this case, we have not found a similar simplification for the non-attacking description, and we instead believe that the signed filling description is more natural.
\end{remark}

\subsection{(Signed) flagged column LLT polynomials}
\label{ss flagged LLT polynomials for columns}

We will also need a variant of the results above which correspond to the flagged LLT polynomials of
\cite{BHMPS-nspleth} for a tuple of single-column skew shapes.

A {\em  $(\pi, \Sigma)$-column semistandard word} is a map
$T\colon [d] \rightarrow \ZZ_+$ such that for every  $(i,j) \in \Sigma$,
$T(i) < T(j)$.
More generally, a {\em  $(\pi, \Sigma)$-column super word} is a map
$T\colon [d] \rightarrow \Acal$ such that for every  $(i,j) \in \Sigma$,
 $T(j)T(i)$ is column-decreasing.
We say $T$ is \emph{flagged} if  $T(1) \le 1, \dots, \, T(r) \le r$.
Denote the set of
flagged $(\pi, \Sigma)$-column semistandard words
(resp. flagged $(\pi, \Sigma)$-column super words) by
$\FCSSYT(\pi, \Sigma)$ (resp. $\FCSSYTpm(\pi, \Sigma)$).
Define
\begin{align}
\widetilde{\inv}(T) = \# \big\{(i,j) \in \Area(\pi) : \text{$T(j)T(i)$ is column-decreasing}\big\}.
\end{align}

\begin{defn}
\label{d col LLT}
For $\pi \in \DD_r(d)$ and a marking $\Sigma$ of $\pi$,
the associated \emph{flagged column LLT polynomial} is given by
\begin{align}
\Gcol_r(\pi, \Sigma)(\xx;t) =
\sum_{T \in \FCSSYT(\pi, \Sigma) } t^{\widetilde{\inv}(T)} \xx^T,
\end{align}
and the \emph{signed flagged column LLT polynomial} by
\begin{align}
\Gcolpm_r(\pi, \Sigma)(\xx;t) = \sum_{T \in \FCSSYTpm(\pi, \Sigma)} t^{\widetilde{\inv}(T)} (-t)^{-m(T)} \xx^{|T|}.
\end{align}
\end{defn}

\begin{remark}
\label{r col two LLT match}
Similar to \S\ref{ss Dpath vs shapes llt}, the (signed) flagged column LLT polynomials are special cases of those in \cite{BHMPS-nspleth}.
Let $\hat{\pi}' \in \DD_r(d)$ and $\Sigma$ be a marking of  $\hat{\pi}'$.
As discussed in \S\ref{ss Dpath vs shapes llt},
there exists a tuple of single-row skew shapes $\nubold = (\nu^{(1)}, \dots, \nu^{(d)})$ so that \eqref{e attack} holds.
Now let 
$\etabold = (\eta^{(1)}, \dots, \eta^{(d)})$
be the tuple of single-column skew shapes
such that the contents occupied by the boxes of $\eta^{(i)}$ are the same as those of  $\nu^{(i)}$.
We have
\begin{align}
\label{e col two LLT match}
\Gcol_r(\hat{\pi}', \Sigma)(\xx;t) & = \flagGcal_{r, \etabold}(\xx;t), \\
\label{e col two signed LLT match}
\Gcolpm_r(\hat{\pi}', \Sigma)(\xx;t) & = \flagGcal^\pm_{r, \etabold}(\xx;t).
\end{align}
Then \eqref{e t signed to unsigned} implies
\begin{align}
\label{e signed to unsigned cols}
\Pisf_r \big(\Gcolpm_r(\hat{\pi}', \Sigma)(\xx;t^{-1}) \big) = \Gcol_r(\hat{\pi}', \Sigma)(\xx;t^{-1}).
\end{align}
\end{remark}

We have the following variant of  $\chi^\PP_r(\pi, \Sigma)$ for the signed flagged column LLT polynomials.
\begin{defn}
\label{d chi dplusminus col}
For $\pi \in \DD_r(d)$ and a marking $\Sigma$ of $\pi$,
let $\tilde{\chi}^\PP_r(\pi, \Sigma)$ be the result of applying the following
sequence of operators to $1 \in \P(0)$, determined
by traveling along  $\pi$ and  $\Sigma$
 starting from the northeast endpoint of $\pi $:
\begin{itemize}
\item[(i)]  $\frac{1}{t^{-1}-1}\big(d_-^\PP (-t^{-r} d_+^\PP) - (-t^{-r+1} d_+^\PP) d_-^\PP \big) $ for each marked corner in  $\Sigma$,
\item[(ii)]  $d_-^\PP$ for each vertical step not part of a marked corner,
\item[(iii)]  $-t^{-r} d_+^\PP$ for each horizontal step not part of a marked corner,
\end{itemize}
where  $r = h-v$ when we have so far visited $h$ horizontal steps and  $v$ vertical steps.
\end{defn}

\begin{thm}
\label{t Gcal vs chiPP v2}
For $\pi \in \DD_r(d)$ and a marking $\Sigma$ of $\pi$,
\begin{align}
\label{e Gcal vs chiPP v2}
\Gcolpm_r(\pi, \Sigma)(x;t^{-1}) = \tilde{\chi}^\PP_r(\pi, \Sigma)(x;t).
\end{align}
\end{thm}
\begin{proof}
First we address the case  $\Sigma = \varnothing$.
Then (i) never arises in Definition \ref{d chi dplusminus col} and we have
 $t^{|\Area(\pi)|}\tilde{\chi}^\PP_r(\pi, \varnothing) = \chi^\PP_r(\pi, \varnothing)$.
Hence, by Theorem \ref{t Gcal vs chiPP}, \eqref{e two signed flag LLT}, and \eqref{e col two signed LLT match}, we have
\begin{align}
t^{|\Area(\pi)|}\tilde{\chi}^\PP_r(\pi, \varnothing) = \chi^\PP_r(\pi, \varnothing) =
\Growpm_r(\pi, \varnothing)(x;t) = t^{|\Area(\pi)|}\Gcolpm_r(\pi, \varnothing)(x;t^{-1}),
\end{align}
thus establishing \eqref{e Gcal vs chiPP v2} for $\Sigma = \varnothing$.
The case  $\Sigma \ne \varnothing$ now follows by an argument similar
to the analogous step in the proof of Theorem \ref{t Gcal vs chiPP}. 
\end{proof}

\subsection{Symmetrization}
\label{ss Symmetrization}
Here we relate flagged row and column LLT polynomials to ordinary symmetric LLT polynomials.

Let $\Gcal_{\nubold}(\xx;t)$ denote the ordinary symmetric LLT polynomial in the sense of
\cite{HaHaLoReUl05} indexed by the tuple of skew shapes $\nubold$ and defined in terms attacking inversions.
The definition of these is exactly the  $r=0$ special case of the $\flagGcal_{r,\nubold}(\xx;t)$ defined in \eqref{e nspleth paper Gs0}, i.e.,
$\flagGcal_{0,\nubold}(\xx;t) = \Gcal_{\nubold}(\xx;t)$.

\begin{remark}
\label{r flagLLT r0}
Let $\hat{\pi}'$, $\Sigma$, $\nubold$, and  $\etabold$ be as in
Remark \ref{r col two LLT match}, now with  $r = 0$ (so  $\hat{\pi}' \in \DD_0(d)$ and we set  $\pi' = \hat{\pi}'$ as this is consistent with the notation in \S\ref{ss Dpath vs shapes llt}).

(i) The flagged row LLT polynomials are symmetric LLT polynomials when  $r=0$.
Specifically,
\eqref{e two flag LLT} followed by \cite[Proposition 4.1.6]{BHMPS-paths} yields
\begin{align}
\Grow_0(\pi', \Sigma)(\xx;t) = t^{|\Area(\pi')|} \Gcal_{\nubold}(\xx;t^{-1}) = \omega \, \Gcal_{\etabold}(\xx;t).
\end{align}

(ii) Similarly, flagged column LLT polynomials are symmetric LLT polynomials: by \eqref{e col two LLT match},
\begin{align}
\label{e col LLTs sym}
\Gcol_0(\pi', \Sigma)(\xx;t) = \Gcal_{\etabold}(\xx;t).
\end{align}

(iii)
The symmetric LLT polynomials in \eqref{e col LLTs sym} also match
the version of symmetric LLT polynomials $\chi(\pi', \Sigma)$ defined in \cite[\S3.5]{Mellit16};
the definition of $\chi(\pi', \Sigma)$ is the same as  $\Gcol_0(\pi', \Sigma)(\xx;t)$
except that
in $\chi(\pi', \Sigma)$
the condition
$T(i) < T(j)$ for $(i,j) \in \Sigma$ is replaced with $T(i) > T(j)$, and the
attacking inversion condition
$T(i) < T(j)$ for $(i,j) \in \Area(\pi')$ is replaced with  $T(i) > T(j)$.
But since  $\chi(\pi', \Sigma)$ is symmetric, reversing the order of the variables does not change the function, so
we indeed have
\begin{align}
\label{e two llts agree}
\Gcol_0(\pi', \Sigma)(\xx;t) = \Gcal_{\etabold}(\xx;t) = \chi(\pi', \Sigma).
\end{align}
\end{remark}

\begin{prop}
\label{p sym flag LLT}
(i) (Signed) flagged row LLT polynomials match up as follows under Hecke and Weyl symmetrization:
for $\theta \in \DD_r(d)$ and a marking $\Sigma$ of  $\theta$,
\begin{equation}
\label{ep sym LLT1}
\begin{tikzcd}
\Growpm_r(\theta, \Sigma)
\ar[mapsto, d,"\Pisf_r"] \ar[mapsto,r,"\hsym_r"]
 &\Growpm_{r-1}(\mathsf{N} \theta, \Sigma)
 \ar[mapsto,d,"\Pisf_{r-1}"]
  \\
\Grow_r(\theta, \Sigma)
\ar[mapsto,r,"\Weyl_r"]&
 \Grow_{r-1}(\mathsf{N} \theta, \Sigma)
\end{tikzcd}
\end{equation}
where  $\mathsf{N} \theta \in \DD_{r-1}(d)$ is the result of adding a north step to the southwest end of $\theta$.
In particular, in the notation of Remark \ref{r flagLLT r0} with $\pi' = \mathsf{N}^r \theta$,
\begin{align}
\label{ep sym flag LLT}
\Weyl_1 \cdots \Weyl_r (\Grow_r(\theta, \Sigma)) = \Grow_0(\mathsf{N}^{r} \theta, \Sigma)
= \omega \, \Gcal_{\etabold}(\xx;t).
\end{align}

(ii) Similarly, (signed) flagged column LLT polynomials match up as follows under Hecke and Weyl symmetrization:
\begin{equation}
\label{ep sym LLT2}
\begin{tikzcd}
\Gcolpm_r(\theta, \Sigma)(\xx;t^{-1})
\ar[maps to, d,"\Pisf_r"] \ar[mapsto,r,"\hsym_r"]
 &\Gcolpm_{r-1}(\mathsf{N} \theta, \Sigma)(\xx;t^{-1})
 \ar[mapsto,d,"\Pisf_{r-1}"]
  \\
\Gcol_r(\theta, \Sigma)(\xx;t^{-1})
\ar[mapsto,r,"\Weyl_r"]&
 \Gcol_{r-1}(\mathsf{N} \theta, \Sigma)(\xx;t^{-1})
\end{tikzcd}
\end{equation}
In particular,
\begin{align}
\label{ep sym LLT col 2}
\Weyl_1 \cdots \Weyl_r (\Gcol_r(\theta, \Sigma)) =
\Gcol_0(\mathsf{N}^r \theta, \Sigma) =
\Gcal_{\etabold}(\xx;t).
\end{align}
\end{prop}
\begin{proof}
The  $\hsym_r$ arrow of \eqref{ep sym LLT1} is immediate from Theorem \ref{t Gcal vs chiPP} and Definition \ref{d chi dplusminus}
(keeping in mind Remark \ref{rem: hsym and d-}).  The  $\Weyl_r$ arrow then follows from Proposition \ref{p dminus and sym}.
Statement (ii) follows similarly using Theorem \ref{t Gcal vs chiPP v2} in place of Theorem \ref{t Gcal vs chiPP}.
\end{proof}

\subsection{Nonsymmetric chromatic polynomials and flagged unicellular LLT polynomials}

Here we briefly mention a connection with related work, which will not be used elsewhere in the paper.
By \cite[Propositions 3.4--3.5]{CarMel}, chromatic symmetric functions of incomparability graphs of natural unit interval orders
and unicellular LLT polynomials are related by the plethystic
transformation \(f[\xx] \mapsto f[\xx/(1-t)]\).
 We will now see that a similar relation holds between our
flagged row LLT polynomials  $\Grow_r(\pi, \Sigma)(\xx ;t)$ in the unicellular case (i.e. $\Sigma = \varnothing$)
and nonsymmetric versions of
chromatic symmetric functions introduced by Haglund-Wilson \cite{HaglundWilson} and further studied by
Tewari-Wilson-Zhang \cite{TWZchromatic}.

The functions  $\boldsymbol{\chi}_{\mathbf{F}, \pi}$ from \cite{TWZchromatic}, which depend on  $d, r, N \in \ZZ_+$, flag bounds $\mathbf{F}= (0 \le F_1 \le \cdots \le F_r \le N)$, and $\pi \in \DD_r(d)$, are readily written in the notation of this section:
\begin{align}
\boldsymbol{\chi}_{\mathbf{F}, \pi} =
\sum_{T \text{ non-attacking}}
t^{\inv(T)} \xx^T,
\end{align}
where the sum is just as in \eqref{e chiprime} except that now the letters of  $T$ must be  $\le N$ and
the flagging condition  $T(i) \le i$ for  $i\in [r]$ is replaced by the more general $T(i) \le F_i$ for $i \in [r]$;  $\inv(T)$ is also just as in \eqref{e chiprime},
as defined in \S \ref{ss flagged LLT polynomials}.
Hence, in particular, by \eqref{e Gcal chiprime}, these functions for flag bounds  $(1,\dots,r)$ are essentially the same as our signed flagged row LLT polynomials up to a power of  $(1-t)$:
\begin{align}
\boldsymbol{\chi}_{(1,\dots, r), \pi} = \chi'_r(\pi)[\xx_N ;t] = (1-t)^{r-d}\Growpm_r(\pi, \varnothing)[\xx_N ;t],
\end{align}
where we have used the notation $f[\xx_N;t] = f(x_1,\dots, x_N, 0,\dots;t)$.

It was initially conjectured the  $\boldsymbol{\chi}_{\mathbf{F}, \pi}$ were Demazure character positive but
this turned out to be false.  Our computations support the conjecture that they are Demazure character positive in the case
$\mathbf{F} = (1,\dots, r)$.
However, even supposing this is true, 
this situation differs in an
important way from our atom positivity conjectures on unsigned flagged row LLT polynomials and related polynomials
(see \S\ref{ss atom pos 2}):
we showed in \S\ref{ss Symmetrization} that $\Grow_r(\pi, \Sigma)$ Weyl symmetrizes to the corresponding symmetric LLT polynomial
$\Grow_0(\mathsf{N}^r \pi, \Sigma)$
and therefore atom positivity of $\Grow_r(\pi, \Sigma)[\xx_N;t]$ implies the Schur positivity of $\Grow_0(\mathsf{N}^r \pi, \Sigma)[\xx_N;t]$ via \eqref{e atom key pos}.
In contrast, the $\boldsymbol{\chi}_{\mathbf{F}, \pi}$ \emph{do
not have the property that they Weyl symmetrize to the corresponding
chromatic symmetric function}, i.e., $\pi_{w_0} (\boldsymbol{\chi}_{\mathbf{F}, \pi})$ is typically not equal to $\boldsymbol{\chi}_{(N,\dots, N), \pi}$.
Therefore, Demazure character or atom positivity of $\boldsymbol{\chi}_{\mathbf{F}, \pi}$ does
not imply the Schur positivity of $\boldsymbol{\chi}_{(N,\dots, N), \pi}$
via \eqref{e atom key pos}.

\section{The nonsymmetric compositional  \texorpdfstring{$(km,kn)$}{(km,kn)}-shuffle theorem}
\label{s ns comp kmkn thm}
Here we prove a nonsymmetric compositional $(km,kn)$-shuffle theorem.
Our approach is to first establish a signed version of it and to then apply the $r$-nonsymmetric plethsym map $\Pisf_r$.
The former is proven by interpreting the last step in Mellit's proof of the (symmetric) compositional $(km,kn)$-shuffle theorem
in terms of the signed flagged row LLT polynomials developed in \S\ref{s LLT}.

\subsection{Double Dyck path algebra endomorphisms}
The algebraic side of the nonsymmetric compositional $(km,kn)$-shuffle theorem involves
ingredients from Mellit's version.
In particular, it involves endomorphisms  $N$ and  $S$ of  $\AA_{t,q}$,
which are variants of well-known automorphisms of the DAHA.

\begin{prop}[{\cite[Proposition 3.16]{Mellit16}}]
There is a unique  $\QQ(q,t)$-algebra endomorphism
$N \colon \AA_{t,q} \to \AA_{t, q}$ determined by
\begin{align}
\label{e N map}
e_r \mapsto e_r, \ \ T_i \mapsto T_i, \ \ d_-  \mapsto d_- , \ \ d_+^* \mapsto d_+^*,
 \ \ d_+ \mapsto -(qt)^{-1} z_1 d_+,
\end{align}
and a unique  $\QQ(q,t)$-algebra endomorphism
$S \colon \AA_{t,q} \to \AA_{t, q}$ determined by
\begin{align}
e_r \mapsto e_r, \ \ T_i \mapsto T_i, \ \ d_-  \mapsto d_- , \ \ d_+ \mapsto d_+,
 \ \ d_+^* \mapsto - y_1 d_+^*.
\end{align}
\end{prop}

Define the following elements of  $\SL_2(\ZZ)$
\begin{align}
N' =
\begin{pmatrix}
1 & 1 \\
0 & 1
\end{pmatrix}, \quad
S' =
\begin{pmatrix}
1 & 0 \\
1 & 1
\end{pmatrix}.
\end{align}

\begin{propdef}[{\cite[Theorem 3.19]{Mellit16}}]
\label{pd rho}
Let  $m \in \ZZ_+$,  $n \in \ZZ_{\ge 0}$ with  $\gcd(m,n)=1$.
Let $A'$ be a word in the alphabet $\{N',S'\}$ such that the product of these matrices has the form
$\begin{pmatrix}
m & * \\
n & *
\end{pmatrix}$ and let
$A$ be the corresponding composition of the endomorphisms  $N, S$.
Such a word exists and can be constructed using Farey sequences (see \cite[Remark 3.20]{Mellit16}).
Let  $\rho^*_{m,n} \colon \AA_t \to \AA_{t,q}$ be the composition  $A \circ \rho^*$, where
 $\rho^* \colon \AA_t \to \AA_{t,q}$ is as in Remark \ref{rem:why ddpa}.
This 
is independent of the choice of word  $A'$.
\end{propdef}

\begin{example}
For  $(m,n)= (1,1)$, we could choose  $A' = S' N'^k$ for any  $k\ge 0$ and
\begin{align}
  \rho^*_{1,1} = S \circ \rho^*.
\end{align}
As another example, for any  $a,b \ge 0$,
\begin{align}
\rho^*_{ab+1,b} = N^a S^b \circ \rho^*.
\end{align}
\end{example}

\subsection{Signed, nonsymmetric compositional \texorpdfstring{$(km,kn)$}{(km,kn)}-shuffle theorem}
\label{ss Nonsymmetric compositional kmkn shuffle theorem}
Let  $m,n$ be relatively prime positive integers and
$\alpha = (\alpha_1, \dots, \alpha_r) \in \ZZ_+^r$ be a strict composition of size $k$.
Let  $s = \frac{n}{m}$ and  $s_- = s - \epsilon$ for small  $\epsilon > 0$.
Set
\begin{align}
\label{e def alg side}
\alg_{m,n}^\alpha \defeq (-1)^{k(m+n+1)} t^{r-k}
\PP_r^{-1} \rho^*_{m,n}\big( y_1^{\alpha_1-1}\cdots y_r^{\alpha_r-1} d_+^r  \big) \cdot 1 \ \in \P(r).
\end{align}
For comparison, the expression $(-1)^{kn} d_-^r \PP_r \alg_{m,n}^\alpha \in V_0 = \Lambda(W)$
is the form of the algebraic side of the compositional $(km,kn)$-shuffle theorem that Mellit uses in his proof (equivalent to the original by \cite[\S3.8]{Mellit16}).

The signed, nonsymmetric compositional $(km,kn)$-shuffle theorem we prove below
expresses \eqref{e def alg side} as a sum of signed flagged row LLT polynomials.
To describe the combinatorial side, we need to expand on the notation on Dyck paths from \S\ref{ss Dyck path stats}.

Let $\DD^\alpha_{km,kn}$ denote the set of  $(km,kn)$-Dyck paths which touch the diagonal
exactly at the points
$(k_{j} m, k_{j} n)$, where $k_{j} = \sum _{i < j}\alpha_{i}$ for $j=1,\ldots, r+1$.

For any $\pi \in \DD^\alpha_{km,kn}$, the $r$ north steps of  $\pi$ touching the diagonal are the first  $r$ north steps in sweeping reading order and are mutually attacking.  This means that the associated attacking Dyck path $\pi'$
begins with  $r$ north steps, and we let $\hat{\pi}'\in \DD_r(kn)$ denote the partial Dyck path obtained by removing these steps.
Define a distinguished marking of $\pi'$, denoted
$\Sigma_\pi$, to be the set of pairs of consecutive north steps along vertical runs of  $\pi$.
Since the corners of  $\hat{\pi}'$ are the same as
those of  $\pi'$, the marking
$\Sigma_\pi$ of  $\pi'$ may also be regarded
as a marking of  $\hat{\pi}'$.

Following \cite{Mellit16},
for $\pi \in \DD^\alpha_{km,kn}$,
define  $\dinv(\pi)$ to be the number of pairs  $(e,f)$ where  $e$ is an east step of  $\pi$
and $f$ is a north step of  $\pi$, such that  $e$ is to the left of  $f$ and there is a
line of slope  $s_-$ which intersects  $e$ and  $f$.
Also define
\begin{align}
\label{e d maxtdinv}
\maxtdinv(\pi) \ = \ \big(\text{\# of attacking pairs of north steps of $\pi$}\big) \  = \ \area(\pi').
\end{align}

\begin{thm}
\label{t ns shuffle unmod}
For relatively prime positive integers $m,n$ and
a strict composition  $\alpha = (\alpha_1, \dots, \alpha_r)$ of size $k$,
\begin{align}
\label{e ns shuffle unmod}
\alg_{m,n}^\alpha
=
\sum_{\pi \in \DD^\alpha_{km,kn} }  q^{\area(\pi)} t^{\dinv(\pi) - \maxtdinv(\pi) }
\Growpm_r(\hat{\pi}', \Sigma_\pi)\,,
\end{align}
where $\area(\pi)$ and $\Growpm_r(\hat{\pi}', \Sigma_\pi)$ are defined in \S\ref{ss Dyck path stats} and \S\ref{ss flagged LLT polynomials}.
\end{thm}
\begin{proof}
For $\pi \in \DD^\alpha_{km,kn}$, define
$\chi_r(\hat{\pi}', \Sigma_\pi) \in V_r$ just as $\chi^\PP_r(\hat{\pi}', \Sigma_\pi)$ in
Definition \ref{d chi dplusminus} but with $d_+, d_-$ in place of  $-d_+^\PP, d_-^\PP$.
Thus, by Theorem \ref{t Gcal vs chiPP},
\begin{equation}
\Growpm_r(\hat{\pi}', \Sigma_\pi)
= \chi^\PP_r(\hat{\pi}', \Sigma_\pi)
= (-1)^{kn} \PP_r^{-1} \chi_r(\hat{\pi}', \Sigma_\pi).
\end{equation}
Hence \eqref{e ns shuffle unmod}
is equivalent to the following identity of~\cite{Mellit16}:
\begin{multline}
\label{e ns shuffle unmod 2}
(-1)^{k(m+1)}\, t^{r-k}\,
\rho^*_{m,n}\big( y_1^{\alpha_1-1}\cdots y_r^{\alpha_r-1} d_+^r  \big) \cdot 1  \\
= \sum_{\pi \in \DD^\alpha_{km,kn} }  q^{\area(\pi)} t^{\dinv(\pi) - \maxtdinv(\pi) } \chi_r(\hat{\pi}', \Sigma_\pi).
\end{multline}
To explain how this follows from \cite{Mellit16}, we first give an overview of Mellit's proof of the $(km,kn)$-compositional shuffle conjecture, and then
focus on the specific details needed for \eqref{e ns shuffle unmod 2}.

For  $h\in \RR$, let $\ell_h$ be the line
$\{(x,y) : y = s_- x + h\}$ of slope $s_-$.
Following~\cite[\S4.1]{Mellit16}, for \(\pi \in \DD_{km,kn}^\alpha\) and \(h\in \RR\),
set \(D_\pi^h = O_{p_j} \cdots O_{p_1} \cdot 1\) for certain operators \(O_{p_i} \in \AA_{t,q}\) indexed by the lattice points $p_1,\dots, p_j$ weakly below
\(\pi\) and above \(\ell_h\),
ordered by their distance to $\ell_{-\epsilon}$.
We do not need the full description of the $O_{p_i}$, but to get a
flavor, a few examples of \(D_\pi^h\) are given below with \(\ell_h\)
depicted as a dashed line.
   \newcommand{\thissweep}[4]
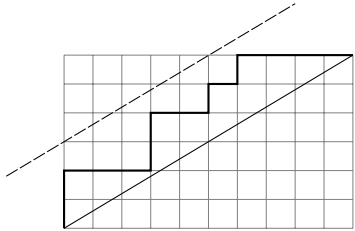
{
     \tikz[scale=.384, baseline=.3cm]{
       \pgfmathsetmacro{\yint}{#1};
       \pgfmathsetmacro{\xstart}{#2};
       \pgfmathsetmacro{\xend}{#3};
       \pgfmathsetmacro{\epsl}{#4}
       \pgfmathsetmacro{\smin}{6/10-\epsl};
       \draw[help lines] (0,0) grid (10,6);
       \draw[thick] (0,0) -- (0,2) -- (3,2) -- (3,4) -- (5,4) -- (5,5)
       -- (6,5) -- (6,6) -- (10,6);
       \draw (0,0) -- (10,6);
       \draw[dash pattern=on 5pt off 1pt] (\xstart,\smin*\xstart+\yint) -- (\xend,{\smin*\xend+\yint});
       \node at (10-\xstart,\smin*\xstart+\yint) {\phantom{$\cdot$}};
     }
   }
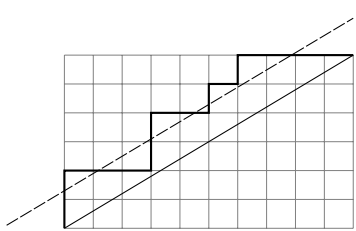
\vspace{-0.5cm}
\begin{center}
 \begin{tikzpicture}[ampersand replacement=\&]
    \matrix[matrix of math nodes, column sep={-0.2cm,0cm}, nodes={anchor=base}] (m)  { 
     \thissweep{3}{-2}{8}{0.0021} \&
     \thissweep{1.3}{-2}{10}{0.0021} \&
     \thissweep{-0.021}{-1}{11}{0.02} \\
     D_\pi^{3} = 1 \&
     D_\pi^{1.3} = t^4 d_- d_+^4(1)\&
     D_\pi^{-\epsilon} = q^6 t^3  d_-^2 \frac{[d_-,d_+]}{q-1} d_-
     \frac{[d_-,d_+]}{q-1} d_- d_+^4(1)
     \\
     };
  \end{tikzpicture}
\end{center}
In general, $O_{p_i}$ is a scalar multiple of  $\idelm$, $d_+$, $d_-$, or $[d_-,d_+]$, determined by the relative position of $\pi$ and $p_i$.

An \emph{admissible coloring of  $\ell_h$} is a subset of line segments of  $\ell_h$, each of which has its left endpoint on a vertical lattice line and its right endpoint on a horizontal lattice line
in the rectangle  $[0,km] \times [0,kn]$.
As in \cite[pg. 27]{Mellit16}, define, for an admissible coloring  $c$ of  $\ell_h$,
\begin{align}
D_{s,c} = \sum_{\substack{\text{$(km,kn)$-Dyck paths $\pi$} \\  (\text{region below }\pi) \cap \ell_h = c.}} D^h_\pi.
\end{align}
The core of Mellit's proof is to inductively establish an
identity equating the sum $D_{s,c}$ with a single product of the operators \(T_i, y_i, z_i, d_+\) applied to 1, determined from the lattice lines intersecting~$c$.

Now to specifically address \eqref{e ns shuffle unmod 2},
consider \(h_1 = kn-kms_-+\epsilon\) so that \(\ell_{h_1}\)
intersects \((km,kn+\epsilon)\) and let \(c_\alpha\) be the admissible coloring of  $\ell_{h_1}$ given by the union of  $r$ line segments such that the  $j$-th segment
has left endpoint with  $x$-coordinate  $m \sum_{i < j} \alpha_i$ and
right endpoint with  $y$-coordinate  $n \sum_{i \le j} \alpha_i$.
For example, for  $n=m=1$ and  $\alpha = (2,2,1)$,  $c_\alpha$ consists of the three line segments below.
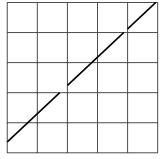
\begin{align*}
\raisebox{-4mm}{
\begin{tikzpicture}[scale=.4] 
\draw [very thin, black!70] (0,0) grid (5,5);
\draw [line width=0.22mm] (0,6*.06) -- (2/.94-6*.06/.94,2);
\draw [line width=0.22mm] (2,2*.94+6*.06) -- (4/.94-6*.06/.94,4);
\draw [line width=0.22mm] (4,4*.94+6*.06) -- (5/.94-6*.06/.94,5);
\end{tikzpicture}}
\end{align*}

For the admissible coloring $c= c_\alpha$, the abovementioned identity for  $D_{s,c_\alpha}$ is the second to last displayed equation in~\cite{Mellit16}, which states
\begin{equation}
\label{e mellit ns shuffle}
 D_{s,c_\alpha} = (-1)^{k(m+1)}\, (qt)^{r-k}\,
\rho^*_{m,n}\big( y_1^{\alpha_1-1}\cdots y_r^{\alpha_r-1} d_+^r  \big) \cdot 1.
\end{equation}
Note that for a  $(km,kn)$-Dyck path  $\pi$, we have $(\text{region below }\pi) \cap \ell_{h_1} = c_\alpha$ if and only if  $\pi \in \DD_{km,kn}^\alpha$,
so  $D_{s, c_\alpha} = \sum_{\pi\in \DD_{km,kn}^\alpha} D^{h_1}_\pi$.
To connect this to the right hand side of \eqref{e ns shuffle unmod 2}, it is shown in \cite[\S4.1]{Mellit16} that
\begin{align}
\label{e Mellit LLT term}
D^{-\epsilon}_\pi = q^{\area(\pi)} t^{\dinv(\pi) - \maxtdinv(\pi) } d_-^r \, \chi_r(\hat{\pi}',\Sigma_\pi),
\end{align}
which is close to what we want. What we actually want is
\begin{equation}
\label{e Mellit LLT term2}
D^{h_1}_\pi = q^{\area(\pi) + r-k} t^{\dinv(\pi) - \maxtdinv(\pi) } \chi_r(\hat\pi', \Sigma_\pi),
\end{equation}
which can be deduced from the proof of \eqref{e Mellit LLT term}.
To see this, observe that the operator sequence to go from  $D^{h_1}_\pi$ to  $D^{-\epsilon}_\pi$ corresponds to
the lattice points between the
lines  $\ell_{h_1}$ and \(\ell_{-\epsilon}\), which are the  $k+1$
lattice points on the diagonal.
According to \cite[\S4.1]{Mellit16} (see also \cite[Remark~4.6]{Mellit16}),
each non-touchpoint
lattice point registers a  $q$ factor, while each of the $r+1$
touchpoint lattice points registers a $d_-$, except the top-right one, which registers a factor of  $t^0$.

We conclude from \eqref{e Mellit LLT term2} that the right hand side of \eqref{e ns shuffle unmod 2} is equal to
$q^{k-r} D_{s,c_\alpha}$. Combining this with \eqref{e mellit ns shuffle} proves \eqref{e ns shuffle unmod 2}.
\end{proof}

\begin{remark}
\label{rem: match Mellit chi}
Our notation \(d_-^r \chi_r(\hat{\pi}',\Sigma_\pi)\) in the proof above is equal to
 \(\chi(\pi',\Sigma_\pi)\) from~\cite[\S\S3.5, 4.1]{Mellit16},
which was also previously mentioned in Remark \ref{r flagLLT r0} (iii).
\end{remark}

\subsection{The nonsymmetric compositional shuffle theorem} 
The virtue of a signed, compositional shuffle theorem, expressed as a sum of signed flagged LLT polynomials, is that its image under the
$r$-nonsymmetric plethysm map $\Pisf_r$ yields a positive, combinatorial sum of flagged LLT polynomials.
Precisely, applying
  $\Pisf_r$
to Theorem \ref{t ns shuffle unmod} and using Theorem \ref{t signed to unsigned} yields the following combinatorial expression for $\Pisf_r \, \alg_{m,n}^\alpha $
which is manifestly monomial positive.

\begin{thm}
\label{t ns shuffle kmkn mod}
With the notation from the beginning of
\S\ref{ss Nonsymmetric compositional kmkn shuffle theorem} and in \S \S \ref{ss Dyck path stats}--\ref{ss flagged LLT polynomials},
\begin{align}
\label{e ns shuffle kmkn mod}
 \Pisf_r \, \alg_{m,n}^\alpha =
\sum_{\pi \in \DD^\alpha_{km,kn} }  q^{\area(\pi)} t^{\dinv(\pi) - \maxtdinv(\pi) } \,
\Grow_r(\hat{\pi}', \Sigma_\pi).
\end{align}
\end{thm}

Theorem \ref{t ns shuffle kmkn mod} can also be written in a form closer to the classical statement of the shuffle theorem.
Define a \emph{flagged weak word parking function} $w$ on $\pi \in \DD^\alpha_{km,kn}$
to be a map from the north steps of  $\pi$ to  $\ZZ_+$ that is weakly decreasing going north along vertical runs and such that the entry on the north step from the point
$(m\sum_{i < j}\alpha_i, n \sum_{i < j}\alpha_i)$ is  $\le j$, for  $j = 1,\dots, r$.
Denote the set of these by  $\FWPF'_\alpha(\pi)$.

For  $w \in \FWPF'_\alpha(\pi)$, define the \emph{weak temporary diagonal inversion statistic} by
\begin{align}
\label{e wtdinv def}
\rm{tdinv}'(\pi, w) = \big|\big\{ \text{pairs of north steps  $i$, $j$ of  $\pi$} : \text{ $i$ attacks   $j$ and $w(i) \ge w(j)$}  \big\}  \big|.
\end{align}

\begin{cor}
\label{c ns shuffle kmkn mod}
With the notation above,
\begin{align}
\label{ec ns shuffle kmkn mod}
\Pisf_r \, \alg_{m,n}^\alpha =
\sum_{\pi \in \DD^\alpha_{km,kn} } \sum_{w \in \FWPF'_\alpha(\pi)} q^{\area(\pi)} t^{\dinv(\pi) - \maxtdinv(\pi) +\wtdinv(\pi, w)} \,
\xx^{\content(w)},
\end{align}
where $\xx^{\content(w)} = \prod_{i \in \ZZ_+} x_i^{\# \text{ of  $i$'s in $w$}} $.
\end{cor}

\begin{remark}
Two other variants of the  $\wtdinv$ statistic appear in the literature and elsewhere in this paper.
Replacing  $w(i) \ge w(j)$ with  $w(i) > w(j)$ in \eqref{e wtdinv def} yields the temporary dinv statistic  $\tdinv(\pi,w)$ used in \cite{BeGaSeXi16, Mellit16}.
Replacing $w(i) \ge w(j)$ with  $w(i) < w(j)$ yields the dinv statistic  $\dinv(\pi, w)$ used in
\cite{HagMoZa12} and in Theorem \ref{t mod nsshuffle intro} and \eqref{e dinv HMZ def} of this paper.
\end{remark}

As touched on in the introduction and as will be further discussed in \S\ref{ss atom pos 2}, two advantages to working on the modified side are that
the nonsymmetric shuffle theorems have a conjectural positive combinatorial interpretation,
and that the ordinary shuffle theorem is recovered simply by Weyl symmetrizing.

To see this symmetrization property,
first note that
\begin{align}
\label{e two plethysms}
\Pisf_0 \PP_0^{-1} (f[W]) = f[-W] = (-1)^{d} \omega f[W],
\end{align}
for any symmetric function  $f$ of degree  $d$.

By \eqref{ep sym flag LLT} and \eqref{e two llts agree},
$\Weyl_1 \cdots \Weyl_r \Grow_r(\hat{\pi}', \Sigma_\pi) =
 \omega \spa \chi(\pi', \Sigma_\pi)$, where
 $\chi(\pi', \Sigma_\pi)$ is the LLT polynomial with the notation of \cite[\S3.5]{Mellit16},
 discussed in Remark \ref{r flagLLT r0} (iii).
Hence, applying  $\Weyl_1 \cdots \Weyl_r$ to both sides of
\eqref{e ns shuffle kmkn mod} yields the first equality below
\begin{multline}
\label{e sym kmkn}
\sum_{\pi \in \DD^\alpha_{km,kn} }  q^{\area(\pi)} t^{\dinv(\pi) - \maxtdinv(\pi) } \,
\omega \, \chi(\pi', \Sigma_\pi)=
\Weyl_1 \cdots \Weyl_r\, \Pisf_r \, \alg_{m,n}^\alpha  \\
= \Pisf_0 \, \hsym_1 \cdots \hsym_r \, \alg_{m,n}^\alpha
 =
\Pisf_0 \PP_0^{-1} d_-^r \PP_r \alg_{m,n}^\alpha
= \omega (-1)^{k(m+1)} t^{r-k}\,
d_-^r \, \rho^*_{m,n}\big( y_1^{\alpha_1-1}\cdots y_r^{\alpha_r-1} d_+^r  \big) \cdot 1;
 \end{multline}
the second equality is by Proposition \ref{p dminus and sym}, the third is by
Proposition \ref{p: d- and hsym}, and the last is by \eqref{e two plethysms}.
This is exactly  $\omega$ of Mellit's statement of the compositional  $(km,kn)$-shuffle theorem.

\section{Modified nonsymmetric Macdonald polynomials}
\label{s mod nsmac}

The (compositional) shuffle theorems involve the
$\nabla$ operator, which acts diagonally on modified Macdonald polynomials.
To prove our nonsymmetric version of this result (Theorem~\ref{t mod
  nsshuffle intro}), we need some
background on stable and modified nonsymmetric Macdonald polynomials \cite{BechtloffWeising23, BHMPS-nspleth}.

\subsection{Stable nonsymmetric Macdonald polynomials}

We follow the conventions of \cite{BHMPS-nspleth} for nonsymmetric Macdonald polynomials  $E_\beta$ and their integral forms  $\Ecal_\beta$.
For $\beta \in \NN^N$, let  $E_\beta(\xx_N;q,t) = E_\beta(x_1,\dots, x_N;q,t)$
denote the monically normalized nonsymmetric Macdonald polynomial; it satisfies the
triangularity condition
$E_{\beta }(\xx_N;q,t) = \xx^{\beta } + \sum _{\alpha <\beta } c_{\alpha }(q,t) \xx^{\alpha },$
for a certain order $<$ on  $\ZZ^N$ defined via Bruhat order on the extended affine Weyl group
(see, e.g., \cite[\S 5.4]{BHMPS-nspleth} or \cite[\S 2.1]{HagHaiLo08}).
Knop's integral form nonsymmetric Macdonald polynomials \cite{Knop97} are then given by
\begin{equation}\label{e:integral-form}
\Ecal _{\beta }(\xx_N;q,t) = \bigl(\prod \nolimits_{\, u\in \dg (\beta )}
(1-q^{a(u)+1}\, t^{l(u)+1}) \bigr)\, E_{\beta }(\xx_N;q,t),
\end{equation}
where  $\dg(\beta) = \big\{(i,j) \in \ZZ_{+}^2 :
1 \le j \le N, \, -\beta_{N+1-j}+1 \le i \le 0  \big\}$
and the arm $a(u)$ and leg $l(u)$ are nonnegative integer statistics on the boxes of  $\dg(\beta)$ (see, e.g., \cite[\S 7.1]{BHMPS-nspleth}).

\begin{remark}
The  $E_\beta$ here agree with \cite{BHMPS-nspleth, HagHaiLo08, HalversonRam}, and in terms of the notation here, the nonsymmetric Macdonald polynomials
in \cite{BechtloffWeising23} are $E_\beta(\xx_N;q^{-1},t)$ and those in \cite{Knop97} are $E_{\beta_N, \dots, \beta_1}(x_N, \dots, x_1;q,t)$.
The  $\Ecal_\beta$
here agree with \cite{BHMPS-nspleth, HagHaiLo08}, and in terms of the notation here, those in \cite{Knop97} are $\Ecal_{\beta_N, \dots, \beta_1}(x_N, \dots, x_1;q,t)$.
\end{remark}

Denote a variant of  $\Ecal_\beta$ by
\begin{align}
\label{e qflip nsmac}
\tEcal_\beta(\xx_N; q,t) =
q^{\mathsf{n}((\beta_+)^*)} \Ecal_\beta(\xx_N; q^{-1},t),
\end{align}
where $\beta_+$ is the partition rearrangement of  $\beta$,  $(\beta_+)^*$ denotes its conjugate partition,
and $\mathsf{n}(\mu ) = \sum_{i} (i-1)\mu_{i}.$

Below we will freely use notation on convergence and Weyl and Hecke symmetrization from
\S\ref{ss convergence} and \S\ref{ss Weyl and Hecke sym}.

\begin{defn}
\label{d stable nsmac}
For $(\eta|\lambda) \in \pairsr$,
the \emph{integral form stable $r$-nonsymmetric Macdonald polynomial}
$\stE_{\eta|\lambda}(\xx;q,t) \in \P(r)$ is given by
\begin{align}
\label{e Tsym fact 2}
 \stE_{\eta| \lambda}(\xx;q,t)
= \lim_{N \to \infty} \hsym_{r+1}^N  \, \tEcal_{(\eta; \lambda; 0^{N-r-\ell(\lambda)})}(\xx_N;q,t),
\end{align}
where the sequence converges $t$-adically to a well-defined limit by
\cite[Proposition 7.4.5]{BHMPS-nspleth}.
\end{defn}

\begin{remark}
(i) Bechtloff Weising introduced stable versions $\widetilde{E}_{(\eta|\lambda)}$ of $E_\beta(\xx_N;q^{-1},t)$ and showed that
they have a description using partial Hecke symmetrization \cite[Corollary 4.26]{BechtloffWeising23},
and we have closely followed this description for the definition above;
our versions are related by
\begin{align}
\label{e two versions of stable J}
 \stE_{\eta| \lambda} = q^{\mathsf{n}(((\eta;\lambda)_+)^*)} \bigl(\prod \nolimits_{\, u} (1-q^{-a(u)-1}\, t^{l(u)+1})\bigr)\, \widetilde{E}_{(\eta|\lambda)}
\end{align}
where the product is over  $u = (i,j) \in \dg ((\eta;\lambda))$ with  $i < 0$;
see \cite[Corollary 7.4.6 and Remark 7.4.9(ii)]{BHMPS-nspleth}.
Below we will recall several results of \cite{BechtloffWeising23} and will
state them in terms of the  $\stE_{\eta|\lambda}$'s rather than $\widetilde{E}_{(\eta|\lambda)}$'s;
no adjustment is required as the results only depend on the functions in \eqref{e two versions of stable J} up to a scalar factor.

(ii) Up to a  $q$-flip, the integral form version defined above was introduced in \cite{BHMPS-nspleth};
more precisely,
$J_{\eta| \lambda}(x_1,\dots, x_r, Y ;q,t)$ with $Y = x_{r+1}, x_{r+2}, \ldots $ from \cite[\S 7.4]{BHMPS-nspleth}
is equal to $q^{\mathsf{n}(((\eta;\lambda)_+)^*)} \stE_{\eta| \lambda}(\xx;q^{-1},t)$ here.
It was established in \cite{BHMPS-nspleth} that $J_{\eta| \lambda}(x_1,\dots, x_r, Y ;q,t)$ has coefficients in $\ZZ[q,t]$;
this implies that the same holds for the $\stE_{\eta| \lambda}(\xx;q,t)$ because the highest  $q$ power of  $J_{\eta| \lambda}$ is
$q^{\mathsf{n}(((\eta;\lambda)_+)^*)}$,
as follows from our formula for  $J_{\eta|\lambda}$ in terms of signed flagged LLT polynomials \cite[Proposition~7.5.2]{BHMPS-nspleth}.

(iii) Although $\Ecal_\beta(\xx_N;q,t) \in  \ZZ[q,t][\xx_N]$ by \cite{Knop97}, this is typically not true for the  $\tEcal_\beta(\xx;q,t)$
since the highest  $q$ power of  $\Ecal_\beta(\xx_N;q,t)$ is typically larger than  $q^{\mathsf{n}((\beta_+)^*)}$.
\end{remark}

While  $\pairsr$ is the natural index set for a basis of  $\P(r)$, we also want an index set for a basis of  $\Pas$.
Accordingly, we define
$$
\pairs  =
\bigsqcup_{r \ge 0 } \, \big\{(\eta|\lambda) \in \pairsr \spa : \spa \eta_r>0 \big\}\,,
$$
where, by convention, when $r=0$, there is the single possible $\eta=\varnothing$.

The  $\Ecal_\alpha$ ($\alpha \in \NN^n$)
are joint eigenfunctions of the Cherednik operators $\Y_1,\dots, \Y_n \in \H_n$.
The following result of Bechtloff Weising establishes the analogous result
in the stable limit.

\begin{thm}[{\cite[Corollary 4.15]{BechtloffWeising23}}]
\label{t Weising 38}
For $(\eta|\lambda) \in \pairsr$, and any  $j \in \ZZ_+$,
\begin{align}
\label{e t Weising 38}
\YIW_j \stE_{\eta|\lambda} =
\begin{cases}
q^{\eta_j} \spa t^{r+ \ell(\lambda) +1 -\st((\eta;\lambda))_j}  \, \stE_{\eta|\lambda}
& \text{$j \le r$ and  $\eta_j > 0$}, \\
0  & \text{otherwise},
\end{cases}
\end{align}
where
$\st((\eta;\lambda))$ is the word obtained from $(\eta;\lambda)$ by relabeling its smallest letter with  $1,\dots, k$
from left to right, the next smallest by  $k+1,\dots$, and so on.
\end{thm}

The  $\stE_{\eta|\lambda}$ are compatible for different  $r$ in the sense that
\begin{align}
\label{e different r for nsmac0}
\stE_{\eta|\lambda} =\stE_{(\eta,0)|\lambda}.
\end{align}
This will be proved in the next subsection.

\begin{thm}[{\cite[Theorem 4.30]{BechtloffWeising23}}]
\label{t BW eigenbasis}
The set $\big\{ \stE_{\eta|\lambda} : (\eta|\lambda) \in \pairs \big\}$
forms an eigenbasis of $\Pas$ for the commuting operators  $\YIW_i$.
For each  $r\ge 0$, the set $\big\{ \stE_{\eta|\lambda} : (\eta|\lambda) \in \pairsr \big\}$ forms an eigenbasis of the subspace  $\P(r)$.
\end{thm}

\begin{prop}[{\cite[Proposition 7.8.6]{BHMPS-nspleth}}]\label{p:Hsym J}
\label{p: hsym J}
Partial Hecke symmetrization acts in the following simple way on integral form stable $r$-nonsymmetric Macdonald polynomials.
For any $(\eta|\lambda) \in \pairsr$,
\begin{align}
\label{e Tsym fact}
\hsym_r \, \stE_{\eta | \lambda}   = \stE_{ (\eta_1, \dots, \eta_{r-1})| (\eta_r, \lambda)_+}.
\end{align}
Full Hecke symmetrization yields
\begin{align}
\label{e Tsym fact2}
\hsym_1 \cdots \hsym_r \, \stE_{\eta | \lambda}(\xx;q,t)   = \stE_{ \varnothing | (\eta; \lambda)_+}(\xx;q,t) = q^{\mathsf{n}(((\eta;\lambda)_+)^*)} J_{(\eta;\lambda)_+}(\xx;q^{-1},t),
\end{align}
where  $J_\mu(\xx;q,t)$ is the integral form symmetric Macdonald polynomial.
\end{prop}

Bechtloff Weising
uses Theorem \ref{t Weising 38} to deduce an
explicit formula for $T_i \stE_{\eta|\lambda}$, for $i \in [r-1]$, in terms of $\stE_{\eta|\lambda}$ and  $\stE_{s_i(\eta) | \lambda}$,
an analog of the Knop-Sahi intertwiner formula for the  $E_\beta$'s. For later use, we record the following result, which is mostly a consequence of this formula.

\begin{prop}
\label{p: prop 5.2}
For any  $(\eta | \lambda) \in \pairsr$ and  $i  \in [r-1]$,
\begin{align}
\label{e Ti on stE}
T_i \spa \stE_{\eta | \lambda} \in \spn_{\QQ(q,t)} \{\stE_{\eta | \lambda}, \stE_{s_i(\eta) | \lambda} \}.
\end{align}
\end{prop}
\begin{proof}
For  $\eta_i \ne \eta_{i+1}$, the result follows from the formula of Bechtloff Weising mentioned just above.
Now suppose $\eta_i = \eta_{i+1}$.
Then $\tEcal_{(\eta; \lambda; 0^{N-r-\ell(\lambda)})}(\xx_N;q,t)$ is symmetric in  $x_i, x_{i+1}$ (see \cite[eq. (5.4.4)]{Macdonald03}),
and so
$T_i \spa \tEcal_{(\eta; \lambda; 0^{N-r-\ell(\lambda)})}= \tEcal_{(\eta; \lambda; 0^{N-r-\ell(\lambda)})}$.
We deduce $T_i \spa \stE_{\eta | \lambda} = \stE_{\eta|\lambda}$,
as follows from Definition \ref{d stable nsmac}, the commutation relation $T_i \hsym_{r+1}^N = \hsym_{r+1}^N T_i$ for  $i < r$, and the fact that $T_i \spa (\QQ(q,t)[\xx_N])_e \subseteq (\QQ(q,t)[\xx_N])_e$.
\end{proof}

\subsection{Modified nonsymmetric Macdonald polynomials}
\begin{defn}[{\cite[\S 7.5]{BHMPS-nspleth}}]
\label{d modnsmac}
For  $(\eta | \lambda) \in \pairsr$, define the \emph{modified $r$-nonsymmetric Macdonald polynomial},
in terms of the  $r$-nonsymmetric plethysm map $\Pisf_r$ from \eqref{e: nspleth def}, by
\begin{align}
\tE_{\eta|\lambda} = \Pisf_r \, \stE_{\eta|\lambda}.
\end{align}
\end{defn}

\begin{prop}
The map  $\Pisf_r$ and the elements  $\stE_{\eta|\lambda} $ and $\tE_{\eta|\lambda}$
are compatible for different  $r$ in the following sense:
\begin{align}
\label{e different r for Pisf}
\text{the restriction of $\Pisf_r$ to $\P(r-1) \subseteq \P(r)$ agrees with  $\Pisf_{r-1}$,}
\end{align}
\begin{align}
\label{e different r for nsmac}
\stE_{\eta|\lambda} =\stE_{(\eta,0)|\lambda},
\end{align}
\begin{align}
\label{e different r for modnsmac}
\tE_{\eta|\lambda} =\tE_{(\eta,0)|\lambda}.
\end{align}
\end{prop}
\begin{proof}
By \cite[Remark 4.1.2 (i) and Proposition 4.3.5]{BHMPS-nspleth},
if $g(\xx_r) \in \QQ(q,t)[\xx_r]$ is independent of  $x_r$,
then
$\pol_{\xx_r}\!\!\big(\frac{g(\xx_r)}{\prod_{1 \le i < j \le r}(1-t \spa x_i/x_j)}  \big)=
\pol_{\xx_{r-1}}\!\!\big(\frac{g(\xx_r)}{\prod_{1 \le i < j \le r-1}(1-t \spa x_i/x_j)}  \big) $,
where $\pol_{\xx_N}$ is the polynomial part operator from \eqref{e:nsCalpha}
taken with respect to the basis  $\{\Dcal_\alpha : \alpha \in \ZZ^N\}$ of  $\QQ(q,t)[x_1^{\pm 1}, \cdots, x_N^{\pm 1}]$.
The identity in \eqref{e different r for Pisf} then follows by computing  $\Pisf_r$ on the basis
$\big\{x_1^{\eta_1} \cdots x_r^{\eta_r} s_\lambda(\xx) : (\eta| \lambda) \in \pairsr \big\}$ of  $\P(r)$, which has the convenient property that
the subset with  $\eta_r = 0$ is a basis for  $\P(r-1)$.

The identity \eqref{e different r for modnsmac}
follows from the formula for  $\tE_{\eta|\lambda}$ as a sum of flagged LLT polynomials \cite[Theorem 7.6.6 (b)]{BHMPS-nspleth} and
\cite[Remark 7.6.8]{BHMPS-nspleth} which gives an explicit combinatorial formula for these flagged LLT terms.  Adding the extra 0 to go from  $\eta | \lambda$ to
$(\eta, 0) | \lambda$ adds an extra empty row to the tableaux in this formula, but the remark shows that this does not change the flagging conditions
or the $\inv$ statistic.
Identity \eqref{e different r for nsmac} then follows from \eqref{e different r for Pisf} and \eqref{e different r for modnsmac}.
\end{proof}

The $\tE_{\eta|\lambda}$, as  $(\eta|\lambda)$ ranges over  $\pairs$, form a basis for  $\Pas$.
For each  $r \ge 0$, the subset indexed by $\pairsr$ forms a basis of $\P(r)$.

\begin{prop}[{\cite[Proposition 7.8.3]{BHMPS-nspleth}}]
\label{p sym modnsmac}
Weyl symmetrization acts in the following simple way on modified $r$-nonsymmetric Macdonald polynomials.
For any $(\eta|\lambda) \in \pairsr$,
\begin{align}
\Weyl_r \, \tE_{\eta | \lambda}   = \tE_{ (\eta_1, \dots, \eta_{r-1})| (\eta_r, \lambda)_+}.
\end{align}
Full Weyl symmetrization yields
\begin{align}
\Weyl_1 \cdots \Weyl_r \spa \tE_{\eta|\lambda}(\xx;q,t)
=\tE_{\varnothing |(\eta;\lambda)_+}(\xx;q,t)
=\omega\Htild_{(\eta; \lambda)_+}( \xx;q,t).
\end{align}
\end{prop}

Propositions \ref{p: hsym J} and \ref{p sym modnsmac} are nicely summarized by the following commutative diagram,
which holds for any $(\eta|\lambda) \in \pairsr$ and $\mu = (\eta;\lambda)_+$.
\vspace{-2mm}

 \newcommand{\toplabel}{f[\xx] \mapsto f[\xx/(1-t)]}
\begin{equation}
  \label{e mac comm diagram}
\begin{tikzcd}[column sep = 2.8cm, row sep = 1.3cm, nodes={inner sep = 1.2mm}]
 q^{\mathsf{n}(\mu^*)} J_{\mu}(\xx;q^{-1}, t)
\ar[r, mapsto, swap, "\toplabel"]
&
 \omega\Htild_{\mu}(\xx;q,t)
\\
\ar[mapsto,u, "\hsym_1 \cdots \hsym_{r}"]
\stE_{\eta|\lambda}(\xx;q,t)    \ar[mapsto,r, "\Pisf_r"] &
\tE_{\eta|\lambda}(\xx;q,t)
 \ar[u, mapsto, swap,"\, \raisebox{-2mm}{$\Weyl_1 \cdots \Weyl_r$}"]
\end{tikzcd}
\end{equation}

\section{Nonsymmetric \texorpdfstring{$\nabla$}{∇}
formulation of the nonsymmetric shuffle theorem}
\label{s ns nab ns shuffle}

The operator $\nabla$, introduced in \cite{BeGaHaTe99}, is the linear
operator on symmetric functions which acts diagonally on the basis of
modified Macdonald polynomials $\Htild_{\mu }(\xx;q,t)$ by
$\nabla \Htild _{\mu } = q^{\mathsf{n}(\mu^{*})} t^{\mathsf{n}(\mu )}\Htild_{\mu }$.
We now introduce two nonsymmetric versions of  $\nabla$: an operator $\umnab$ which acts diagonally on the $\stE_{\eta|\lambda}$'s and $\modnab$ which acts diagonally on the $\tE_{\eta| \lambda}$'s. Then, in Theorem~\ref{t two nablas}, we show that the  $\PP_r$ conjugate of $\umnab$ is equal, up to a sign, to the operator on $V_*$ induced by the endomorphism $N$ of $\mathbb A_{t,q}$. From there, we deduce the nonsymmetric compositional shuffle theorem stated in terms of $\modnab$ (Theorem~\ref{t mod nsshuffle intro}).

\subsection{Nonsymmetric \texorpdfstring{$\nabla$}{∇} defined via nonsymmetric Macdonald polynomials}
Define the \emph{signed nonsymmetric nabla operator} $\umnab \colon \Pas \to \Pas$ to be the $\QQ(q,t)$-linear map
determined by
\begin{align}
\label{e umnab def}
\umnab \stE_{\eta| \lambda}(\xx ;q,t) =
q^{\mathsf{n}(\mu^*)}t^{\mathsf{n}(\mu)}\stE_{\eta| \lambda}(\xx ;q,t),
\qquad \text{where $\mu = (\eta; \lambda)_+$},
\end{align}
as  $(\eta | \lambda)$ ranges over  $\pairs$.
Note that by \eqref{e different r for nsmac}, \eqref{e umnab def} then holds for any  $(\eta|\lambda) \in \pairsr$ for any  $r \ge 0$.

Recall that in \eqref{e mod nabla intro},
we defined the \emph{nonsymmetric nabla operator} $\modnab \colon \Pas \to \Pas$ to be the $\QQ(q,t)$-linear map
determined by
\begin{align}
\label{e modnab def}
\modnab \tE_{\eta| \lambda}(\xx ;q,t) =
q^{\mathsf{n}(\mu^*)}t^{\mathsf{n}(\mu)}\tE_{\eta| \lambda}(\xx ;q,t),
\qquad \text{where $\mu= (\eta; \lambda)_+$},
\end{align}
as  $(\eta | \lambda)$ ranges over  $\pairs$.
Note that by \eqref{e different r for modnsmac}, \eqref{e modnab def} then holds for any  $(\eta|\lambda) \in \pairsr$ for any  $r \ge 0$.

Note that since the set $\{\stE_{\eta|\lambda} : (\eta | \lambda) \in \pairsr \}$
spans  $\P(r)$,  $\umnab$ fixes the subspace  $\P(r) \subseteq \Pas$.
The same goes for  $\modnab$.

The operator  $\modnab$ is a nonsymmetric version of  $\omega\nabla \omega$ in the sense that
\begin{align}
\label{e modnab and sym}
\Weyl_1 \cdots \Weyl_r \spa \modnab \spa f
= \omega \spa \nabla \spa \omega \spa \Weyl_1 \cdots \Weyl_r \spa f
\end{align}
for any  $f \in \P(r)$.
This holds by Proposition \ref{p sym modnsmac}.
This implies that $\modnab$ restricted to
the space of symmetric functions agrees with $\omega\nabla \omega$, but is also a significantly stronger notion.

Also note that by Definition \ref{d modnsmac},
\begin{align}
\label{e umnab and modnab}
\Pisf_r \umnab \Pisf_r^{-1}  = \modnab \quad  \text{as operators on  $\P(r)$}.
\end{align}

\subsection{An operator identity for \texorpdfstring{$\umnab $}{∇̅}}
The goal of this subsection is to prove the following operator identity, which will allow us to relate $\umnab$ to the endomorphism $N$ of $\AA_{t,q}$.
\begin{lemma}
\label{l x1 to y1x1}
We have the following identity of operators on  $\Pas$:
\begin{align}
\label{el x1 to y1x1}
\umnab \X_1 = (qt)^{-1} \YIW_1 \X_1 \umnab.
\end{align}
\end{lemma}

To prove this, we need to analyze which  $\stE_{\gamma|\mu}$'s appear in the expansion of
$x_1 \stE_{\eta|\lambda}$.
For this, we use a `Monk rule' for nonsymmetric Macdonald polynomials,
which is an explicit formula for $x_j \spa E_\alpha$ in the basis $\{E_\beta : \beta \in \NN^N\}$.
We use the version given by Halverson-Ram \cite{HalversonRam} (see also \cite{Baratta, LascouxSchubertandMacdonald}).
Note that
their rule is expressed in terms of the $E_\beta(\xx_N;q,t)$'s rather than the  $\tEcal_\beta(\xx_N; q,t)$'s,
but we do not need to make an adjustment for this since we only care about which coefficients are nonzero.

\begin{prop}
\label{p Monk}
For any  $\alpha \in \NN^N$,
\begin{align}
x_1 \spa \tEcal_\alpha \in \spn_{\QQ(q,t)} \big\{ \tEcal_\beta :  (\beta_1-1, \beta_2, \dots, \beta_N)_+ = \alpha_+ \big\}.
\end{align}
\end{prop}
\begin{proof}
By \cite[Theorem 3.1 (a)]{HalversonRam},
the $\beta$ for which  $\tEcal_\beta$ appears with nonzero coefficient in  $x_j \spa \tEcal_\alpha$ are of the form  $\rot_C(\alpha)$ for  $j\in C$, where  $\rot_C(\alpha)$ is defined as follows:
let  $C = \{i_1 < \cdots < i_k\} \subseteq \{1,\dots, N\}$.  Define
$\rot_C(\alpha)$ to be the result of replacing the subsequence  $\alpha_{i_1}, \dots, \alpha_{i_k}$ of $\alpha$ with $\alpha_{i_k}+1,\alpha_{i_1}, \dots, \alpha_{i_{k-1}}$, and leaving the rest of  $\alpha$ as is.
For example, for
\begin{align}
\alpha = \mathbf{2}\spa 1\spa 1\spa \mathbf{7}\spa \mathbf{4} \spa 3,  \, C = \{1,4,5\},
\quad \quad \text{we have } \rot_C(\alpha) = \mathbf{5}\spa 1\spa 1\spa \mathbf{2}\spa \mathbf{7} \spa 3,
\end{align}
where the subsequence with indices  $C$ is shown in bold.

Thus, when  $j=1$,  $\beta = \rot_C(\alpha)$ for some  $C \subseteq \{1,\dots, N\}$ containing 1,
and so  $(\beta_1-1, \beta_2, \dots, \beta_N)$ is a rearrangement of  $\alpha$.
\end{proof}

We will use this to prove a similar result for the integral form stable  $r$-nonsymmetric Macdonald polynomials  $\stE_{\eta|\lambda}$ after two preparatory results.
Recall from \S\ref{ss convergence} that   $(\QQ(q,t)[\xx_N])_e = \{\sum_{\beta \in \NN^N} c_\beta \xx^\beta : \ord(c_\beta) \ge e\}$.

\begin{lemma}
\label{l Hecke sym nsmac}
(i) For  $\beta \in \NN^N$,
\begin{align}
\label{el Hecke sym nsmac}
\hsym_{r+1}^N \tEcal_\beta  \in \spn_{\QQ(q,t)} \{\tEcal_\alpha  : \alpha \in \SS_{(1^r; N-r)}\beta \},
\end{align}
where  $\SS_{(1^r; N-r)}\beta$ denotes the orbit of  $\beta$ under the symmetric group action permuting the last  $N-r$ indices.

(ii) For  $\beta \in \NN^N$, $\hsym_{r+1}^N \tEcal_\beta \notin (\QQ(q,t)[\xx_N])_1$.

(iii)  $(\QQ(q,t)[\xx_N])_e = \{\sum_\beta c_\beta \tEcal_\beta : \ord(c_\beta) \ge e\}$.
\end{lemma}

\begin{proof}
Part (i) follows from the Knop-Sahi recurrence,
which expresses  $T_i E_\beta$ as a  $\QQ(q,t)$-linear combination of  $E_\beta$ and  $E_{s_i \beta}$ (see, e.g., \cite[\S5.4]{BHMPS-nspleth}).

For part (ii), first, for the monically normalized  $E_\beta$ when  $\beta_{r+1} \ge \cdots \ge \beta_N$, from eq. (178) 
and the paragraph surrounding
eq. (316) in \cite{BHMPS-nspleth}, we have that the coefficient of
$\xx^{\beta}$ in  $\hsym_{r+1}^N E_\beta$ is equal to $\frac{\prod_{i} [m_i]_t!}{[N-r]_t!}$, where  $m_i$ is the number of  $i$'s in  $(\beta_{r+1}, \dots, \beta_N)$.
This constant has order of vanishing 0, and so $\hsym_{r+1}^N E_\beta \notin (\QQ(q,t)[\xx_N])_1$.
Since  $E_\beta$ and  $\Ecal_\beta$ differ by a scalar factor which has order of vanishing 0 (see \eqref{e:integral-form}),
(ii) follows for  $\beta$ with $\beta_{r+1} \ge \cdots \ge \beta_N$.  The general case then follows from the fact that  $\hsym_{r+1}^N \Ecal_\beta = \hsym_{r+1}^N \Ecal_{\alpha}$ for any $\alpha \in \SS_{(1^r;N-r)}\beta$, by \cite[Proposition 7.1.6]{BHMPS-nspleth}.

For (iii), the inclusion  $\supseteq$ follows from $\tEcal_\beta \in \ZZ[q,q^{-1},t][\xx_N]$.
For the reverse inclusion, first note that the matrix  $A$ expressing the $\tEcal_\beta$'s in terms of monomials has coefficients in  $\ZZ[q,q^{-1},t]$
and is triangular with diagonal entries $q^{\mathsf{n}((\beta_+)^*)} \bigl(\prod \nolimits_{\, u\in \dg (\beta )}(1-q^{-a(u)-1}\, t^{l(u)+1}) \bigr)$.
Hence, restricting to degree  $d$ and using the adjugate formula for  $A^{-1}$,
we see that $A^{-1}$ has entries in  $\frac{1}{\det(A)} \ZZ[q,q^{-1},t]$. Since  $\ord(\det(A)) =0$, the inclusion  $\subseteq$ follows.
\end{proof}

\begin{cor}
\label{c E orbit sums}
Suppose that
\begin{align}
\label{ec E orbit sums}
\sum_{\alpha} d_\alpha \, \hsym_{r+1}^N (\tEcal_\alpha)  \equiv \sum_{\beta \in \NN^N} c_\beta \, \hsym_{r+1}^N (\tEcal_\beta)  \quad \bmod (\QQ(q,t)[\xx_N])_e,
\end{align}
for coefficients  $d_\alpha, c_\beta \in \QQ(q,t)$,
with the sum on the left containing at most one $\alpha$ from each  $\SS_{(1^r; N-r)}$ orbit of  $\NN^N$.
Then $\ord(d_\alpha) < e$ implies that for at least one $\beta$ in the orbit $\SS_{(1^r; N-r)} \alpha$, $c_\beta$ is nonzero.
\end{cor}
\begin{proof}
Expand both sides of \eqref{ec E orbit sums} into the  $\tEcal_\alpha$ basis of  $\QQ(q,t)[\xx_N]$, and denote
the expansion on the left by $\sum_{\tau \in \NN^N} a_\tau \, \tEcal_\tau$.
Given $\ord(d_\alpha) < e$, Lemma \ref{l Hecke sym nsmac} (ii) implies  $d_\alpha \, \hsym_{r+1}^N(\tEcal_\alpha) \notin (\QQ(q,t)[\xx_N])_e$.
Then Lemma \ref{l Hecke sym nsmac} (i) and the assumption on the left sum imply
that for some $\tau \in \SS_{(1^r; N-r)} \alpha$,
$a_\tau \, \tEcal_\tau \notin (\QQ(q,t)[\xx_N])_e$.
From Lemma \ref{l Hecke sym nsmac} (i) and (iii), it follows that at least one  $c_\beta$ among $\beta \in \SS_{(1^r; N-r)} \alpha$ is nonzero.
\end{proof}

For the next two results, it is convenient to denote elements of $\pairsr$ by a single symbol,
and we define the following accompanying notation: for  $\thbold = (\eta | \lambda) \in \pairsr$, we let $\thbold$ also stand for the concatenation  $(\eta;\lambda)$,
so that $\thbold_+ = (\eta;\lambda)_+$,  $\thbold_1 = \eta_1$ when  $r \ge 1$,  $\thbold-\varepsilon_1= (\eta_1-1,\dots, \eta_r; \lambda)$, etc.
Also set  $\ell(\thbold) = r + \ell(\lambda)$.

\begin{prop}
\label{p Monk stable}
For any  $\kapbold \in \pairs$,
\begin{align}
x_1 \spa \stE_\kapbold \in
\spn_{\QQ(q,t)} \big\{ \stE_\thbold : \thbold \in \pairsr \text{ with $r \ge 1$ and }
 (\cc{\thbold}-\varepsilon_1)_+ = \kapbold_+ \big\}.
\end{align}
\end{prop}

\begin{proof}
By \eqref{e different r for nsmac},
 $\stE_\kapbold = \stE_{\eta|\lambda}$ for some
$(\eta|\lambda) \in \pairsr$ where  $r \ge 1$. Hence,
$x_1 \stE_\kapbold\in\P(r)$ and we can write
\begin{align}
\label{e x1 prod}
x_1 \stE_{\kapbold} = \sum_{\thbold \in \pairsr} d_\thbold \, \stE_\thbold,
\end{align}
with coefficients $d_\thbold \in \QQ(q,t)$.
We need to show that  $d_\thbold \ne 0$ only when $(\cc{\thbold}-\varepsilon_1)_+ = \kapbold_+$.

Define  $D_{\kapbold'}^{N} = \stE_{\kapbold'}[\xx_N;q,t] - \hsym_{r+1}^N (\tEcal_{(\kapbold';0^{N-\ell(\kapbold')})})$ for  $\kapbold'  \in \pairsr$ and  $N \ge r$,
 so
that by Definition \ref{d stable nsmac}, for every  $e > 0 $, $D_{\kapbold'}^{N} \in (\QQ(q,t)[\xx_N])_e$ for sufficiently large  $N$.
For any  $N \ge r$, setting the variables  $x_{N+1}, x_{N+2}, \dots$ to 0 in \eqref{e x1 prod}, we obtain
\begin{multline}
\label{e x1 prod2}
x_1 \hsym_{r+1}^N (\tEcal_{(\cc{\kapbold};0^{N-\ell(\kapbold)})}) + x_1 D_{\kapbold}^{N} = x_1 \stE_{\kapbold}[\xx_N;q,t] \\
= \sum_{\thbold \in \pairsr} d_\thbold \, \stE_\thbold[\xx_N;q,t] =
\sum_{\thbold \in \pairsr} d_\thbold \big(\hsym_{r+1}^N(\tEcal_{(\cc{\thbold};0^{N-\ell(\thbold)})}) + D_{\thbold}^{N}\big),
\end{multline}
By Proposition \ref{p Monk} and the fact that multiplication by $x_1$ commutes with  $\hsym_{r+1}^N$, we have
\begin{align}
\label{e x1 prod3}
x_1 \spa \hsym_{r+1}^N (\tEcal_{(\cc{\kapbold};0^{N-\ell(\kapbold)} )}) =
 \hsym_{r+1}^N \bigg( \sum_{\beta \in \NN^N \spa : \spa (\beta-\varepsilon_1)_+ = \kapbold_+} c_\beta \, \tEcal_{\beta} \bigg),
\end{align}
for coefficients  $c_\beta \in \QQ(q,t)$.
Putting \eqref{e x1 prod2} and \eqref{e x1 prod3} together we have
\begin{align}
  \sum_{\beta \in \NN^N \spa :\spa (\beta-\varepsilon_1)_+ = \kapbold_+} \!\!\! c_\beta \, \hsym_{r+1}^N(\tEcal_{\beta} )
 = \sum_{\thbold \in \pairsr} \!\! d_\thbold \, \hsym_{r+1}^N(\tEcal_{(\cc{\thbold};0^{N-\ell(\thbold)}  )}) + \sum_{\thbold \in \pairsr} \!\! d_\thbold \spa D_{\thbold}^{N} -  x_1 \spa D_{\kapbold}^{N}.
\end{align}
Note that there are finitely many nonzero $d_\thbold$ and they are independent of  $N$, while the $c_\beta$ depend on $N$.
Choose any $e  \ge 0$ which is greater than $\max \{\ord(d_\thbold) : d_\thbold \ne 0\}$.
Set  $e' = \min \{\ord(d_\thbold) : d_\thbold \ne 0\}$
(we need to allow for the possibility some \(\ord(d_\thbold)\) are negative).
Then for sufficiently large  $N$, we have  $D_{\thbold}^{N} \in (\QQ(q,t)[\xx_N])_{e-e'}$ for all nonzero $d_\thbold$
and $D_\kapbold^{N} \in (\QQ(q,t)[\xx_N])_{e}$.
Hence, for sufficiently large $N$,
\begin{align}
 \sum_{\beta \in \NN^N \spa : \spa (\beta-\varepsilon_1)_+ = \kapbold_+} c_\beta \, \hsym_{r+1}^N (\tEcal_{\beta})
 \equiv \sum_{\thbold \in \pairsr} d_\thbold \, \hsym_{r+1}^N(\tEcal_{(\cc{\thbold};0^{N-\ell(\thbold)} ) })  \quad \bmod (\QQ(q,t)[\xx_N])_{e}
\end{align}
and the assumptions of Corollary \ref{c E orbit sums} are satisfied. Applying the corollary, we have that for each nonzero  $d_\thbold$, some
$\beta $ in the orbit $\SS_{(1^r;N-r)} (\cc{\thbold};0^{N-\ell(\thbold) } )$ has   $c_\beta \ne 0$, and therefore  $(\cc{\thbold}-\varepsilon_1)_+ = (\beta-\varepsilon_1)_+ =  \kapbold_+$.
\end{proof}

\begin{proof}[Proof of Lemma \ref{l x1 to y1x1}]
We will compare the actions of $\YIW_1 \X_1$ and  $\umnab \X_1 \umnab^{-1}$
on the basis of integral form stable $r$-nonsymmetric Macdonald polynomials.
Let  $\kapbold \in \pairs$.  According to Proposition \ref{p Monk stable}, we can write $x_1 \stE_\kapbold = \sum_\thbold d_\thbold \stE_\thbold$ for
some coefficients $d_\thbold \in \QQ(q,t)$, over $\thbold \in \pairsr$ with  $r \ge 1$ and such that $(\cc{\thbold}-\varepsilon_1)_+ = \kapbold_+$.
Hence,
\begin{align}
\label{e nab x1 nab}
\umnab \X_1 \umnab^{-1} \stE_\kapbold = \sum_\thbold d_\thbold \, q^{\mathsf{n}((\thbold_+)^*)-\mathsf{n}((\kapbold_+)^*)} \,
t^{\mathsf{n}(\thbold_+)-\mathsf{n}(\kapbold_+)} \, \stE_\thbold.
\end{align}

On the other hand, applying Theorem \ref{t Weising 38} and noting that since
$(\cc{\thbold}-\varepsilon_1)_+ = \kapbold_+$, the top case of \eqref{e t Weising 38} always applies, we have
\begin{align}
\label{e y1 x1}
\YIW_1 \X_1 \, \stE_\kapbold = \YIW_1 \sum_\thbold d_\thbold \, \stE_\thbold
=\sum_\thbold d_\thbold \, q^{\cc{\thbold}_1}  t^{\ell(\thbold)+1-\st(\cc{\thbold})_1} \, \stE_\thbold.
\end{align}
For $\thbold$ appearing in the sum,
the skew shape $\thbold_+/\kapbold_+$ is a single box
lying in some row  $a$ and column  $b$.
Still using $(\cc{\thbold}-\varepsilon_1)_+ = \kapbold_+$, and given that there are $\big(\ell(\thbold)+1-\st(\cc{\thbold})_1\big)$ parts of $\thbold_+$ which
are  $\ge \cc{\thbold}_1$,
it follows that $b = \cc{\thbold}_1$ and  $a =\ell(\thbold)+1-\st(\cc{\thbold})_1$.
Moreover, $b-1 = \mathsf{n}((\thbold_+)^*)-\mathsf{n}((\kapbold_+)^*)$ and  $a-1 = \mathsf{n}(\thbold_+)-\mathsf{n}(\kapbold_+)$ directly from definition of the statistic $\mathsf{n}$.
Thus, the right hand side of \eqref{e y1 x1} is  $qt$ times the right hand side of \eqref{e nab x1 nab}.
\end{proof}

\subsection{Two versions of nonsymmetric nabla agree}

By \cite[Theorem 3.13]{Mellit16}, the left $\AA_{t,q}$-module homomorphism $\AA_{t,q}e_0 \to V_*$,  $e_0 \mapsto 1 \in V_0$ is surjective; let  $I$ denote the kernel.
On \cite[pg. 19]{Mellit16} it is noted that
the map $\AA_{t,q}e_0 \to \AA_{t,q}e_0$ induced by the endomorphism $N$ defined in \eqref{e N map} preserves  $I$ and therefore yields
endomorphisms  $\Mnab \colon V_r \to V_r$
for all  $r \ge  0$. This collection of endomorphisms satisfies
$\Mnab (1) = 1$ for  $1 \in V_0$, and
\begin{align}
\label{e nablaprime and N}
\Mnab \circ  L = N(L) \circ \Mnab \quad \text{for all } L\in \AA_{t,q}.
\end{align}
In fact,  $\Mnab$ is determined by these two properties.

Mellit shows \cite[pg. 19]{Mellit16} that  $\Mnab$ restricted to  $V_0 = \Lambda(W)$ is the same as
the operator $\nabla$ up to a sign.
The following result generalizes this to the nonsymmetric setting, in the sense that it relates
$\Mnab$, acting now on any $V_r$, to a nonsymmetric version of $\nabla$ defined to act diagonally on a Macdonald basis.
To see that it is indeed a generalization of Mellit's result, observe that when $r=0$, by \eqref{e two plethysms} and \eqref{e modnab and sym}, both sides of \eqref{et two nablas 2} become $\omega \nabla \omega f$.

\begin{thm}
\label{t two nablas}
For any $f \in \P(r)$ of degree  $d$,
\begin{align}
\label{et two nablas}
\umnab f & = (-1)^d \Mnab^{\PP} f, \\
\label{et two nablas 2}
\modnab f  &= (-1)^d \Pisf_r \Mnab^{\PP}  \Pisf_r^{-1} f,
\end{align}
where  $\Mnab^{\PP}$ is the endomorphism of  $\P(r)$ given by $\PP_r^{-1} \Mnab \PP_r$.
\end{thm}

\begin{remark}
The operator  $\umnab$ is clearly an invertible operator on  $\Pas$.
Thus, Theorem \ref{t two nablas} implies that  $\Mnab$ is invertible, which is not clear from the definition.
\end{remark}

We will prove Theorem \ref{t two nablas} with the help of the following lemma.

\begin{lemma}
\label{l N on y}
The endomorphism  $N$ satisfies
\begin{align}
N(T_i) = T_i, \ \ N(z_i) = z_i, \ \ N(y_1) = -(qt)^{-1}\spa z_1 y_1.
\end{align}
\end{lemma}
\begin{proof}
The first two identities are immediate from \eqref{e N map}.
For the last, we compute, for the loop $y_1 \in \AA_{t,q}$ at quiver vertex  $r \ge 1$,
\begin{align*}
  N(y_1 )
  & = N \left( \frac{1}{t^{r-1}(t-1)} [d_+,d_-] T_{r-1} \cdots T_1
    \right) \\
   & = \frac{-1}{t^{r-1}(t-1)} (qt)^{-1} (z_1 d_+d_- - d_- z_1 d_+)  T_{r-1} \cdots
    T_1 \\
    & = \frac{-1}{t^{r-1}(t-1)} (qt)^{-1} (z_1 d_+d_-  -  z_1 d_- d_+)  T_{r-1} \cdots
    T_1 \\
  & = -(qt)^{-1} z_1 y_1 \,,
\end{align*}
where in the third equality, we have used the relation  $d_- z_i= z_i d_-$ of  $\AA_{t,q}$
\cite[eq. 7.3(d)]{IonWu}, which holds for paths starting from quiver vertex $r$ when $1 \le i <r$;
this applies since precomposing with  $d_+$ yields the relation $d_- z_i d_+= z_i d_- d_+$ for  $1 \le i \le r$, as an identity of paths starting at quiver vertex  $r$.
\end{proof}

\begin{proof}[Proof of Theorem \ref{t two nablas}]
Note that second identity \eqref{et two nablas 2} is immediate from \eqref{et two nablas}
and \eqref{e umnab and modnab}.

We now prove \eqref{et two nablas}.
First, by Theorem \ref{t IW two actions}, the definition of  $N$ in
\eqref{e N map}, the relationship between \(N\) and \(\Mnab\)
  in~\eqref{e nablaprime and N}, and Lemma \ref{l N on y},
\begin{align}
\label{e nablaprime a}
\Mnab^\PP \, T_i &= T_i \, \Mnab^\PP,   \\
\label{e nablaprime b}
\Mnab^\PP \, d_-^\PP &= d_-^\PP \, \Mnab^\PP,   \\
\label{e nablaprime c}
\Mnab^\PP \, \X_1 &= -(qt)^{-1} \YIW_1 \X_1 \, \Mnab^\PP,
\end{align}
where \eqref{e nablaprime a} is an identity of operators from  $\P(r) \to \P(r)$ and  $i < r$,
\eqref{e nablaprime b} is an identity of operators from $\P(r) \to \P(r-1)$,
and \eqref{e nablaprime c} is an identity of operators from $\P(r) \to \P(r)$ for  $r \ge 1$.

Next, we see that the corresponding identities hold for $\umnab$:
by Proposition \ref{p: prop 5.2},
\begin{align}
\label{e umnab a}
\umnab \, T_i &= T_i \, \umnab,
\end{align}
and \eqref{e Tsym fact} (keeping in mind Proposition~\ref{p: d- and hsym}) gives
\begin{align}
\label{e umnab b}
\umnab \, d_-^\PP &= d_-^\PP \, \umnab.
\end{align}
Finally, Lemma \ref{l x1 to y1x1} is the identity corresponding to
\eqref{e nablaprime c}, up to a sign.

Let $1_{k}$ denote the element 1 in  $\P(k)$ or $V_{k}$.
To complete the proof it suffices to show that
any element of $\P(r)$ can be generated using the operators $\X_1$, $T_i$, and $d_-^\PP$, starting from $1_{r'}$ for  $r' \ge r$, and that both  $\umnab$ and  $\Mnab^\PP$ act as the identity on  $1_{r'}$.
Indeed, by \cite[eq. (16)]{Mellit16}, the following elements form a basis for  $V_r$:
\begin{align}
d_-^\ell \,  y_1^{\eta_1}\cdots y_r^{\eta_r} y_{r+1}^{\lambda_1} \cdots y_{r+\ell}^{\lambda_\ell} \cdot 1_{r+\ell},   \quad \text{ for } (\eta| \lambda) \in \pairsr,
\end{align}
where   $\ell = \ell(\lambda)$.
Hence, we have the corresponding basis of  $\P(r)$
\begin{align}
\label{e basis Pr}
(d_-^\PP)^\ell \,  \X_1^{\eta_1}\cdots \X_r^{\eta_r} \X_{r+1}^{\lambda_1} \cdots \X_{r+\ell}^{\lambda_\ell} \cdot 1_{r+\ell},   \quad \text{ for } (\eta| \lambda) \in \pairsr.
\end{align}
The operators in \eqref{e basis Pr} are generated by $\X_1$, $T_i$, and $d_-^\PP$ since  $\X_j = t^{j-1} \, T_{j-1}^{-1}\cdots T_1^{-1} \X_1 T_{1}^{-1}\cdots T_{j-1}^{-1}$.

Finally, we have $\umnab (1_{k}) = 1_{k}$ as  $1_{k} = \stE_{\eta|\lambda}$ for  $(\eta|\lambda)$ the empty pair $(\varnothing  | \varnothing)$, and
\begin{equation}
  \Mnab  (1_{k}) = \Mnab (d_+^*)^{k}  \cdot 1_{0} = N( (d_+^*)^k ) \Mnab (1_{0}) = (d_+^*)^k \Mnab (1_{0}) = 1_{k}.   \qedhere
\end{equation}
\end{proof}

\subsection{The \texorpdfstring{$\modnab$}{nabla} formulation of the nonsymmetric compositional shuffle theorem}
\label{ss nablashuffle}
Say that a map $\varphi \colon M \to N$ of $\QQ(q,t)$-vector spaces is \emph{$\QQ(q,t)$-antilinear} if it is $\QQ$-linear and  $\varphi(c(q,t) g) = c(q^{-1}, t^{-1}) \varphi(g)$
for any  $c(q,t) \in \QQ(q,t)$ and  $g \in M$.

\begin{prop}[{\cite[Proposition 3.15]{Mellit16}}]
There is a  $\QQ(q,t)$-antilinear algebra involution
on the Dyck path algebra $\barw \colon \AA_t \to \AA_t$ determined by
\begin{align}
\label{e barw map}
e_r \mapsto e_r, \ \ T_i \mapsto T_i^{-1}, \ \ d_-  \mapsto d_- , \ \ d_+ \mapsto -t^{-r} d_+,
\end{align}
where the $d_+$'s are arrows starting at quiver vertex  $r$.
\end{prop}

\begin{thm}[{\cite[Theorem 3.12]{Mellit16}}]
\label{t At iso V}
We have the following isomorphism of left  $\AA_t$-modules:
\begin{align}
\AA_t e_0 \xrightarrow{\cong}  \bigoplus_{r=0}^{\infty} y_1 \cdots y_r V_r,
\quad \text{determined by }  e_0 \mapsto 1 \in V_0.
\end{align}
\end{thm}

The involution $\barw$ preserves the left ideal  $\AA_t e_0$, i.e., it
restricts to a \(\QQ\)-linear involution  $\AA_t e_0 \to \AA_t e_0$ (not an  $\AA_t$-module homomorphism).
We let  $\barwb \colon \bigoplus_{r=0}^{\infty} y_1 \cdots y_r V_r
\to \bigoplus_{r=0}^{\infty} y_1 \cdots y_r V_r$ denote the corresponding involution obtained via Theorem \ref{t At iso V}.  It takes  $1 \in V_0$ to itself and
\begin{align}
\label{e barwprime}
\barwb \circ  L = \barw(L) \circ \barwb \quad \text{for all } L\in \AA_{t}\,.
\end{align}
It is determined by these two properties.

The next two results require the notation  $d^{CM}_-$, etc.
from \S\ref{ss two actions} for the Carlsson-Mellit \cite{CarMel} action of  $\AA_{t,q}$ on $V_*$,
and the isomorphism  $\Q_r \colon y_1 \cdots y_r V_r \to V_r,  \, y_1 \cdots y_r v \mapsto  v.$

\begin{thm}[{\cite[Theorem 7.4]{CarMel}}]
\label{t CarMel N}
There are $\QQ(q,t)$-antilinear maps $\N \colon V_r \to V_r$ for  $r\ge 0$
such that  $\N(1) = 1$ for $1 \in V_0$, and which satisfy
\begin{align}
\label{e CarMel N def}
\ \  \N \spa T^{CM}_i  = (T^{CM}_i)^{-1} \spa \N, \ \
 \N \spa d^{CM}_-  = d^{CM}_- \spa \N,
\ \ \N \spa d^{CM}_+ = d^{* CM}_+ \spa \N.
\end{align}
Moreover,  $\N \colon V_r \to V_r$ is an involution, i.e., $\N^2 = \idelm$.
\end{thm}

Let  $\N^\Q$ denote the collection of maps $\Q_r^{-1} \N \Q_r \colon y_1 \cdots y_r V_r \to y_1 \cdots y_r V_r$ for  $r \ge 0$.

\begin{lemma}
\label{l nablaprime vs CarMel N}
We have the following identities of operators on  $(y_1 \cdots y_r V_r)_{r\ge 0}$:
\begin{align}
\label{el nablaprime vs CarMel N 1}
\Mnab\barwb  &= \N^\Q,   \\
\label{el nablaprime vs CarMel N 2}
\barwb \Mnab^{-1}  &= \N^\Q.
\end{align}
\end{lemma}
\begin{proof}
Note that it is not yet clear that $\Mnab \barwb$ preserves the space $(y_1 \cdots y_r V_r)_{r\ge 0}$,
but since  $\N^\Q$ does, this will follow once we prove \eqref{el nablaprime vs CarMel N 1}
as an identity of maps from  $y_1 \cdots y_r V_r \to V_r$.
To prove this, we will compare how  $\N^\Q$ and $\Mnab \barwb $ intertwine the generators of  $\AA_t$.

By `conjugating' \eqref{e CarMel N def} by the maps  $\Q_r$, we have
$\N^\Q \Q_r^{-1} T^{CM}_i \Q_r = \Q_r^{-1} (T^{CM}_i)^{-1} \Q_r \spa \N^\Q$,
$\N^\Q \spa \Q_{r-1}^{-1} d^{CM}_-  \Q_r = \Q_{r-1}^{-1} d^{CM}_- \Q_r \spa \N^\Q$, and
$\N^\Q \spa \Q_{r+1}^{-1}d^{CM}_+ \Q_r = \Q_{r+1}^{-1}d^{* CM}_+ \Q_r^{-1} \spa \N^\Q.$
Proposition~\ref{CarMel vs Mel} then gives
\begin{align}
    \label{e N vs gens 1}
  \N^\Q \,  T_i = T_i^{-1} \N^\Q, \\
  \label{e N vs gens 2}
  \N^\Q \,  d_- = d_-  \, \N^\Q, \\
  \label{e N vs gens 3}
  \N^\Q \,  (-d_+) = y_1 d_+^* \, \N^\Q.
\end{align}

Now to address  $\Mnab \barwb$, \eqref{e N map} and \eqref{e barw map} directly give
\begin{align}
\label{e N vs gens 1b}
& \Mnab \barwb \,  T_i = T_i^{-1} \Mnab \barwb, \\
\label{e N vs gens 2b}
& \Mnab \barwb\,  d_- = d_-  \, \Mnab \barwb.
\end{align}
Further, computing with \eqref{e barwprime} and \eqref{e nablaprime and N}, we have
\begin{align}
  \Mnab \barwb \, d_+= -t^{-r} \Mnab d_+ \barwb=
  -t^{-r} N(d_+) \Mnab \barwb =
  t^{-r} (qt)^{-1} z_1 d_+ \Mnab \barwb,
\end{align}
where these are identities of operators on  $V_r$.
Using the relation \eqref{eq:Star to unstar} of  $\AA_{t,q}$,
which says that $z_1 d_+ = -qt^{r+1} y_1 d_+^*$, we obtain
\begin{align}
\label{e N vs gens 3b}
& \Mnab \barwb (-d_+) = y_1 d_+^* \Mnab \barwb.
\end{align}

Since the $d_+$'s, $d_-$'s, and  $T_i$'s generate
$\AA_t$, they also generate $(y_1 \cdots y_r V_r)_{r\ge 0}$ from  $1 \in V_0$ by Theorem \ref{t At iso V}.
Hence,~\eqref{el nablaprime vs CarMel N 1} follows from the matching intertwining relations
\eqref{e N vs gens 1}--\eqref{e N vs gens 3} and \eqref{e N vs gens 1b}--\eqref{e N vs gens 3b}.

Now that we know  $\Mnab \barwb = \N^\Q$ as operators on $y_1 \cdots y_r V_r$,
using that $\barwb$ is an involution on $y_1 \cdots y_r V_r$, we find
$\Mnab = \N^\Q \barwb$ as operators on $y_1 \cdots y_r V_r$.
Then since  $\N^\Q$ is also an involution on $y_1 \cdots y_r V_r$, $\Mnab$ is invertible as an operator on $y_1 \cdots y_r V_r$ with inverse
 $\Mnab^{-1} = \barwb \N^Q$.  The identity \eqref{el nablaprime vs CarMel N 2} follows.
\end{proof}

Recall that  $\rho_{m,n}^*$, $\rho$, and $\rho^*$ are the algebra homomorphisms from $ \AA_t \to \AA_{t,q}$ from Proposition-Definition~\ref{pd rho} and
Remark~\ref{rem:why ddpa}.

\begin{prop}
\label{p barw nabla}
For any  $L \in \AA_t$,
\begin{align}
\rho_{m,1}^{*}(L) = \barwb \Mnab^{-m} \,  L \, \Mnab^m \barwb
\end{align}
as operators on  $(y_1 \cdots y_r V_r)_{r \ge 0}$.
\end{prop}
Note that it is not clear that
$\rho_{m,1}^{*}(L)$
preserves  $(y_1 \cdots y_r V_r)_{r \ge 0}$, but this will follow from the proof since the right hand side does.

\begin{proof}
We have
\begin{align*}
  \barwb \Mnab^{-m} \,  L \, \Mnab^m \barwb
&= \Mnab^{m} \barwb \,  L \, \barwb \Mnab^{-m}  &&  \text{by  Lemma \ref{l nablaprime vs CarMel N}} \\
&= \Mnab^{m} \barw(L) \Mnab^{-m}  &&  \text{by \eqref{e barwprime}} \\
&=  N^m (\rho(\barw(L)))                      &&  \text{by \eqref{e nablaprime and N}} \\
&=  (N^{m-1} \circ S \circ \rho^{*})(L) && \text{by \cite[eq. (19)]{Mellit16}} \\
&=  \rho_{m,1}^{*}(L).                                  &&   \qedhere
\end{align*}
\end{proof}

\begin{thm}
\label{t umnab and alg}
The algebraic side of the signed nonsymmetric compositional $(km,kn)$-shuffle theorem for  $n=1$
(defined in \eqref{e def alg side})
can be expressed in terms of $\umnab$ and  $\barwb^\PP \defeq \PP_{r}^{-1} \barwb \PP_r$ as follows:
\begin{align}
\alg_{m,1}^\alpha
= (-1)^{k} (-t)^{r-k}
  \barwb^\PP \umnab^{-m} (x_1^{\alpha_1} \cdots x_r^{\alpha_r}).
\end{align}
\end{thm}
\begin{proof}
Using Proposition \ref{p barw nabla} followed by the
easily verified fact  $d_+^r \cdot 1 = (-1)^r y_1 \cdots y_r$,
\begin{multline}
  \rho_{m,1}^{*}(y_1^{\alpha_1-1} \cdots y_r^{\alpha_r-1} d_+^r) \cdot 1
=
\barwb \Mnab^{-m} \,  y_1^{\alpha_1-1} \cdots y_r^{\alpha_r-1} d_+^r \Mnab^m \barwb \cdot 1 \\
=(-1)^r \barwb \Mnab^{-m} \,  y_1^{\alpha_1} \cdots y_r^{\alpha_r}
=(-1)^r \barwb \Mnab^{-m}  \PP_r\, x_1^{\alpha_1} \cdots x_r^{\alpha_r}.
\end{multline}
Hence,
\begin{multline}
\alg_{m,1}^\alpha
=
(-1)^{k(m+2)} t^{r-k}
(-1)^r \PP_r^{-1} \barwb \Mnab^{-m}  \PP_r\, x_1^{\alpha_1} \cdots x_r^{\alpha_r}  \\
=
(-1)^{k(m+2)+r} t^{r-k}
\barwb^\PP (\Mnab^\PP)^{-m} \, x_1^{\alpha_1} \cdots x_r^{\alpha_r}
= (-1)^{k} (-t)^{r-k}
  \barwb^\PP \umnab^{-m} (x_1^{\alpha_1} \cdots x_r^{\alpha_r}),
\end{multline}
where we have used Theorem \ref{t two nablas} for the last equality.
\end{proof}

We prefer to have a nonsymmetric shuffle statement without the $\barwb^\PP$,
and the following result makes this possible.
\begin{prop}
\label{p flip LLT}
The operator  $\barwb^\PP$ takes signed flagged row LLT polynomials to signed flagged column LLT polynomials:
for $\pi \in \DD_r(d)$ and a marking $\Sigma$ of $\pi$,
\begin{align}
\barwb^\PP  \Growpm_r(\pi, \Sigma)(\xx;t)
= (-1)^d \Gcolpm_r(\pi, \Sigma)(\xx;t^{-1}).
\end{align}
\end{prop}
\begin{proof}
Comparing the definitions of $\chi^\PP_r(\pi, \Sigma)$ and $\tilde{\chi}^\PP_r(\pi, \Sigma)$
(Definitions \ref{d chi dplusminus} and \ref{d chi dplusminus col}), we see
$\barwb^\PP
\chi^\PP_r(\pi, \Sigma)
= (-1)^d \tilde{\chi}^\PP_r(\pi, \Sigma)$.
The result then follows from Theorems \ref{t Gcal vs chiPP} and \ref{t Gcal vs chiPP v2}.
\end{proof}

Using Theorem \ref{t umnab and alg} and Proposition \ref{p flip LLT} to rewrite
Theorem \ref{t ns shuffle unmod}, we obtain another version of the
signed nonsymmetric compositional $(km,kn)$-shuffle theorem for  $n=1$, which expresses $\umnab^{-m}$ of a monomial
in terms of signed flagged column LLT polynomials.

\begin{cor}
\label{c umnab shuffle}
Maintain the notation from the beginning of
\S\ref{ss Nonsymmetric compositional kmkn shuffle theorem}, \S \ref{ss Dyck path stats}, and \S \ref{ss flagged LLT polynomials for columns}.
Then for a positive integer $m$ and a strict composition $\alpha = (\alpha_1,\dots, \alpha_r)$
of size  $k$,
\begin{align}
\label{ec umnab shuffle}
(-t)^{r-k} \vartheta \umnab^{-m} (x_1^{\alpha_1} \cdots x_r^{\alpha_r}) =
\sum_{\pi \in \DD^\alpha_{km,k} }  q^{\area(\pi)} t^{\dinv(\pi)- \maxtdinv(\pi)}
\Gcolpm_r(\hat{\pi}', \Sigma_\pi),
\end{align}
where $\vartheta \colon \Pas \to \Pas$ is the  $\QQ(q,t)$-antilinear map
given by  $f(\xx;q,t) \mapsto f(\xx;q^{-1},t^{-1})$.
\end{cor}

We are now ready to prove Theorem \ref{t mod nsshuffle intro} from the introduction.
Recall that the nonsymmetric compositional Hall-Littlewood polynomial  $\nsC_\alpha$ is given by
\begin{align}
\nsC_\alpha =
(-t)^{|\alpha|-r} \Pisf_r( x_1^{\alpha_1} \cdots x_r^{\alpha_r} )
=
 (-t)^{|\alpha|-r} \pol
\Big(\frac{x_1^{\alpha_1} \cdots x_r^{\alpha_r}}{\prod_{1 \le i < j \le r}(1-t x_i/x_j)}  \Big)
\ \in \QQ(t)[\xx_r] \subseteq \P(r).
\end{align}
It is natural to allow  $\alpha\in \ZZ^r$ in the definition, but we will only use the case  $\alpha$ is a strict composition.
These are related to the compositional Hall-Littlewood polynomials by Weyl symmetrization:
$\Weyl_1 \cdots \Weyl_r \nsC_\alpha =C_{\alpha}(\xx;t^{-1})$. This follows from, e.g.,~\cite[Proposition 4.6]{BMPS19}.

Define a \emph{flagged word parking function} $w$ on $\pi \in \DD^\alpha_{km,kn}$
to be a map from the north steps of  $\pi$ to  $\ZZ_+$ which is strictly increasing going north along vertical runs and such that the entry on the north step from the point
$(m\sum_{i < j}\alpha_i, n \sum_{i < j}\alpha_i)$ is  $\le j$, for  $j = 1,\dots, r$.
Denote the set of these by  $\FWPF_\alpha(\pi)$.

For  $w \in \FWPF_\alpha(\pi)$, define the \emph{diagonal inversion statistic} by
\begin{align}
\label{e dinv HMZ def}
\dinv(\pi, w) = \big|\big\{ \text{pairs of north steps  $i$, $j$ of  $\pi$} : \text{ $i$ attacks   $j$ and $w(i) < w(j)$}  \big\}  \big|.
\end{align}

\begin{cor}
\label{c mod nsshuffle}
With notation as in Corollary \ref{c umnab shuffle} and just above,
\begin{align}
\label{ec mod nsshuffle1}
 \vartheta \modnab^{-m} (\nsC_\alpha)
& =
\sum_{\pi \in \DD^\alpha_{km,k} }  q^{\area(\pi)} t^{\dinv(\pi)- \maxtdinv(\pi)}
\Gcol_r(\hat{\pi}', \Sigma_\pi) \\
\label{ec mod nsshuffle2}
& =
\sum_{\pi \in \DD^\alpha_{km,k} } \sum_{w \in \FWPF_\alpha(\pi)} q^{\area(\pi)} t^{\dinv(\pi)-\maxtdinv(\pi)+\dinv(\pi,w)}
 \xx^{\content(w)}.
\end{align}
\end{cor}
\begin{proof}
Applying $\vartheta \spa \Pisf_r \spa \vartheta$ to both sides of \eqref{ec umnab shuffle} and using
\eqref{e signed to unsigned cols} gives \eqref{ec mod nsshuffle1}.
The second identity \eqref{ec mod nsshuffle2} then follows from
$\Gcol_r(\hat{\pi}', \Sigma_\pi) =
\sum_{w \in \FWPF_\alpha(\pi)} t^{\dinv(\pi,w)} \xx^{\content(w)}$,
which comes from
converting the
definition of  $\Gcol_r(\hat{\pi}', \Sigma_\pi)$ to the language of word parking functions, using the correspondence between attacking pairs of north steps of $\pi$ and  $\Area(\pi')$ from \S\ref{ss Dyck path stats}.
\end{proof}

Setting  $m=1$, we obtain Theorem \ref{t mod nsshuffle intro} once we check that the statistics simplify in this case.
\begin{prop}
For a $(k,k)$-Dyck path $\pi$, $\dinv(\pi) = \maxtdinv(\pi)$.
\end{prop}
\begin{proof}
For each east step  $e$ of  $\pi$, let  $\theta(e)$ be the first north step of  $\pi$ which is hit when following a line of slope 1 down from the midpoint of $e$.  One checks that for a pair $(e,f)$ with
$e$ an east step of  $\pi$ and  $f$ a north step of  $\pi$,  $(e,f)$ counts towards  $\dinv(\pi)$ if and only if $\theta(e)$ attacks  $f$ (as in \S\ref{ss Dyck path stats}).
The result then follows from \eqref{e d maxtdinv}.
\end{proof}

It is instructive to see how Corollary \ref{c mod nsshuffle} with $m=1$ Weyl symmetrizes
to the compositional shuffle theorem
conjectured by Haglund-Morse-Zabrocki \cite{HagMoZa12} and proven by Carlsson-Mellit \cite{CarMel}.

First note that we have the following identity of symmetric function operators:
\begin{align}
\label{e omega nab identity}
\vartheta \spa \omega \spa \nabla^{-1} \spa \omega \spa \vartheta  = \nabla,
\end{align}
which follows from $\omega \Htild_\mu(\xx;q,t)=q^{\mathsf{n}(\mu ^*)} t^{\mathsf{n}(\mu )} \Htild_\mu(\xx;q^{-1},t^{-1})$.
Then using \eqref{e modnab and sym},  $\Weyl_1 \cdots \Weyl_r \nsC_\alpha = C_{\alpha}(\xx;t^{-1})$, and \eqref{e omega nab identity}, we have
\begin{multline}
\Weyl_1 \cdots \Weyl_r \spa \vartheta \spa \modnab^{-m} (\nsC_\alpha) =
\vartheta \spa \Weyl_1 \cdots \Weyl_r \spa \modnab^{-m} (\nsC_\alpha) =
\vartheta \spa \omega \nabla^{-m} \omega \spa \Weyl_1 \cdots \Weyl_r \nsC_\alpha \\
= \vartheta \spa \omega \nabla^{-m} \omega \spa C_{\alpha}(\xx;t^{-1})
=
\nabla^m \vartheta \spa C_{\alpha}(\xx;t^{-1}) =
\nabla^m C_{\alpha}(\xx;t),
\end{multline}
while by \eqref{ep sym LLT col 2},
\begin{align}
\label{e sigma LLT}
\Weyl_1 \cdots \Weyl_r \, \Gcol_r(\hat{\pi}', \Sigma_\pi)
=\Gcol_0(\pi', \Sigma_\pi)
=\Gcal_{\etabold}(\xx;t),
\end{align}
where  $\Gcal_{\etabold}(\xx;t)$ is an ordinary LLT polynomial indexed by the tuple of single-column skew shapes $\etabold$,
determined from  $\hat{\pi}', \Sigma_\pi$ as in Remark \ref{r col two LLT match}.
Thus applying $\Weyl_1 \cdots \Weyl_r$ to both sides of
\eqref{ec mod nsshuffle1} yields
\begin{align}
\label{e recover classical}
 \nabla^m C_{\alpha}(\xx;t)
&=
\sum_{\pi \in \DD^\alpha_{km,k} }
\! q^{\area(\pi)} t^{\dinv(\pi)-\maxtdinv(\pi)}
\Gcol_0(\pi', \Sigma_\pi) \\
&= \! \sum_{\pi \in \DD^\alpha_{km,k} }
\sum_{\substack{\text{word parking} \\ \text{functions  $w$ of  $\pi$}}}
\!\! q^{\area(\pi)} t^{\dinv(\pi)-\maxtdinv(\pi)+\dinv(\pi,w)}
 \xx^{\content(w)}, \!
\end{align}
where the word parking functions on  $\pi$ are as in \eqref{e intro comp shuffle} (the same as the flagged word parking functions defined above but without the flag bounds).
After swapping  $q$ and  $t$, this statement for  $m=1$ is the same as the statement of the compositional shuffle conjecture in \cite[Conjecture 4.5]{HagMoZa12}.

\subsection{Atom positivity conjectures}
\label{ss atom pos 2}
We discuss atom positivity conjectures for several objects in this paper.
This ties into the discussion in \S\ref{ss atom pos} and provides context and motivation for
several objects on the modified side of Figure \ref{fig:summary diagram} including $\modnab^{-1} \nsC_\alpha$ and the modified $r$-nonsymmetric Macdonald polynomials  $\tE_{\eta|\lambda}$.

Recall that  $\xx= x_1,x_2, \dots, $ and  $\xx_N = x_1,\dots, x_N$, and that
for $f(\xx) \in \P(r)$, we write  $f[\xx_N]$ for $f(\xx)|_{x_{N+1} = x_{N+2} = \cdots = 0}$.

A typical element of  $\P(r)$ is an infinite sum of monomials and does not lie in
the span of the Demazure atoms, so we need stable versions of Demazure atoms which span $\P(r)$.
For  $(\eta | \lambda) \in \pairsr$ with  $k = \ell(\lambda)$, define the \emph{stable atom} by
\begin{align}
\Acal_{\eta| \lambda} = \Acal_{\eta| \lambda}(\xx) = \Weyl_{r+1} \cdots \Weyl_{r+k} \, \Acal_{(\eta; \lambda)}(\xx_{r+k}) \, \in \P(r),
\end{align}
where  $\Acal_{(\eta; \lambda)}$ is the Demazure atom (defined in \S\ref{ss atom pos}) indexed by the concatenation of  $\eta$ and  $\lambda$.
The stable atoms form a basis for  $\P(r)$ as $(\eta | \lambda)$ ranges over $\pairsr$.

The evaluations of stable atoms in a finite number of variables are related to Demazure atoms as follows.
For $(\eta|\lambda) \in \pairsr$ and  $N \ge r+\ell(\lambda)$, we have
\begin{align}
\label{e stable atom finite}
\Acal_{\eta| \lambda}[\xx_N] =
\pi_{x_{r+1}, \dots,x_N} \Acal_{(\eta; \lambda; 0^d)}(\xx_N) =
\sum_{w \in \SS_{(1^r;N-r)}}
 \Acal_{w(\eta; \lambda; 0^d)}(\xx_N),
\end{align}
where $\pi_{x_{r+1}, \dots,x_N}$ is the operator  $\pi_w$ (see \S \ref{ss atom pos})
for  $w$ the longest permutation in $\SS_{(1^{r}; N-r)}$,
and $d = N - r - \ell(\lambda)$.

\begin{remark}
By the first equality of \eqref{e stable atom finite}, the stable atom  $\Acal_{\eta|\lambda}(\xx)$ is also equal to the strongly convergent limit
$\lim_{N\to \infty }\pi_{x_{r+1}, \dots,x_N} \Acal_{(\eta; \lambda;0^d)}(\xx_N)$.  This is the definition of stable atom used in
\cite[\S 7.6]{BHMPS-nspleth}, and so our notions of stable atoms agree.
\end{remark}

\begin{remark}
\label{r stable atom vs atom}
(i)
Any $f(\xx) \in  \P(r)$ of degree  $d$ can be recovered from
$f[\xx_N]$ for  $N \ge d+r$ via  $f(\xx) = \Weyl_{r+1}\cdots\Weyl_{N} f[\xx_N]$.
This follows, for instance, by writing  $f$ in terms of the basis
$\big\{x_1^{\eta_1} \cdots x_r^{\eta_r} s_\lambda(x_{r+1},x_{r+2}, \dots ) : (\eta|\lambda) \in \pairsr\big\}$
of  $\P(r)$.

(ii) Stable atom positivity essentially reduces to the more familiar notion of atom positivity:
for any $f(\xx)\in \P(r)$ of degree  $d$ and $N \ge d+r$,
$f(\xx)$ is stable atom positive if and only if
$f[\xx_N]$ is atom positive.
This follows from (i), \eqref{e stable atom finite}, and the fact that any $g \in \QQ(q,t)[\xx_N]$ which is symmetric in
 $x_{r+1}, \dots, x_N$ has atom expansion  $g = \sum_\alpha c_\alpha \Acal_\alpha$ with the coefficients $c_\alpha$ constant
 on  $\SS_{(1^r;N-r)}$ orbits.
\end{remark}

Consider the full Weyl symmetrization operator   $\Weyl_1 \cdots \Weyl_r$ from  $\P(r)$ to  $\P(0)$.
By \cite[Lemma 7.8.2]{BHMPS-nspleth}, it acts on stable atoms by
\begin{align}
\label{e sigmabold atom}
\Weyl_1 \cdots \Weyl_r \, \Acal_{\eta| \lambda}(\xx) =
\begin{cases}
s_{(\eta; \lambda)}(\xx) &  \text{if  $(\eta; \lambda)$ is weakly decreasing}, \\
0           &      \text{otherwise}.
\end{cases}
\end{align}
Hence, similar to \eqref{e atom key pos},
we have for $f \in \P(r)$,
\begin{align}
\label{e stable atom key pos}
\text{$f$ is stable atom positive} \implies
\text{$\Weyl_1 \cdots \Weyl_r f$ is Schur positive},
\end{align}
and so stable atom positivity is a natural strengthening of Schur positivity to the space $\P(r)$.

We conjecture that the following elements of  $\P(r)$ on the modified side have coefficients in  $\NN[q,t]$ when expanded in the basis of stable atoms.
Although part (i) of the conjecture implies (ii)--(iv), we have stated these separately in order to emphasize that each could potentially have its
own representation-theoretic significance.

\begin{conj}
(i) The flagged LLT polynomials  $\flagGcal_{r, \nubold}(\xx;t)$ from \S\ref{ss Dpath vs shapes llt} are stable atom positive.
This includes in particular  $\Grow_r(\pi, \Sigma)(\xx;t)$ and
 $\Gcol_r(\pi, \Sigma)(\xx;t)$ from
Definitions \ref{d flag LLT} and \ref{d col LLT}.

(ii) The modified $r$-nonsymmetric Macdonald polynomials  $\tE_{\eta|\lambda}$ are stable atom positive.

(iii) The quantity $\modnab^{-1} \nsC_\alpha$ from Theorem \ref{t mod nsshuffle intro} is stable atom positive.

(iv) The quantity $\Pisf_r \, \alg_{m,n}^\alpha$ from Theorem \ref{t ns shuffle kmkn mod} is stable atom positive.
\end{conj}

Parts (i) and (ii) were formulated in \cite{BHMPS-nspleth},
where we also showed that (i) implies (ii) by
giving a formula for $\tE_{\eta|\lambda}$
as a positive sum of flagged LLT polynomials for tuples of ribbon skew shapes.
Note that by \eqref{e stable atom key pos} and Proposition \ref{p sym modnsmac}, (ii) is a conjectural strengthening of Macdonald positivity.
By Theorems \ref{t mod nsshuffle intro} and \ref{t ns shuffle kmkn mod},
(i) implies (iii) and (iv).

\subsection{Examples}
\label{ss examples}

\begin{remark}
\label{rem: small length enough}
(i) The following improvement on Remark \ref{r stable atom vs atom} (i)--(ii) is useful for explicit computations:
the degree  $d$ subspace of $\P(r)$ has a filtration  $(\P(r)_{d_1,d_2})_{d_1,d_2 \in \NN, \, d_1+d_2=d}$
given by
\begin{align*}
\P(r)_{d_1,d_2}  \defeq & \,
\spn_{\QQ(q,t)} \big\{x_1^{\eta_1} \cdots x_r^{\eta_r} s_\lambda(x_{r+1},x_{r+2}, \dots ) : (\eta| \lambda) \in \NN^r \times \Par, \, |\eta| \ge d_1, \, |\eta| + |\lambda| = d \big\}
 \\
 = & \,
\spn_{\QQ(q,t)} \big\{x_1^{\eta_1} \cdots x_r^{\eta_r} s_\lambda(\xx) : (\eta| \lambda) \in \NN^r \times \Par, \, |\eta| \ge d_1, \, |\eta| + |\lambda| = d \big\}.
\end{align*}
Any $f(\xx) \in  \P(r)_{d_1,d_2}$ can be recovered from
 $f[\xx_N]$ for  $N \ge d_2+r$ via
 $f(\xx) = \Weyl_{r+1}\cdots\Weyl_{N} f[\xx_N]$.
The subspace  $\P(r)_{d_1,d_2}$ is spanned by the subset of stable atoms indexed by  $(\eta|\lambda)$ with  $|\eta| \ge d_1$ and  $|\eta|+|\lambda| = d_1+d_2$.
Hence, as in Remark \ref{r stable atom vs atom} (ii),
for any $f(\xx)\in \P(r)_{d_1, d_2}$ and $N \ge d_2+r$,
$f(\xx)$ is stable atom positive if and only if
$f[\xx_N]$ is atom positive.

(ii)
Given (i), the following additional fact is useful for explicit computations:
for any $\pi \in \DD_r(k)$ and marking $\Sigma$ of $\pi$,
\begin{align}
\label{e deg bound}
\Gcol_r(\pi, \Sigma)(\xx;t) \in  \P(r)_{r,k-r}.
\end{align}
In particular, each LLT term in \eqref{ec mod nsshuffle1}
(and therefore  $\modnab^{-m} \nsC_\alpha$ as well) is determined by its evaluation in only the variables $x_1, \dots, x_k$ ($k = |\alpha|$ is the degree of all these terms).
To prove \eqref{e deg bound}, first note that the signed flagged column LLT polynomials have a formula in terms of non-attacking fillings analogous
to Lemma \ref{l nonattack}, by \cite[Proposition 6.7.9]{BHMPS-nspleth}. Hence, $x_1 \cdots x_r$ divides
$\Gcolpm_r(\pi, \Sigma)(\xx;t)$, and so  $\Gcolpm_r(\pi, \Sigma)(\xx;t) \in \P(r)_{r,k-r}$.
Then apply  $\Pisf_r$ to obtain \eqref{e deg bound}, noting that $\Pisf_r$ restricts to an isomorphism from $\P(r)_{d_1, d_2}$ to itself,
which is straightforward to check from the definition \eqref{e: nspleth def}.
\end{remark}

\begin{example}
\label{ex main ex 1}
We illustrate Corollary \ref{c mod nsshuffle} for $\alpha = (3)$ and $m=1$ (so  $r=1$,  $k=3$).
There are two  $(3,3)$-Dyck paths with touchpoint sequence  $(3)$. Denote these paths by
$\pi$ and $\theta$, and their corresponding attacking
Dyck paths $\pi'$, $\theta'$ and markings by $\Sigma_{\pi}$, $\Sigma_\theta$;
let $\hat{\pi}'$, $\hat{\theta}'$ denote the partial Dyck paths obtained by removing the  $r$ bottommost north steps from  ${\pi}'$, ${\theta}'$, respectively (see \S\ref{ss Dyck path stats}).  This data is shown below.
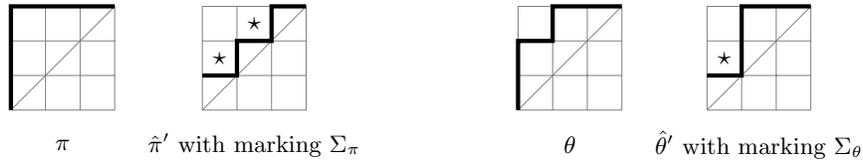
\[
\begin{tikzpicture}[scale=.46,baseline=.6cm]
\draw[help lines] (0,0) grid (3,3);
\draw[very thin, gray] (0,0) -- (3,3);
\draw[ultra thick] (0,0) -- (0,1) -- (0,2) -- (0,3) -- (1,3) -- (2,3) -- (3,3);
\node at (1.5, -1) {{\scriptsize $\pi$}};
\end{tikzpicture}
\ \
\begin{tikzpicture}[scale=.46,baseline=.6cm]
\draw[help lines] (0,0) grid (3,3);
\draw[very thin, gray] (0,0) -- (3,3);
\draw[ultra thick] (0,1) -- (1,1) -- (1,2) -- (2,2) -- (2,3) -- (3,3);
\node at (.5, 1.5) {{\footnotesize $\star$}};
\node at (1.5, 2.5) {{\footnotesize $\star$}};
\node at (1.5, -1) {{\scriptsize $\hat{\pi}'$ with marking $\Sigma_\pi$}};
\end{tikzpicture}
\qquad \qquad \ \
\begin{tikzpicture}[scale=.46,baseline=.6cm]
\draw[help lines] (0,0) grid (3,3);
\draw[very thin, gray] (0,0) -- (3,3);
\draw[ultra thick] (0,0) -- (0,1) -- (0,2) -- (1,2) -- (1,3) -- (2,3) -- (3,3);
\node at (1.5, -1) {{\scriptsize $\theta$}};
\end{tikzpicture}
\  \
\begin{tikzpicture}[scale=.46,baseline=.6cm]
\draw[help lines] (0,0) grid (3,3);
\draw[very thin, gray] (0,0) -- (3,3);
\draw[ultra thick] (0,1) -- (1,1) -- (1,2) --(1,3)--(2,3) -- (3,3);
\node at (.5, 1.5) {{\footnotesize $\star$}};
\node at (1.5, -1) {{\scriptsize $\hat{\theta}'$ with marking $\Sigma_\theta$}};
\end{tikzpicture}
\]
For the present example, identity \eqref{ec mod nsshuffle1} reads
\begin{align}
\label{e example key expansion 3a}
 \vartheta \modnab^{-1} \nsC_{(3)} =
\vartheta \modnab^{-1} (t^{2}x_1^3)
=q^{3}\Gcol_1(\hat{\pi}', \Sigma_\pi) +
q^{2}\Gcol_1(\hat{\theta}', \Sigma_\theta).
\end{align}
These flagged column LLT polynomials expand into stable atoms as follows:
\begin{align}
\label{e ex llt1}
\Gcol_1(\hat{\pi}', \Sigma_\pi)
& \ = \ \Acal_{1|11} \\
\label{e ex llt2}
\Gcol_1(\hat{\theta}', \Sigma_\theta)
& \ =  \  \Acal_{2|1}+ {\Acal}_{1|2}   +  t \, {\Acal}_{1|11}.
\end{align}
To see this explicitly, we display
the associated flagged  $(\hat{\pi}', \Sigma_\pi)$ and  $(\hat{\theta}', \Sigma_\theta)$-column semistandard words, with letters $\le k = 3$,
and using the same conventions as Example \ref{ex row flag llt}.
Computing in variables  $x_1,\dots, x_k$ is sufficient by Remark \ref{rem: small length enough}.

\newcommand{\thisllta}[3]{
\tikz[scale=.34, baseline=.4cm]{
\draw[help lines] (0,0) grid (3,3);
\draw[thick] (0,1) -- (1,1) -- (1,2) -- (2,2) -- (2,3) -- (3,3);
\node at (.5, 1.5) {{\tiny $\star$}};
\node at (1.5, 2.5) {{\tiny $\star$}};
\node at (.5, .5) {\tiny #1};
\node at (1.5, 1.5) {\tiny #2};
\node at (2.5, 2.5) {\tiny #3}; }
}
\newcommand{\thisfwpfa}[3]{
\begin{tikzpicture}[scale=.34,baseline=.4cm]
\draw[help lines] (0,0) grid (3,3);
\draw[thick] (0,0) -- (0,1) -- (0,2) -- (0,3) -- (1,3) -- (2,3) -- (3,3);
\node at (.5, .5)   {\tiny #1};
\node at (.5, 1.5) {\tiny #2};
\node at (.5, 2.5) {\tiny #3};
\end{tikzpicture}
}
\newcommand{\thisllt}[3]{
\tikz[scale=.34,baseline=.4cm]{
\draw[help lines] (0,0) grid (3,3);
\draw[thick] (0,1)--(1,1)--(1,2)--(1,3)--(2,3)--(3,3);
\node at (.5, 1.5) {{\tiny $\star$}};
\node at (.5, .5) {\tiny #1};
\node at (1.5, 1.5) {\tiny #2};
\node at (2.5, 2.5) {\tiny #3};
\ifnum#2<#3
  \node at (1.5, 2.5) {\tiny \(t\)};
\fi }
}

\newcommand{\thisfwpf}[3]{
\begin{tikzpicture}[scale=.34,baseline=.4cm]
\draw[help lines] (0,0) grid (3,3);
\draw[thick] (0,0) -- (0,1) -- (0,2) -- (1,2) -- (1,3) -- (2,3) -- (3,3);
\node at (.5, .5)   {\tiny #1};
\node at (.5, 1.5) {\tiny #2};
\node at (1.5, 2.5) {\tiny #3};
\end{tikzpicture}
}
\noindent\begin{tikzpicture}[ampersand replacement=\&]
\matrix[matrix of math nodes, nodes = {anchor=center}, row sep=0em] (m) {
{\text{$\substack{T \in \FCSSYT(\hat{\pi}', \Sigma_\pi) \, :\\ T(i) \le k}$}}
\&
\thisllta{1}{2}{3}
\&[.5cm]
{\text{$\substack{T \in \FCSSYT(\hat{\theta}', \Sigma_\theta) \, :\\ T(i) \le k}$}}
\&
\thisllt{1}{2}{1}
\&
\thisllt{1}{3}{1}
\&
\thisllt{1}{2}{2}
\&
\thisllt{1}{3}{2}
\&
\thisllt{1}{3}{3}
\&
\thisllt{1}{2}{3}
\\
t^{\widetilde{\inv}(T)} \xx^T \&  x_1x_2x_3 \&
t^{\widetilde{\inv}(T)}\xx^T  \& x_1^2x_2 \& x_1^2 x_3\& x_1 x_2^2\&
x_1x_2x_3 \& x_1 x_3^2               \& t \spa x_1x_2x_3  \\[1.2mm]
\text{\scriptsize $\Acal_{\eta|\lambda}[\xx_k]$ expansion} \&      \text{\small $\Acal^{[k]}_{1|11}$} \& \text{\scriptsize $\Acal_{\eta|\lambda}[\xx_k]$ expansion}
\&
\node (a) {};
\&
\node (b) {};
\&
\&
\text{\small $\Acal^{[k]}_{1|2}$}
\&
\&
t \, \text{\small $\Acal^{[k]}_{1|11}$} \\
};
\node at ($(a)!0.5!(b)$) {\text{\small $\Acal^{[k]}_{2|1}$}};
\draw[decorate,decoration={brace,mirror,amplitude=4pt}] (m-2-4.south west)
-- (m-2-5.south east);
\draw[decorate,decoration={brace,mirror,amplitude=4pt}] (m-2-6.south west)
-- (m-2-8.south east);
\draw (m-3-1.east |- m.south) ++(0,0.25) -- (m-3-1.east |- m-1-2.north);
\draw (m-3-3.east |- m.south) ++(0,0.25) -- (m-3-3.east |- m-1-4.north);
\end{tikzpicture}
Here we have abbreviated  $\Acal_{\eta|\lambda}[\xx_k]$ by  $\Acal^{[k]}_{\eta|\lambda}$.
Note that each finite evaluation  $\Acal_{\eta|\lambda}[\xx_k]$ is a sum of Demazure atoms by \eqref{e stable atom finite};
for example,  $\Acal^{[k]}_{1|2} = \Acal_{120} + \Acal_{102} = (x_1x_2^2) + (x_1x_2x_3 + x_1x_3^2)$.


We can then recover the corresponding instance of the (symmetric) compositional shuffle theorem
by applying the Weyl symmetrization operator  $\Weyl_1 \cdots \Weyl_r = \Weyl_1$ to \eqref{e example key expansion 3a},
 as we did in general in \eqref{e recover classical}.
Using \eqref{e sigmabold atom} yields
\begin{align}
\nabla C_{(3)}
= q^3 s_{111}  + q^2 ( s_{21}   +  t \, s_{111} ),
\end{align}
with  $s_{111}$ and  $s_{21}   +  t \, s_{111}$ being the ordinary LLT polynomials corresponding to
the tuples of single-column skew shapes  $\etabold = \big((111)\big)$ and  $\etabold = \big((1), (22)/(11))$, respectively, as in \eqref{e sigma LLT}.
\end{example}

\begin{example}
\label{ex main ex 2}
We illustrate Corollary \ref{c umnab shuffle} for $\alpha = (3,1)$ and $m=1$ (so  $r=2$,  $k=4$).
There are two  $(4,4)$-Dyck paths with touchpoint sequence  $(3,1)$.
We use notation  $\pi, \theta$,  $\Sigma_\pi, \Sigma_\theta$, etc. for these paths, markings, etc.
as in the previous example.
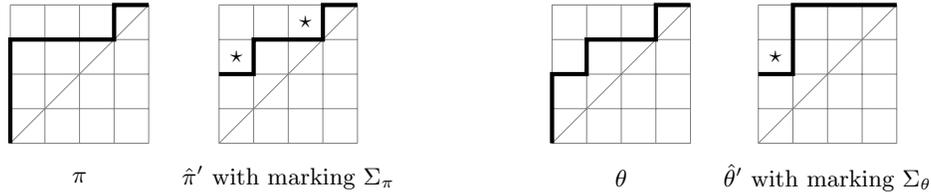
\[
\begin{tikzpicture}[scale=.46,baseline=.6cm]
\draw[help lines] (0,0) grid (4,4);
\draw[very thin, gray] (0,0) -- (4,4);
\draw[ultra thick] (0,0) -- (0,1) -- (0,2) -- (0,3) -- (1,3) -- (2,3) -- (3,3) -- (3,4) -- (4,4);
\node at (2, -1) {{\scriptsize $\pi$}};
\end{tikzpicture}
\ \
\begin{tikzpicture}[scale=.46,baseline=.6cm]
\draw[help lines] (0,0) grid (4,4);
\draw[very thin, gray] (0,0) -- (4,4);
\draw[ultra thick] (0,2) -- (1,2) -- (1,3) -- (2,3) -- (3,3) -- (3,4) -- (4,4);
\node at (.5, 2.5) {{\scriptsize $\star$}};
\node at (2.5, 3.5) {{\scriptsize $\star$}};
\node at (2, -1) {{\scriptsize $\hat{\pi}'$ with marking $\Sigma_\pi$}};
\end{tikzpicture}
\qquad \qquad \ \
\begin{tikzpicture}[scale=.46,baseline=.6cm]
\draw[help lines] (0,0) grid (4,4);
\draw[very thin, gray] (0,0) -- (4,4);
\draw[ultra thick] (0,0) -- (0,1) -- (0,2) -- (1,2) -- (1,3) -- (2,3) -- (3,3) -- (3,4) -- (4,4);
\node at (2, -1) {{\scriptsize $\theta$}};
\end{tikzpicture}
\  \
\begin{tikzpicture}[scale=.46,baseline=.6cm]
\draw[help lines] (0,0) grid (4,4);
\draw[very thin, gray] (0,0) -- (4,4);
\draw[ultra thick] (0,2) -- (1,2) -- (1,3) --(1,4)--(2,4) -- (3,4) -- (4,4);
\node at (.5, 2.5) {{\scriptsize $\star$}};
\node at (2, -1) {{\scriptsize $\hat{\theta}'$ with marking $\Sigma_\theta$}};
\end{tikzpicture}
\]

Identity \eqref{ec mod nsshuffle1} in this case reads
\begin{align}
\label{e example AK expansion}
& \vartheta \modnab^{-1} \nsC_{(3,1)} =
\vartheta \modnab^{-1} (t^{2} x_1^3x_2 + t^3 x_1^4)
=q^{3}\Gcol_2(\hat{\pi}', \Sigma_\pi) +
q^{2}\Gcol_2(\hat{\theta}', \Sigma_\theta).
\end{align}
These flagged column LLT polynomials expand into stable atoms as follows:
\begin{align*}
\Gcol_2(\hat{\pi}', \Sigma_\pi) \  =  \ \, &
t \spa \Acal_{21|1} + t \spa \Acal_{12|1} + t \spa \Acal_{20|11} + t^2 \Acal_{11| 11}
\\[1.4mm]
\Gcol_2(\hat{\theta}', \Sigma_\theta) \ = \ \, &
 t \spa \Acal_{31|\varnothing} +  t\spa \Acal_{13|\varnothing} + t \spa \Acal_{30|1}+
 t^2 \spa \Acal_{22|\varnothing} +
 t^2 \spa \Acal_{21|1} +
 t^2 \spa \Acal_{12|1} + t^2 \spa \Acal_{20|2}  \\
 &  t^3 \spa \Acal_{21|1} +
 t^3 \spa \Acal_{12|1} + t^3 \spa \Acal_{20|11} + t^3 \spa \Acal_{11|2} + t^4 \spa \Acal_{11|11}.
\end{align*}
We show the flagged
$(\hat{\theta}', \Sigma_\theta)$-column semistandard words for  $\Gcol_2(\hat{\theta}', \Sigma_\theta)(x_1,\dots, x_4;t)$
in the same conventions as the previous example.

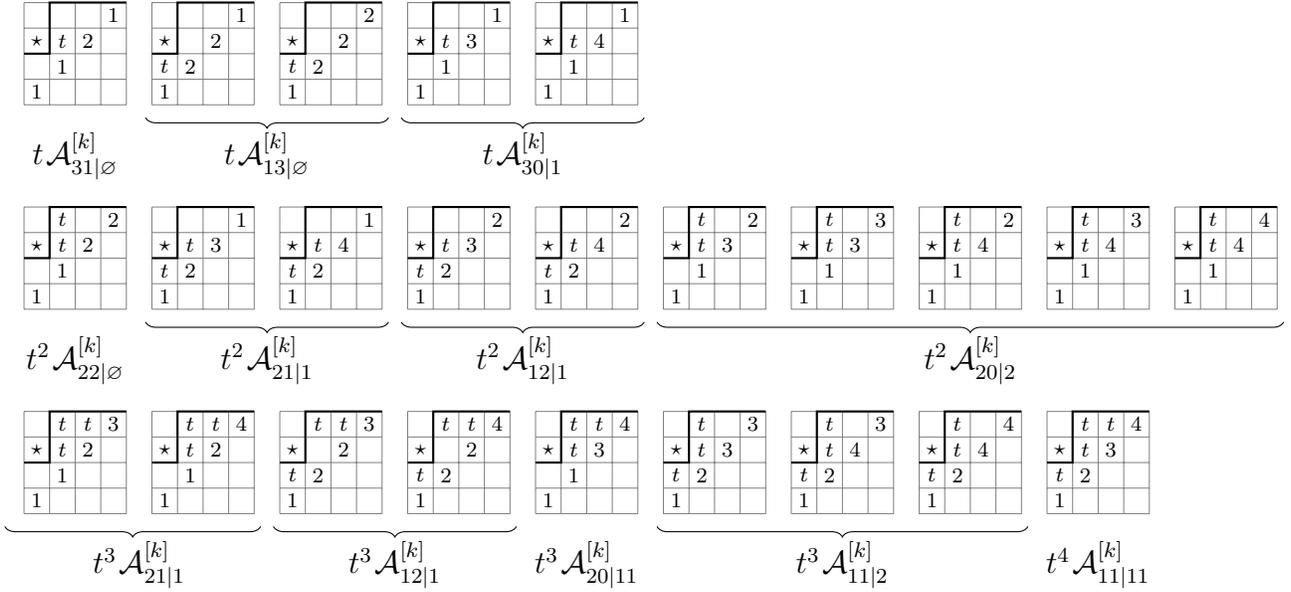
\newcommand{\thisllt}[4]{
\tikz[scale=.34,baseline=.6cm]{
\draw[help lines] (0,0) grid (4,4);
\draw[thick] (0,2) -- (1,2) -- (1,4) -- (4,4);
\node at (.5, 2.5) {{\tiny $\star$}};
\node at (.5, .5) {\tiny #1};
\node at (1.5, 1.5) {\tiny #2};
\node at (2.5, 2.5) {\tiny #3};
\node at (3.5, 3.5) {\tiny #4};
\ifnum#1<#2
    \node at (.5,1.5) {\tiny \(t\)};
\fi
\ifnum#2<#3
    \node at (1.5,2.5) {\tiny \(t\)};
\fi
\ifnum#2<#4
    \node at (1.5,3.5) {\tiny \(t\)};
\fi
\ifnum#3<#4
    \node at (2.5,3.5) {\tiny \(t\)};
\fi
}
}

\noindent\begin{tikzpicture}[scale=.34,baseline=.6cm]
\pgfmathsetmacro{\rowone}{21}
\pgfmathsetmacro{\rowtwo}{13}
\pgfmathsetmacro{\rowthree}{5}
\pgfmathsetmacro{\colspa}{5}
\useasboundingbox (-2,\rowthree-5) rectangle (9.45*\colspa,\rowone+2.5);
\node at (0,\rowone) {\thisllt{1}{1}{2}{1}};
\node at (0,\rowone-4) {\(t \spa \Acal^{[k]}_{31|\varnothing}\)};
\node at (1*\colspa,\rowone) {\thisllt{1}{2}{2}{1}};
\node at (2*\colspa,\rowone) {\thisllt{1}{2}{2}{2}};\draw[decorate,decoration={brace,mirror,amplitude=4pt}] (0.55*\colspa,\rowone-2.5) -- (2.45*\colspa,\rowone-2.5);
\node at (1.5*\colspa,\rowone-4) {\(t \spa \Acal^{[k]}_{13|\varnothing}\)};
\node at (3*\colspa,\rowone) {\thisllt{1}{1}{3}{1}};
\node at (4*\colspa,\rowone) {\thisllt{1}{1}{4}{1}};
\draw[decorate,decoration={brace,mirror,amplitude=4pt}] (2.55*\colspa,\rowone-2.5) -- (4.45*\colspa,\rowone-2.5);
\node at (3.5*\colspa,\rowone-4) {\(t \spa \Acal^{[k]}_{30|1}\)};
%
%
\node at (0*\colspa,\rowtwo) {\thisllt{1}{1}{2}{2}};
\node at (0*\colspa,\rowtwo-4) {\(t^2 \spa \Acal^{[k]}_{22|\varnothing}\)};
\node at (1*\colspa,\rowtwo) {\thisllt{1}{2}{3}{1}};
\node at (2*\colspa,\rowtwo) {\thisllt{1}{2}{4}{1}};
\draw[decorate,decoration={brace,mirror,amplitude=4pt}] (0.55*\colspa,\rowtwo-2.5) -- (2.45*\colspa,\rowtwo-2.5);
\node at (1.5*\colspa,\rowtwo-4) {\(t^2\spa \Acal^{[k]}_{21|1}\)};
\node at (3*\colspa,\rowtwo) {\thisllt{1}{2}{3}{2}};
\node at (4*\colspa,\rowtwo) {\thisllt{1}{2}{4}{2}};
\draw[decorate,decoration={brace,mirror,amplitude=4pt}] (2.55*\colspa,\rowtwo-2.5) -- (4.45*\colspa,\rowtwo-2.5);
\node at (3.5*\colspa,\rowtwo-4) {\(t^2\spa \Acal^{[k]}_{12|1}\)};
\node at (5*\colspa,\rowtwo) {\thisllt{1}{1}{3}{2}};
\node at (6*\colspa,\rowtwo) {\thisllt{1}{1}{3}{3}};
\node at (7*\colspa,\rowtwo) {\thisllt{1}{1}{4}{2}};
\node at (8*\colspa,\rowtwo) {\thisllt{1}{1}{4}{3}};
\node at (9*\colspa,\rowtwo) {\thisllt{1}{1}{4}{4}};
\draw[decorate,decoration={brace,mirror,amplitude=4pt}] (4.55*\colspa,\rowtwo-2.5) -- (9.45*\colspa,\rowtwo-2.5);
\node at (7*\colspa,\rowtwo-4) {\(t^2 \spa \Acal^{[k]}_{20|2}\)};
%
%
\node at (0*\colspa,\rowthree) {\thisllt{1}{1}{2}{3}};
\node at (1*\colspa,\rowthree) {\thisllt{1}{1}{2}{4}};
\draw[decorate,decoration={brace,mirror,amplitude=4pt}] (-.55*\colspa,\rowthree-2.5) -- (1.45*\colspa,\rowthree-2.5);
\node at (0.5*\colspa,\rowthree-4) {\(t^3\spa \Acal^{[k]}_{21|1}\)};
\node at (2*\colspa,\rowthree) {\thisllt{1}{2}{2}{3}};
\node at (3*\colspa,\rowthree) {\thisllt{1}{2}{2}{4}};
\draw[decorate,decoration={brace,mirror,amplitude=4pt}] (1.55*\colspa,\rowthree-2.5) -- (3.45*\colspa,\rowthree-2.5);
\node at (2.5*\colspa,\rowthree-4) {\(t^3\spa \Acal^{[k]}_{12|1}\)};
\node at (4*\colspa,\rowthree) {\thisllt{1}{1}{3}{4}};
\node at (4*\colspa,\rowthree-4) {\(t^3 \spa \Acal^{[k]}_{20|11}\)};
\node at (5*\colspa,\rowthree) {\thisllt{1}{2}{3}{3}};
\node at (6*\colspa,\rowthree) {\thisllt{1}{2}{4}{3}};
\node at (7*\colspa,\rowthree) {\thisllt{1}{2}{4}{4}};
\draw[decorate,decoration={brace,mirror,amplitude=4pt}] (4.55*\colspa,\rowthree-2.5) -- (7.45*\colspa,\rowthree-2.5);
\node at (6*\colspa,\rowthree-4) {\(t^3 \spa \Acal^{[k]}_{11|2}\)};
\node at (8*\colspa,\rowthree) {\thisllt{1}{2}{3}{4}};
\node at (8*\colspa,\rowthree-4) {\(t^4 \spa \Acal^{[k]}_{11|11}\)};
\end{tikzpicture}
\vspace{1mm}

To also illustrate
\eqref{ec mod nsshuffle2}, here are the
flagged word parking functions
$w \in \FWPF_\alpha(\theta)$ with \(\dinv(\theta,w) = 2\), which correspond to the words in the middle row above.

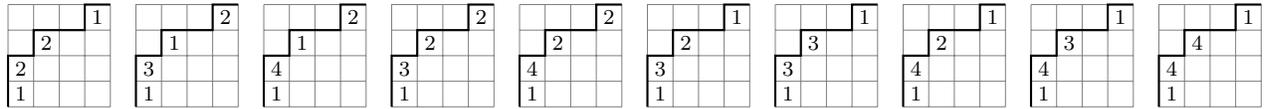
\newcommand{\thisfwpf}[4]{
\tikz[scale=.34,baseline=.6cm]{
\draw[help lines] (0,0) grid (4,4);
\draw[thick] (0,0) -- (0,2) -- (1,2) -- (1,3) -- (3,3) -- (3,4) -- (4,4);
\node at (.5,.5) {\tiny #1};
\node at (3.5, 3.5) {\tiny #2};
\node at (.5,1.5) {\tiny #3};
\node at (1.5,2.5) {\tiny #4};
}}

\noindent\begin{tikzpicture}[scale=.34,baseline=.6cm]
\useasboundingbox (-2,-3) rectangle (47,3);
\pgfmathsetmacro{\colspa}{5}

\node at (0*\colspa,0) {\thisfwpf{1}{1}{2}{2}};
\node at (1*\colspa,0) {\thisfwpf{1}{2}{3}{1}};
\node at (2*\colspa,0) {\thisfwpf{1}{2}{4}{1}};
\node at (3*\colspa,0) {\thisfwpf{1}{2}{3}{2}};
\node at (4*\colspa,0) {\thisfwpf{1}{2}{4}{2}};
\node at (5*\colspa,0) {\thisfwpf{1}{1}{3}{2}};
\node at (6*\colspa,0) {\thisfwpf{1}{1}{3}{3}};
\node at (7*\colspa,0) {\thisfwpf{1}{1}{4}{2}};
\node at (8*\colspa,0) {\thisfwpf{1}{1}{4}{3}};
\node at (9*\colspa,0) {\thisfwpf{1}{1}{4}{4}};
\end{tikzpicture}

We can then recover the corresponding instance of the (symmetric) compositional shuffle theorem
by applying the Weyl symmetrization operator  $\Weyl_1\Weyl_2$ to \eqref{e example AK expansion}.
Using \eqref{e sigmabold atom}, this gives
\begin{align}
\nabla C_{(3,1)}
= q^3 ( t \, s_{211} + t^2 s_{1111}) +  q^2 \big(t \, s_{31}  + t^2 s_{211} + t^2 s_{22}+ t^3 s_{211}+t^4 s_{1111}  \big).
\end{align}
The coefficient of  $q^3$ (resp.  $q^2$) here is the ordinary LLT polynomial  $\Gcal_{\etabold}(\xx;t)$ corresponding to
the tuple of single-column skew shapes  $\etabold = \big((1), (111)\big)$ (resp. $\etabold = \big((2)/(1), (1), (22)/(11)\big)$)
as in \eqref{e sigma LLT}.
\end{example}

\vfill
\pagebreak

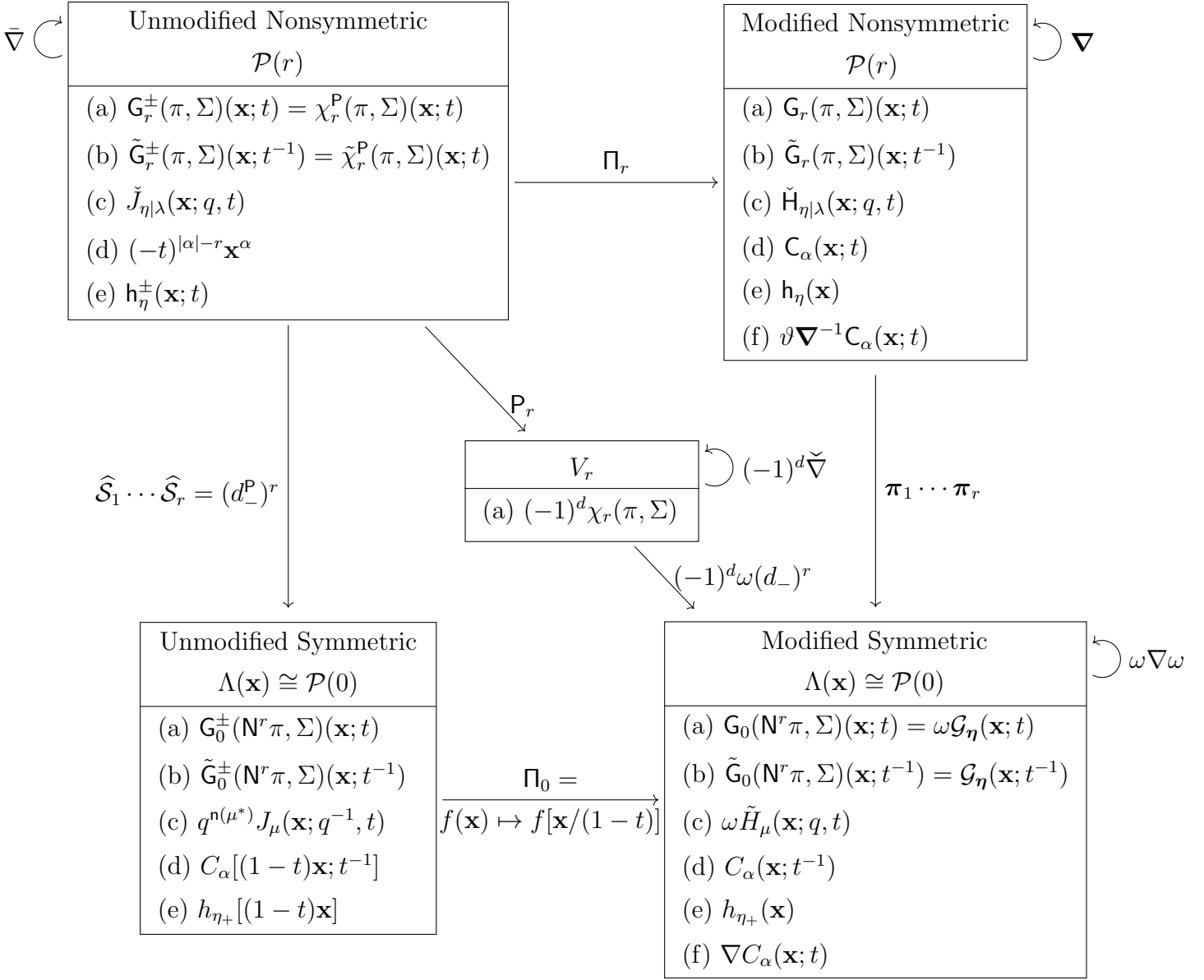
\begin{figure}
  \centerfloat
\begin{tikzpicture}[scale=1]
\def\drawLineBelowRow#1#2{
  \draw (#2.west|-#2-#1-1.south) -- (#2.east|-#2-#1-1.south);
}
\def\drawBoundingBox#1#2{
  \draw (#2.west|-#2-#1-1.south) -- (#2.east|-#2-#1-1.south) --
  (#2.east|-#2-1-1.north) -- (#2.west|-#2-1-1.north) -- cycle;
}

   \matrix[matrix of math nodes, nodes={anchor=west}] (nsunmod) at (0,10) {
     \node[align=center, minimum width=6.5cm] (nsunmod-1-1) {Unmodified Nonsymmetric}; \\
     \node[align=center, minimum width=6.5cm] (nsunmod-2-1) {$\P(r)$};\\
     \text{(a) }\Growpm_r(\pi, \Sigma)(\xx;t) = \chi_r^\PP(\pi,\Sigma)(\xx;t) \\
     \text{(b) } \Gcolpm_r(\pi, \Sigma)(\xx;t^{-1}) = \tilde{\chi}_r^\PP(\pi,\Sigma)(\xx;t)  \\
     \text{(c) } \stE_{\eta|\lambda}(\xx;q,t) \\
     \text{(d) } (-t)^{|\alpha|-r} \xx^\alpha\\
     \text{(e) } \fh^\pm_\eta(\xx;t)\\
     \phantom{\text{(f)}} \\
   };
   \drawBoundingBox{7}{nsunmod}
   \drawLineBelowRow{2}{nsunmod}

  \matrix [matrix of math nodes, nodes={anchor=west}] (symunmod) at (0,0) {
    \node[align=center, minimum width=4.5cm] (symunmod-1-1) {Unmodified Symmetric}; \\
    \node[align=center, minimum width=4.5cm] (symunmod-2-1) {\(\Lambda(\xx) \cong
      \Pcal(0)\)}; \\
    \text{(a) } \Growpm_0(\mathsf{N}^r \pi, \Sigma)(\xx;t) \\
    \text{(b) } \Gcolpm_{0} (\mathsf{N}^r \pi, \Sigma)(\xx;t^{-1})\\
    \text{(c) } q^{\mathsf{n}(\mu^*)} J_\mu(\xx;q^{-1},t)\\
    \text{(d) } C_\alpha[(1-t)\xx;t^{-1}]\\
    \text{(e) } h_{\eta_+}[(1-t)\xx]\\
    \phantom{\text{(f)}} \\
   };
   \drawBoundingBox{7}{symunmod}
  \drawLineBelowRow{2}{symunmod}

  \matrix [matrix of math nodes, nodes={anchor=west}] (nsmod) at (9.5,10) {
    \node[align=center, minimum width=4.5cm] (nsmod-1-1) {Modified Nonsymmetric};\\
    \node[align=center, minimum width=4.5cm] (nsmod-2-1) {$\P(r)$};\\
    \text{(a) }\Grow_r(\pi,\Sigma)(\xx;t) \\
    \text{(b) }\Gcol_r(\pi,\Sigma)(\xx;t^{-1}) \\
    \text{(c) }\tE_{\eta|\lambda}(\xx;q,t) \\
    \text{(d) }\nsC_\alpha(\xx;t) \\
    \text{(e) }\fh_\eta(\xx) \\
    \text{(f) }\vartheta \modnab^{-1} \nsC_\alpha(\xx;t) \\
  };
   \drawBoundingBox{8}{nsmod}
  \drawLineBelowRow{2}{nsmod}

  \matrix [matrix of math nodes, nodes={anchor=west}] (symmod) at (9.5,0) {
    \node[align=center, minimum width=6.5cm] (symmod-1-1)
    {Modified Symmetric}; \\
    \node[align=center, minimum width=6.5cm] (symmod-2-1)
    {\(\Lambda(\xx) \cong \Pcal(0)\)}; \\
    \text{(a) }\Grow_0(\mathsf{N}^r \pi, \Sigma)(\xx;t)
    = \omega
    \Gcal_{\etabold}(\xx;t) \\
    \text{(b) }\Gcol_0(\mathsf{N}^r \pi, \Sigma)(\xx;t^{-1}) =
    \Gcal_{\etabold}(\xx;t^{-1}) \\
    \text{(c) }\omega \Htild_\mu(\xx;q,t) \\
    \text{(d) }C_\alpha(\xx;t^{-1})\\
    \text{(e) }h_{\eta_+}(\xx) \\
    \text{(f) }\nabla C_\alpha(\xx;t)\\
  };
  \drawBoundingBox{8}{symmod}
  \drawLineBelowRow{2}{symmod}

  \draw[->, shorten >=0.1cm, shorten <=-.75cm] (nsunmod) -- (symunmod) node[midway, left] {\(\hsym_1 \cdots
    \hsym_r = (d_-^\PP)^r\)};
  \draw[->, shorten >=0.1cm, shorten <=0.1cm] (nsunmod) -- (nsmod) node[midway, above] {\(\Pisf_r\)};
  \draw[->, shorten >=0.1cm, shorten <=0.1cm] (nsmod) -- (symmod) node[midway, right] {\(\pibold_1
    \cdots \pibold_r\)};
  \draw[->, shorten >=0.1cm, shorten <=0.1cm] (symunmod) -- (symmod)
  node[midway, above] {\(\Pisf_0 = \)}
  node[midway, below] {\(f(\xx) \mapsto f[\xx/(1-t)]\)};

  \matrix [matrix of math nodes, draw, nodes={anchor=west}] (Vr) at (4.75,5) {
    \node[align=center, minimum width=3.5cm] (Vr-1-1) {\(V_r\)}; \\
    \text{(a) }(-1)^d \chi_r(\pi,\Sigma) \\
  };
  \drawLineBelowRow{1}{Vr}

  \draw[->, shorten >=0.2cm, shorten <=-1cm] (nsunmod) -- (Vr) node[midway, right] {\(\mathsf{P}_r\)};
  \draw[->, shorten >=0.1cm, shorten <=0.1cm] (Vr) -- (symmod) node[midway, right] {\((-1)^d \omega (d_-)^r\)};

  \draw[->] (nsunmod.north west)+(-0.1,-1) arc (300:60:3mm)
  node[midway, left] {\(\umnab\)};
  \draw[->] (nsmod.north east)+(0.1,-1) arc (-120:120:3mm)
  node[midway, right] {\(\modnab\)};
  \draw[->] (Vr.north east)+(0.1,-.7) arc (-120:120:3mm)
  node[midway, right] {\((-1)^d \Mnab\)};
  \draw[->] (symmod.north east)+(0.1,-1) arc (-120:120:3mm)
  node[midway, right] {\(\omega \nabla \omega\)};
\end{tikzpicture}
\caption{\label{fig:summary diagram}This diagram summarizes most of the notation in the paper.
Restricting each table to a fixed row gives a commutative
  diagram between families of almost symmetric polynomials, e.g., restricting to (c)
  gives~\eqref{e mac comm diagram}.
  Automorphisms given by curved arrows are related via intertwining relations implied by the diagram, e.g.,
  \(\modnab\) and \(\omega \nabla \omega\) are automorphisms of
  \(\P(r)\) and \(\P(0)\), respectively, and
  $\Weyl_1 \cdots \Weyl_r \modnab = \omega \nabla \omega \Weyl_1
  \cdots \Weyl_r$ (see~\eqref{e modnab and sym}). Another example is
  given by \(\PP_r \umnab = (-1)^d \Mnab \PP_r\) (see~\eqref{et two nablas}).}
\end{figure}

\clearpage
\bibliographystyle{plain}
\bibliography{references}
\addresseshere

\end{document}